\documentclass[reqno,11pt]{amsart}
\usepackage{amsmath}
\usepackage{amssymb}
\usepackage{tabularx}
\usepackage{enumerate}
\usepackage[dvips]{psfrag,graphicx}
\usepackage{color}
\usepackage{subfigure}
\usepackage{cases}
\usepackage{bbm}

\usepackage{mathrsfs}
\usepackage{empheq}
\usepackage{amsfonts,amssymb,amsmath}
\usepackage{epsfig}

\usepackage{comment}

\makeatletter
\@addtoreset{equation}{section}
  
\makeatother
\newtheorem{thm}{Theorem}[section]
\newtheorem{lem}[thm]{Lemma}   
\newtheorem{cor}[thm]{Corollary}
\newtheorem{prop}[thm]{Proposition}
\newtheorem{rem}[thm]{Remark}

\begin{document}
\title[Coexistence-segregation dichotomy]
{Coexistence-segregation dichotomy in the full cross-diffusion limit 
of the stationary SKT model}
 \thanks{This research was
partially supported by JSPS KAKENHI Grant Number 22K03379.}
\author[J. Inoue]{Jumpei Inoue$^\dag$}
\author[K. Kuto]{Kousuke Kuto$^\ddag$}
\author[H. Sato]{Homare Sato$^\dag$}
\thanks{$\dag$ Department of Pure and Applied Mathematics, 
Graduate School of Fundamental Science and Engineering,
Waseda University, 
3-4-1 Ohkubo, Shinjuku-ku, Tokyo 169-8555, Japan.}
\thanks{$\ddag$ Department of Applied Mathematics, 
Waseda University, 
3-4-1 Ohkubo, Shinjuku-ku, Tokyo 169-8555, Japan.}
\thanks{{\bf E-mail:} \texttt{kuto@waseda.jp}}
\date{\today}

\begin{abstract} 
This paper studies the Lotka-Volterra competition model 
with cross-diffusion terms under homogeneous 
Dirichlet boundary conditions. 
We consider the asymptotic behavior of positive steady-states 
as equal two cross-diffusion coefficients tend to infinity
(so-called the full cross-diffusion limit).
The first result shows that, at the full cross-diffusion limit,
the set of positive steady-state solutions can be classified
into two types: the {\it small coexistence} or 
the {\it complete segregation}.
The small coexistence can be characterized by the set
of positive solutions of a nonlinear elliptic equation,
while the complete segregation can be characterized by the set of
nodal solutions of another nonlinear elliptic equation.
The second result concerns the global bifurcation structure 
of the 1d model with large cross-diffusion terms to show that 
the branch of small coexistence bifurcates from the trivial solution 
and many branches of complete segregation bifurcate from solutions 
on the branch of small coexistence. 
Finally, in order to verify the theoretical results, 
we chase some bifurcation branches numerically via 
the continuation software \texttt{pde2path}.
\end{abstract}

\keywords{cross-diffusion,
competition model,
limiting systems, 
nonlinear elliptic equations,
the sub-super solution method,
bifurcation,
perturbation,
\texttt{pde2path}} \maketitle

\section{Introduction}
In this paper, we are concerned with   
the following Lotka-Volterra
competition model with linear and nonlinear diffusion terms:
\begin{equation}\label{para}
\begin{cases}
u_{t}=\Delta [\,(d_{1}+\alpha v)u\,]+
u(a_{1}-b_{1}u-c_{1}v)
\quad&\mbox{in}\ \Omega\times (0,T),\\
v_{t}\,=\Delta [\,(d_{2}+\beta u)v\,]+
v(a_{2}-b_{2}u-c_{2}v)
\quad&\mbox{in}\ \Omega\times (0,T).
\end{cases}
\end{equation}
In terms of a model of the mathematical biology, 
unknown functions $u(x,t)$ and $v(x,t)$ represent 
the respective population densities at location $x\in\Omega$
and time $t\in (0,T)$ of two species competing
in the habitat $\Omega$.
Throughout this paper,
the habitat
$\Omega$ is assumed to be a bounded domain in $\mathbb{R}^{N}$ 
with a smooth boundary $\partial\Omega $ if $N\ge 2$;
an interval if $N=1$.
Typical boundary conditions imposed in conjunction with \eqref{para} 
are homogeneous Neumann boundary conditions $\partial_{\nu}u=\partial_{\nu}v=0$ on $\partial\Omega\times (0,T)$,
or 
homogeneous Dirichlet boundary ones
$u=v=0$ on $\partial\Omega\times (0,T)$.
The former reflects a situation where the species 
cannot enter or exit across the boundary $\partial\Omega$, 
while the latter reflects that the habitat $\Omega$ is surrounded by 
an environment of threats to both competing species.
In \eqref{para},
$\Delta :=\sum^{N}_{j=1}\partial^{2}/\partial x_{j}^{2}$
denotes the usual Laplacian;
coefficients
$a_{i}$,
$b_{i}$ and
$c_{i}$ 
$(i=1,2)$
are positive constants;
$a_{1}$ and $a_{2}$ denote the growth rate of respective species;
$b_{1}$ and $c_{2}$ denote the intra-specific competition coefficients;
$c_{1}$ and $b_{2}$ denote the inter-specific competition coefficients.
In diffusion terms pertaining to the Laplacian,
$\Delta u$ and $\Delta v$ 
are linear diffusion terms describing 
a spatially random movement of each species,
whereas
$\Delta (uv)$ is a nonlinear diffusion term,
which is usually called {\it cross-diffusion},
describing the interaction of diffusion caused by 
the population pressure resulting from interference
between different species.
Roughly speaking, each cross-diffusion describes
a tendency of each species to diffuse faster 
where there are more individuals of their competitors. 
For instance, if two competing species share common resources (feed), 
the introduction of cross-diffusion terms is natural 
because each species can more 
rationally consume the resources (feed) when 
they are away from their competitors.
We refer to the book by
Okubo and Levin \cite{OL} for detailed mechanism 
of the biological diffusion.
Such a Lotka-Volterra competition system with cross-diffusion 
(and some additional terms) 
was proposed by Shigesada, Kawasaki and Teramoto
\cite{SKT} in order to realize 
{\it segregation phenomena} of two competing species observed in 
ecosystems.
Nowadays,
beyond their bio-mathematical aim,
many mathematicians have continued to study 
a class of Lotka-Volterra systems with cross-diffusion 
as a prototype of diffusive interactions.
Such a class of Lotka-Volterra systems with cross-diffusion 
is referred to as {\it the SKT model} celebrating the authors of
the pioneering paper 
\cite{SKT}.
We refer to book chapters by J\"ungel \cite{Ju}, Ni \cite{Ni},
and Yamada \cite{Yam1, Yam2}
as surveys for mathematical works relating to the SKT model.

This paper focuses on the stationary problem of \eqref{para}
which is the following system of nonlinear elliptic PDEs:
\begin{equation}\label{ell}
\begin{cases}
\Delta [\,(d_{1}+\alpha v)u\,]+
u(a_{1}-b_{1}u-c_{1}v)=0
\quad&\mbox{in}\ \Omega,\\
\Delta [\,(d_{2}+\beta u)v\,]+
v(a_{2}-b_{2}u-c_{2}v)=0
\quad&\mbox{in}\ \Omega.
\end{cases}
\end{equation}
The stationary problem of the SKT model
such as \eqref{ell} has been studied extensively
(see e.g., the above book chapters and references therein).
Among other things, Lou and Ni \cite{LN2} 
developed a mathematical procedure to study the limiting behavior 
of positive nonconstant solutions when one of the 
two cross-diffusion coefficients tends to infinity
(so-called the {\it unilateral cross-diffusion limit}).
According to a main result in \cite{LN2},
concerning \eqref{ell}
under homogeneous Neumann boundary conditions
$\partial_{\nu}u=\partial_{\nu}v=0$ on $\partial\Omega$
with $N\le 3$,
any sequence of positive nonconstant solutions,
which is denoted by $\{(u_{n}, v_{n})\}$,
of the system with $\alpha=\alpha_{n}\to\infty$ 
(whereas $\beta\ge 0$ is fixed) generically
satisfies one of the 
following two convergence scenarios
passing to a subsequence if necessary:
\begin{enumerate}[(I)]
\item
there exist positive functions $u$ and $V$ such that
$(u_{n}, \alpha_{n}v_{n})\to (u,V)$ in 
$C^{1}(\overline{\Omega})^{2}$;
\item
there exists a positive constant $\tau$ such that
$u_{n}v_{n}\to\tau$ uniformly in $\overline{\Omega}$.
\end{enumerate}
The first scenario (I) implies a situation where 
the species $v$ 
(with fixed cross-diffusion coefficient $\beta$)
decays with the order $O(1/\alpha_{n})$ as
the cross-diffusion capacity $\alpha_{n}$
of the competitor species becomes large.
On the other hand, the second scenario (II) can be expected to realize 
segregation of two competing species in some sense because $\tau$ is spatially 
homogeneous (so-called, the {\it incomplete segregation}).
Since $\tau$ remains positive, then the 
{\it complete segregation} 
(such as,
in each territory of one species, 
the other species are almost absent) cannot expected when one of two cross-diffusion coefficients is very large.
That each scenario actually occurs was confirmed by perturbation
from the solution of the corresponding limiting system:
For the limiting system characterizing (I), 
the second author of this paper studied the 
bifurcation structure of solutions in \cite{Ku1} and
the stability/instability of these solutions was studied by
Li and Wu \cite{LW1}.
For
the limiting system characterizing (II),
Lou, Ni and Yotsutani \cite{LNY1}
exhibited the detailed information on the set 
of solutions in the one-dimensional case.
Since then, for the limiting system characterizing (II),
the existence of solutions in the higher-dimensional case
and the stability analysis were studied in
\cite{KW, LNY2, MSY, NWX, WWX, Wu, WX}.
In addition, we refer to \cite{Ka1, Ka2, LX, LW2, ZW}
as recent papers that discuss the limiting systems and related issues.

Concerning the unilateral cross-diffusion limit 
for the Dirichlet problem, 
the second author and Yamada \cite{KY1, KY2}
discussed the limiting behavior of positive solutions of
\eqref{ell} with $\beta=0$ under homogeneous Dirichlet boundary conditions $u=v=0$ on $\partial\Omega$.
It was shown in \cite{KY1} that
any sequence of positive solutions
of the system with $\alpha=\alpha_{n}\to\infty$ 
generically satisfies only the above first scenario (I) 
passing to a subsequence.
Furthermore, the global bifurcation structure of 
positive solutions to the associated limiting system
was also obtained in \cite{KY1, KY2}.

Recently, the second author \cite{Ku2} developed 
a procedure to study the limiting behavior 
of positive nonconstant solutions when both cross-diffusion coefficients 
tend to infinity at the same rate
(so-called the {\it full cross-diffusion limit}).
It was shown in \cite{Ku2} that,
concerning \eqref{ell} under homogeneous Neumann boundary conditions,
any sequence of positive nonconstant solutions
of the system with $\alpha=\alpha_{n}\to\infty$,
$\beta=\beta_{n}\to\infty$ and
$\gamma_{n}:=\alpha_{n}/\beta_{n}\to\gamma$
with some $\gamma>0$ generically satisfies only the above (II)
passing to a subsequence.
Furthermore, the second author \cite{Ku3} showed the existence of
a limiting system characterizing the incomplete segregation
in the full cross-diffusion limit for the Neumann problem
and revealed the global bifurcation structure of the one-dimensional
limiting system.

After reviewing the known results on the unilateral or 
the full 
cross-diffusion limit on \eqref{ell}, 
we are naturally tempted to consider 
the full cross-diffusion limit for the Dirichlet problem.
As the first step of the issue,
this paper deals with the following stationary SKT model:
\begin{subequations}\label{SKT}
\begin{empheq}[left={\empheqlbrace}]{alignat=2}
\label{SKT1}
\Delta [\,(1+\alpha v)u\,]+
u(\lambda m(x)-b_{1}u-c_{1}v)=0
\ \ &\mbox{in}\ \Omega,\\
\label{SKT2}
\Delta [\,(1+\alpha u)v\,]+
v(\lambda m(x)-b_{2}u-c_{2}v)=0
\ \ &\mbox{in}\ \Omega
\end{empheq}
under homogeneous Dirichlet boundary conditions
\begin{equation}\label{SKT3}
u=v=0\ \ \mbox{on}\ \partial\Omega.
\end{equation}
\end{subequations}
In view of the modelling of \eqref{SKT}, 
each competing species is assumed to retain 
the same level of diffusion capacity, and then, 
the linear diffusion coefficients and
the cross-diffusion coefficients
are respectively taken in common for both species.
By virtue of scalings for unknown functions,
we may set the linear diffusion coefficients
to $1$ without loss of generality.
For an orthodox setting \eqref{SKT} in the stationary SKT model,
we adopt the simplification described above in the diffusion part
under the Laplacian, while for the reaction terms, 
coefficients
$a_{1}$ and $a_{2}$ are set to be 
$\lambda m(x)$ with spatial dependence.
In terms of the competition model,
the function $m(x)$ can be interpreted as 
representing the spatial distribution of common resources 
(feed) for both competing species.
It will be assumed that $m(x)\ge 0$ is nontrivial and in
$L^{\infty}(\Omega )$;
$\lambda$ is a positive parameter,
which indicates the amount of resources (or feed). 
By virtue of the elliptic regularity theory (e.g., \cite{GT}),
any weak solution $(u,v)$ of \eqref{SKT} belongs to 
$[\,W^{2,p}(\Omega)\cap W^{1,p}_{0}(\Omega )\,]^{2}$
for any $p\ge 1$.
Throughout this paper, we call a pair of functions 
$(u,v)$ 
satisfying $u>0$ and $v>0$ in $\Omega$,
in addition to \eqref{SKT}, 
a positive solution of \eqref{SKT}.
It is noted that,
if $u\ge 0$ and $v\ge 0$ are not identically zero in $\Omega$
and satisfy \eqref{SKT}, then $(u,v)$ is a positive solution
by the strong maximum principle.

\begin{table}[]
\begin{tabular}{lllll}
\cline{1-3}
\multicolumn{1}{|l|}{}                                 & \multicolumn{1}{l|}{Neumann BC}   & \multicolumn{1}{c|}{Dirichlet BC}   &  &  \\ \cline{1-3}
\multicolumn{1}{|c|}{Unilateral cross-diffusion limit} & \multicolumn{1}{c|}{(I), (II)} & \multicolumn{1}{c|}{(I)}         &  &  \\ \cline{1-3}
\multicolumn{1}{|c|}{Full cross-diffusion limit}       & \multicolumn{1}{c|}{(II)}      & \multicolumn{1}{c|}{(III), (IV)} &  &  \\ \cline{1-3}
                                                       &                                &                                  &  & 
\end{tabular}
\caption{Classification of the cross-diffusion limit}
\end{table}

The purpose of this paper is to prove the existence of
limiting 
systems classifying behaviors  
of positive solutions of \eqref{SKT} as
$\alpha\to\infty$ and to construct the
bifurcation curve of solutions to \eqref{SKT} with large $\alpha>0$
by the perturbation of solutions to limiting systems. 
It should be noted here that 
in a case where $\alpha$ is sufficiently large, 
the necessary and sufficient range of 
$\lambda$ for the existence of positive solutions of 
\eqref{SKT} is known.
In order to express the range, we introduce 
the eigenvalue problem
with nonnegative weight function $m(x)(\not\equiv 0)$:
\begin{equation}\label{ev}
-\Delta\phi=\lambda m(x)\phi\ \ \mbox{in}\ \Omega,\quad
\phi=0\ \mbox{on}\ \partial\Omega.
\end{equation}
It is known that all eigenvalues consist of a 
positive sequence (e.g., \cite{CC, Ni}):
\begin{equation}\label{ev2}
(0<)\lambda_{1}(m)<\lambda_{2}(m)<
\cdots <\lambda_{j}(m)<\cdots
\quad\mbox{with}\quad\lim_{j\to\infty}\lambda_{j}(m)=\infty.
\end{equation}
In the following, we omit the dependence of each eigenvalue on $m$ and 
denote $\lambda_{j}=\lambda_{j}(m)$.
Concerning the Dirichlet problem \eqref{SKT} with large $\alpha>0$,
combining the results by Ruan \cite[Theorem 4.1]{Ru} and 
the second author with Yamada \cite[Theorem 3.4 and Lemma 3.4]{KY2},
we know that $\lambda_{1}$ gives a threshold for the nonexistence/existence
of positive solutions:.
\begin{lem}\label{exlem}
Suppose that $\alpha>0$ is sufficiently large.
If $0<\lambda\le \lambda_{1}$, then \eqref{SKT} admits 
no positive solution, whereas if $\lambda>\lambda_{1}$, then 
\eqref{SKT} admits at least one 
positive solution.
\end{lem}
Regardless of Lemma \ref{exlem},
the detailed structure of the set of positive solutions of \eqref{SKT}
still remains unknown, since the above nonexistence/existence 
of positive solutions was obtained by topological arguments 
using the degree or the global bifurcation theory.
Therefore, in this paper, we study in detail 
the set of positive solutions of \eqref{SKT} 
with large $\alpha$.

The first main result of this paper will show that the limiting behavior of
any sequence $\{(u_{n}, v_{n})\}$ of positive solutions
of \eqref{SKT} with $\alpha=\alpha_{n}\to\infty$ generically 
satisfies either of the following two scenarios, passing to 
a subsequence if necessary:
\begin{enumerate}[(I)]
\item[(III)]
({\it small coexistence})
there exists a positive function $U$ such that
$(\alpha_{n}u_{n}, \alpha_{n}v_{n})\to (U,U)$ in 
$C^{1}(\overline{\Omega})^{2}$;
\item[(IV)]
({\it complete segregation})
there exists a sign-changing function $w\in C^{1}(\overline{\Omega})$ 
such that
$(u_{n}, v_{n})\to (w_{+}, w_{-})$ uniformly in 
$\overline{\Omega}$, where
$w_{+}:=\max\{\,w,0\,\}$ and
$w_{-}:=\max\{\,0,-w\,\}$.
\end{enumerate}
It will be shown that 
in the former scenario (III), 
the limit function $U$ can be characterized by 
a positive solution of the Dirichlet problem 
for a nonlinear elliptic equation
(the first limiting system), 
while in the latter scenario (IV), 
the limit function $w$ can be characterized by 
a sign-changing solution of the Dirichlet problem  
for another nonlinear elliptic equation
(the second limiting system),
see Theorem \ref{convthm} below.
The next result will show that the set of all positive
solutions of the first limiting system forms a simple curve 
$\mathcal{C}_{\infty}$
parameterized by $\lambda\in (\lambda_{1},\infty)$,
which bifurcates from the trivial solution at $\lambda=\lambda_{1}$
(Theorem \ref{thm13}).
Furthermore, for \eqref{SKT} with large $\alpha$,
we construct a bifurcation branch $\mathcal{C}_{\alpha, \varLambda}$
of the subset of positive solutions
(Theorem \ref{Cathm}).
It will be shown that
$\mathcal{C}_{\alpha, \varLambda}$
bifurcates from the trivial solution at $\lambda=\lambda_{1}$
and extends in the direction $\lambda>\lambda_{1}$,
and that
the $(u,v)$ component remains small with the order $O(1/\alpha )$.
More precisely, both $u$ and $v$ are close to $U/\alpha$ for 
the positive solution $U$ to the first limiting system.
This pure mathematical result will be realized in the numerical bifurcation
diagram produced in the final section,
see the blue curve in Figure 1.
From the viewpoint of the biological model, 
the existence of such positive solutions seems somewhat interesting.
In the case where the cross-diffusion coefficient $\alpha$
is large, a repulsive interaction between 
competing species becomes stronger 
(see e.g., \cite[Section 5.4]{OL}). 
This result suggests a possibility that, 
despite such a mutually repulsive situation,  
competitive species can coexist in a small number 
with a similar spatial configuration (small coexistence),
see Figure 2.
On the other hand, the scenario (IV) may realize
the complete segregation in the sense that
the territories of competing species
become separated each other as $\alpha\to\infty$.
It is noted that the compete segregation (IV) 
where $u_{n}v_{n}\to 0$ 
is totally different from
the incomplete segregation (II) where $u_{n}v_{n}\to
\tau>0$, and (IV) has not been observed in 
any previous research of the cross-diffusion limit. 
In the one-dimensional case,
if $m(x)$ is a positive constant,
then the complete structure of
all solutions of the second limiting system is well-known as
follows:
for each $j\in\mathbb{N}$,
the set $\mathcal{S}^{j}_{\infty}$ of solutions with exact $j-1$ 
zeros in the open interval $\Omega$
forms a pitchfork bifurcation curve 
which bifurcates from the trivial solution at $\lambda=\lambda_{j}$,
and the upper and lower branches respectively are parameterized by
$\lambda\in (\lambda_{j}, \infty)$ 
(Lemma \ref{basiclem}).
This paper will prove the existence of positive solutions
$(u,v)$ near $(w_{+},w_{-})$ for each $w\in\mathcal{S}^{j}_{\infty}$
with $j\ge 2$
if $\alpha$ is sufficiently large.
In the case where $\alpha$ is sufficiently large, 
the final main result ensures that,
for each integer $j\ge 2$, the set of positive solutions of \eqref{SKT}
perturbed by $\mathcal{S}^{j}_{\infty}$ 
forms a pitchfork bifurcation curve 
which bifurcates from a positive
solution on $\mathcal{C}_{\alpha, \varLambda}$ at some
$\lambda=\mu_{j,\alpha}$ near $\lambda_{j}$
(Theorem \ref{g2ndbifthm}),
see the red $(j=2)$ and purple $(j=3)$ curves in Figure 1.

Throughout the mathematical argument in this paper, 
the following two transformations are important. 
The first is the transformation $(U,V)=\alpha (u,v)$, 
which enables us to observe the solution of small coexistence  
on an appropriately enlarged scale.
The second is $(W,Z)=(U-V, (1+V)U)$, 
which reduces the quasilinear system \eqref{SKT} to 
a semilinear system \eqref{WZeq} below, where $W$ represents the phase difference 
between $U$ and $V$. 
In fact, the branch $\mathcal{C}_{\alpha, \varLambda}$
of small coexistence will be obtained by 
the scaling and the perturbation of
the branch of solutions with $W=0$
to the limiting system \eqref{WZeq0}
of \eqref{WZeq}
as $\alpha\to\infty$.
A difficulty of the proof of Theorem \ref{g2ndbifthm} 
is to find a
pitchfork bifurcation point on the branch 
$\mathcal{C}_{\alpha, \varLambda}$ of positive,
namely nontrivial, solutions.
Our idea is first to show that the limiting system \eqref{WZeq0}
has the set of solutions that is topologically isomorphic to 
$E_{j}$ only when $\lambda=\lambda_{j}$, where
$E_{j}$ denotes the eigenspace of \eqref{ev} with $\lambda=\lambda_{j}$.
This fact shows that, restricted to the limit system \eqref{WZeq0}, 
the set of solutions with $W\neq 0$ bifurcates from the set of solutions with 
$W=0$ at $\lambda=\lambda_j$ and extends unboundedly
in the direction perpendicular to the $\lambda$ axis. 
We then perturb the perpendicular branch of the limiting system 
to the case where $\alpha$ is very large, 
thereby show the existence of solutions of \eqref{SKT} not in 
$\mathcal{C}_{\alpha, \varLambda}$. 
In particular, when $\dim E_{j}=1$, 
we show that the bifurcation point of the limiting system 
is mapped to the bifurcation point of a perturbed system 
with large $\alpha$ by applying the implicit function theorem 
to the linearized system with an integral constraint.
From the assertion and proof of the result,
we read that the complete segregation emerges 
from near each eigenvalue $\lambda_{j}$ by 
the bifurcation destroying the (almost) congruence 
of $u$ and $v$ on the branch $\mathcal{C}_{\alpha, \varLambda}$ 
of small coexistence.

The contents of this paper are as follows: 
In Section 2, we state main results of this paper. 
In Section 3, we prove the convergence result of positive solutions 
of \eqref{SKT} as $\alpha\to\infty$. 
In Section 4, we discuss the structure of solutions with $W = 0$ 
to the limiting system characterizing 
the scenario (III)
and the perturbation of such solutions.
In Section 5, we discuss the structure of solutions with $W\neq 0$ 
to the limiting system and
their perturbations.
In Section 6, symmetry breaking bifurcation points on the
branch $\mathcal{C}_{\alpha, \varLambda}$ of the small coexistence
will be shown.
In Section 7, we prove the global structure of positive solutions of
\eqref{SKT} with large $\alpha$ 
in the special case where $N=1$ and $m(x)$ is a positive constant.
In Section 8,
we exhibit some numerical simulations on the bifurcation diagram
of positive solutions by using 
the continuation software \texttt{pde2path}.
The numerical bifurcation diagrams in Section 8
will support the theoretical results of this paper.

Throughout this paper,
the usual norms of the spaces $L^{p}(\Omega )$
for $p\in [\,1,\infty )$ and $L^{\infty }(\Omega )$
are denoted by
$$
\|u\|_{p}:=
\left(\displaystyle\int_{\Omega }|u(x)|^{p}dx\right)^{1/p},
\qquad
\|u \|_{\infty }:=\mbox{ess.}\sup_{x\in\overline{\Omega }}|u(x)|.
$$
Hence $\|u\|_{\infty}=\max_{x\in\overline{\Omega}}|u(x)|$
in a case where $u\in C(\overline{\Omega })$.
Furthermore,
the functional spaces used extensively will be denoted as follows: 
$$
X:=W^{2,p}(\Omega)\cap W^{1,p}_{0}(\Omega),\quad
Y:=L^{p}(\Omega),\quad
\boldsymbol{X}:=X\times X,\quad
\boldsymbol{Y}:=Y\times Y.
$$

\section{Main results}
In this section, we list main results of this paper.
The first result shows two convergence
scenarios in the full cross-diffusion limit:
\begin{thm}\label{convthm}
Suppose that $\lambda>\lambda_{1}$ 
and $\lambda\neq\lambda_{j}$ for any $j\ge 2$.
Let $\{(u_{n}, v_{n})\}$ be any sequence of positive solutions
of \eqref{SKT} with $\alpha=\alpha_{n}\to\infty$.
Then either of the following two convergence situations occurs
passing to a subsequence if necessary:
\begin{enumerate}[(i)]
\item
There exists a positive function $U\in X$ for any $p\ge 1$ 
such that
$$
\lim_{n\to\infty}(\alpha_{n}u_{n}, \alpha_{n}v_{n})=(U,U)
\quad\mbox{in}\ C^{1}(\overline{\Omega})
\times C^{1}(\overline{\Omega}).$$
Furthermore, $U$ satisfies
\begin{equation}\label{LS1}
\begin{cases}
\Delta [\,(1+U)U\,]+\lambda m(x)U=0\ \ &\mbox{in}\ \Omega,\\
U=0\ \ &\mbox{on}\ \partial\Omega.
\end{cases}
\end{equation}
\item
There exists a sign-changing function $w\in X$ for any $p\ge 1$ such that
$$
\lim_{n\to\infty}(u_{n}, v_{n})=(w_{+}, w_{-})
\ \ \mbox{uniformly in}\ \overline{\Omega},
$$
where $w_{+}=\max\{\,-w,0\,\}$ and $w_{-}=\max\{\,0,-w\,\}$.
Furthermore, $w$ is a solution of 
\begin{equation}\label{LS2}
\begin{cases}
\Delta w+\lambda m(x)w-b_{1}(w_{+})^{2}+c_{2}(w_{-})^{2}=0
\ \ &\mbox{in}\ \Omega,\\
w=0\ \ &\mbox{on}\ \partial\Omega.
\end{cases}
\end{equation}
\end{enumerate}
\end{thm}
In the first convergence situation (i),
both competing species $u_{n}$ and $v_{n}$ with large cross-diffusion
$\alpha_{n}$ stay small with the order $O(1/\alpha_{n})$ and are close to
a common profile $U/\alpha_{n}$ in $\Omega$.
Then we call this situation occurring in the full 
cross-diffusion limit {\it small coexistence}.
On the other hand, the second convergence situation (ii),
territories of two competing species become segregate 
completely each other
as $\alpha_{n}\to\infty$
because $u_{n}v_{n}\to 0$. 
Therefore, we call the situation (ii) in the full cross-diffusion limit
{\it complete segregation}.
Furthermore, it can be said that
the boundary of territories of two competing species
can be characterized by the nodal set of solutions of \eqref{LS2}.
As mentioned in the previous section, 
the full cross-diffusion limit for the Neumann problem
is characterized by the incomplete segregation 
such that,
in the high density region of one species, the other remains comparatively few
since $u_{n}v_{n}\to\tau >0$.

The second result gives the global bifurcation structure of
positive solutions of the first limiting system \eqref{LS1} 
of small coexistence type with bifurcation parameter $\lambda$:
\begin{thm}\label{thm13}
The set of all positive solutions of \eqref{LS1}
consists of the following curve parameterized by
$\lambda\in (\lambda_{1}, \infty)$:
$$
\mathcal{C}_{\infty}:=
\{\,(\lambda, U(\,\cdot\,,\lambda))\in (\lambda_{1}, \infty )\times X\,\},
$$
where
$(\lambda_{1}, \infty)\ni\lambda\mapsto U(\,\cdot\,,\lambda)\in X$
is of class $C^{1}$.
Furthermore, it holds that
$$
\lim_{\lambda\searrow\lambda_{1}}U(\,\cdot\,,\lambda)= 0
\ \ \mbox{in}\ X
$$
and
$$
\lim_{\lambda\to\infty}\dfrac{U(\,\cdot\,,\lambda)}{\lambda}= \varPsi
\ \ \mbox{in}\ C^{1}(\overline{\Omega}),
$$
where $\varPsi$ is the unique positive solution
to the elliptic equation:
$$
-\Delta\varPsi^{2}=m(x)\varPsi\ \ \mbox{in}\ \Omega,\qquad
\varPsi=0\ \ \mbox{on}\ \partial\Omega.
$$
\end{thm}
The third result asserts that, 
when $\alpha$ is sufficiently large, 
a bifurcation curve of the subset of positive solutions to \eqref{SKT}
is constructed by the perturbation of the bifurcation curve 
$\mathcal{C}_{\infty}$:

\begin{thm}\label{Cathm}
For any $\varLambda >\lambda_{1}$,
there exists $\overline{\alpha}=\overline{\alpha}(\varLambda )>0$ such that,
if $\alpha >\overline{\alpha}$, then there exists a bifurcation curve
$\mathcal{C}_{\alpha, \varLambda}$ parameterized by 
$\lambda\in (\lambda_{1}, \varLambda\,]$
consisting of positive solutions of \eqref{SKT} as follows
$$
\mathcal{C}_{\alpha, \varLambda}=
\{\,(\lambda, u_{0, \alpha}(\,\cdot\,,\lambda), 
v_{0, \alpha}(\,\cdot\,,\lambda))\in
(\lambda_{1}, \varLambda\, ]\times\boldsymbol{X}\,\},
$$
where
$$
(\lambda_{1}, \varLambda\,]\times (\overline{\alpha},\infty)\ni 
(\lambda,\alpha)\mapsto
(u_{0, \alpha}(\,\cdot\,,\lambda), v_{0, \alpha}(\,\cdot\,,\lambda))
\in\boldsymbol{X}
$$
is of class $C^{1}$ satisfying
$$
\lim_{\lambda\searrow\lambda_{1}}
(u_{0, \alpha}(\,\cdot\,,\lambda), v_{0, \alpha}(\,\cdot\,,\lambda))=(0,0)
\ \ \mbox{in}\ \boldsymbol{X}
$$
and
\begin{equation}\label{ualto0}
\lim_{\alpha\to\infty}
\alpha \,(u_{0, \alpha}(\,\cdot\,,\lambda), 
v_{0, \alpha}(\,\cdot\,,\lambda))
=
(U(\,\cdot\,,\lambda), U(\,\cdot\,,\lambda))
\ \ \mbox{in}\ \boldsymbol{X}\ \ \mbox{for each}\ 
\lambda\in (\lambda_{1}, \varLambda\, ]
\end{equation}
with $(\lambda, U(\,\cdot\,\lambda))\in\mathcal{C}_{\infty}$ obtained in
Theorem \ref{thm13}.
Furthermore, $\mathcal{C}_{\alpha, \varLambda}$ can be extended as a connected
subset of positive solutions to the range $\lambda> \varLambda$.
\end{thm}
In the final section, the numerical bifurcation diagram will
support the theoretical result, Theorem \ref{Cathm}.
In Figure 1, the blue curve realizes 
$\mathcal{C}_{\alpha, \varLambda}$.
For a numerical verification of \eqref{ualto0},
see Figure 2 exhibiting small similar profiles of $u$ and $v$ 
corresponding to a point on  $\mathcal{C}_{\alpha, \varLambda}$.

The second limit system \eqref{LS2} 
also appears in the fast reaction limit of the diffusive 
Lotka-Volterra competition system (e.g., \cite{DD})
and it has been studied from this perspective. 
In the one-dimensional case with any homogeneous $m>0$,
the shooting argument  (e.g., \cite{HY})
enables us to classify the set of solutions of
\eqref{LS2} denoted as
$$
\widetilde{\mathcal{S}}_{\infty}:=\{
(\lambda,w)\in\mathbb{R}_{+}\times C^{2}([-\ell, \ell ])\,:\,
\mbox{$(\lambda,w)$ satisfies \eqref{LS2}}\,\}.
$$
\begin{lem}\label{basiclem}
Suppose that
$\Omega=(-\ell, \ell )$ and $m(x)=m$
(constant)
$>0$
for all $x\in\overline{\Omega}$.
For each $j\in\mathbb{N}$,
define
$$
\widetilde{\mathcal{S}}_{j,\infty}^{+}:=\{\,
(\lambda ,w)\in\widetilde{\mathcal{S}}_{\infty}\,:\,
\mbox{the number of zeros of $w$ on $(-\ell, \ell )$ is $j-1$ with
$w'(0)>0$}\,\}
$$
and
$$
\widetilde{\mathcal{S}}_{j,\infty}^{-}:=\{\,
(\lambda ,w)\in\widetilde{\mathcal{S}}_{\infty}\,:\,
\mbox{the number of zeros of $w$ on $(-\ell, \ell )$ is $j-1$ with
$w'(0)<0$}\,\}.
$$
Then, $\widetilde{\mathcal{S}}_{j, \infty}:=
\widetilde{\mathcal{S}}_{j, \infty}^{+}\,\cup\,\{(\lambda_{j}, 0)\}
\,\cup\,\widetilde{\mathcal{S}}_{j, \infty}^{-}$ forms a simple curve bifurcating from
the trivial solution $w=0$ at $\lambda =\lambda_{j}$ such that
$\widetilde{\mathcal{S}}_{j, \infty}^{+}$ and 
$\widetilde{\mathcal{S}}_{j, \infty}^{-}$ are 
parameterized by $\lambda\in (\lambda_{j}, \infty )$, such as
$$
\widetilde{\mathcal{S}}_{j, \infty}^{\pm}:=\{\,
(\lambda ,w^{\pm}_{j, 0}(\,\cdot\,,\lambda))\in
(\lambda_{j}, \infty)\times X\,\},
$$
where
$(\lambda_{j},\infty)\ni\lambda\mapsto
w^{\pm}_{j, 0}(\,\cdot\,,\lambda)$ are of class
$C^{1}$, and especially for the case where $j$ is even,
$
\widetilde{\mathcal{S}}^{-}_{j,\infty}=
\{\,w(-x)\,:\,w\in\widetilde{\mathcal{S}}^{+}_{j,\infty}\,\}.
$
Furthermore, it holds that
$$\widetilde{\mathcal{S}}_{\infty}=\bigcup\limits^{\infty}_{j=1}\widetilde{\mathcal{S}}_{j, \infty}\,\cup\,
\{\,(\lambda,w)\,:\,\lambda >0,\ w=0\,\}.$$
\end{lem}
The following result shows that,
for each fixed $j\ge 2$, 
a pitchfork bifurcation curve of position solutions of \eqref{SKT}
bifurcates
from a solution on $\mathcal{C}_{\alpha, \varLambda}$, and 
the upper (resp.\,lower) branch is approximated by
$(\lambda, w_{+}, 
w_{-})$
with 
$(\lambda, w)\in\widetilde{\mathcal{S}}^{+}_{j,\infty}$
(resp.\, $\widetilde{\mathcal{S}}^{-}_{j,\infty}$),
where
$w_{+}:=\max\{\,w,0\,\}$ and
$w_{-}:=\max\{\,-w,0\,\}$.

\begin{thm}\label{g2ndbifthm}
Suppose that
$\Omega=(-\ell, \ell )$ and $m(x)=m$
(constant)
$>0$
for all $x\in\overline{\Omega}$.
Let $j\in\mathbb{N}$ satisfy $j\ge 2$.
For any small $\delta>0$ and large $\varLambda >\lambda_{j}$,
there exists  
$\overline{\alpha }=
\overline{\alpha }(j, \delta, \varLambda )>0$ such that,
for any fixed $\alpha>\overline{\alpha}$,
the set of positive solutions of \eqref{SKT}
with $\lambda\in (\lambda_{j}-\delta, \varLambda\,]$  
contains a pitchfork bifurcation curve 
$\mathcal{S}_{j,\alpha,\varLambda}$
bifurcating from a solution
$$
(\lambda, u, v)=
(\mu_{j, \alpha},
u_{0, \alpha}(\,\cdot\,, \mu_{j,\alpha}),
v_{0, \alpha}(\,\cdot\,, \mu_{j,\alpha}))
\in\mathcal{C}_{\alpha, \varLambda},
$$
where the mapping
$(\overline{\alpha},\infty )\ni \alpha\mapsto
\mu_{j,\alpha}\in\mathbb{R}$
is continuous and satisfies
\begin{equation}\label{biflim}
\lim_{\alpha\to\infty}\mu_{j,\alpha }=\lambda_{j}
\quad\mbox{for each}\ j\ge 2.
\end{equation}
The upper branch $\mathcal{S}_{j, \alpha, \varLambda}^{+}$
of $\mathcal{S}_{j,\alpha,\varLambda}$
is parameterized as
$$
\mathcal{S}_{j, \alpha, \varLambda}^{+}=
\{\,
(\mu_{j,\alpha}, u_{0,\alpha}(\,\cdot\,, \mu_{j,\alpha}),
v_{0,\alpha}(\,\cdot\,, \mu_{j,\alpha}))+
(\xi_{j, \alpha}(s), \widetilde{u}_{j, \alpha }(\,\cdot\,,s),
\widetilde{v}_{j, \alpha}(\,\cdot\,,s))
\,:\,s\in (0,T_{j, \alpha, \varLambda}^{+}\,]\,\},
$$
while the lower branch $\mathcal{S}_{j, \alpha, \varLambda}^{-}$
of $\mathcal{S}_{j,\alpha,\varLambda}$ is parameterized as
$$
\mathcal{S}_{j, \alpha, \varLambda}^{-}=
\{\,
(\mu_{j,\alpha}, u_{0,\alpha}(\,\cdot\,, \mu_{j,\alpha}),
v_{0,\alpha}(\,\cdot\,, \mu_{j,\alpha}))+
(\xi_{j, \alpha}(s), \widetilde{u}_{j, \alpha }(\,\cdot\,,s),
\widetilde{v}_{j, \alpha}(\,\cdot\,,s))
\,:\,s\in [\,-T_{j, \alpha, \varLambda}^{-}, 0)\,\}
$$
with some positive numbers
$T_{j, \alpha, \varLambda}^{\pm}$,
where
$$
[\,-T_{j,\alpha, \varLambda}^{-}, 
T_{j, \alpha, \varLambda}^{+}\,]
\ni s\mapsto 
(\xi_{j,\alpha}(s), 
\widetilde{u}_{j,\alpha}(s), 
\widetilde{v}_{j,\alpha}(s))
\in \mathbb{R}_{+}\times\boldsymbol{X}
$$
is of class $C^{1}$ with
$$
(\xi_{j,\alpha}(0), 
\widetilde{u}_{j,\alpha}(\,\cdot\,,0), 
\widetilde{v}_{j,\alpha}(\,\cdot\,,0))
=(0, 0, 0)$$
and
$
\mu_{j,\alpha}+\xi_{j, \alpha}(
-T_{j,\alpha}^{-})
=\varLambda
=
\mu_{j,\alpha}+\xi_{j, \alpha}(
T_{j,\alpha}^{+})
$
and
\begin{equation}\label{comp}
\begin{split}
&\lim_{\alpha\to\infty}
(\mu_{j,\alpha}+\xi_{j,\alpha}(s),
u_{0,\alpha}(\,\cdot\,, \mu_{j,\alpha})+
\widetilde{u}_{j, \alpha}(s),
v_{0, \alpha}(\,\cdot\,, \mu_{j, \alpha})+
\widetilde{v}_{j, \alpha}(\,\cdot\,,s))\\
=&
\begin{cases}
(\lambda, w_{j,0}^{+}(\,\cdot\,,\lambda)_{+},
w_{j,0}^{+}(\,\cdot\,,\lambda)_{-})\quad
&\mbox{if}\ s>0,\\
(\lambda, w_{j,0}^{-}(\,\cdot\,,\lambda)_{+},
w_{j,0}^{-}(\,\cdot\,,\lambda)_{-})\quad
&\mbox{if}\ s<0
\end{cases}
\end{split}
\end{equation}
for some $(\lambda, w_{j,0}^{+}(\,\cdot\,,\lambda))\in
\widetilde{\mathcal{S}}^{+}_{j,\infty}$ and
$(\lambda, w_{j,0}^{-}(\,\cdot\,,\lambda))\in
\widetilde{\mathcal{S}}^{-}_{j,\infty}$.
Furthermore, if $j$ is even, then any
\begin{equation}
\begin{split}
&(\lambda(s), u(\,\cdot\,,s), v(\,\cdot\,,s))\\
:=&
(\mu_{j,\alpha}, u_{0,\alpha}(\,\cdot\,, \mu_{j,\alpha}),
v_{0,\alpha}(\,\cdot\,, \mu_{j,\alpha}))+
(\xi_{j, \alpha}(s), \widetilde{u}_{j, \alpha }(\,\cdot\,,s),
\widetilde{v}_{j, \alpha}(\,\cdot\,,s))\in \mathcal{S}_{j,\alpha, \varLambda}
\end{split}
\nonumber
\end{equation}
satisfies
\begin{equation}\label{evensym}
(\lambda (s),
u(x,s),
v(x,s))=
(\lambda (-s),
u(-x,-s),
v(-x,-s))
\end{equation}
for any $x\in [-\ell, \ell ]$ and 
$s\in [-T^{-}_{j, \alpha, \varLambda }, T^{+}_{j, \alpha, \varLambda }]$
with $T^{+}_{j, \alpha, \varLambda }=T^{-}_{j, \alpha, \varLambda }$.
\end{thm}
In the numerical bifurcation diagram (Figure 1),
the red and purple curve realize $\mathcal{S}_{2, \alpha, \varLambda}^{\pm}$
and $\mathcal{S}_{3, \alpha, \varLambda}^{\pm}$, respectively.
The complete segregation expressed as \eqref{comp} can be realized
by the numerical profiles of $u$ and $v$ in Figure 4(b).

\section{Full cross-diffusion limit}
In this section, we study the limiting behavior of positive solutions
of \eqref{SKT} as $\alpha\to\infty$ to prove Theorem \ref{convthm}.
\subsection{A priori estimate}
The following a priori estimate independent of $\alpha$ will 
play an important role in
a compactness argument for the sequence of 
positive solutions as $\alpha\to\infty$:
\begin{lem}\label{aprlem}
Let $(u,v)$ be any nonnegative solution of \eqref{SKT}.
Then it holds that
\begin{equation}\label{L2est}
\|u\|_{2}\le\dfrac{\lambda\|m\|_{2}}{b_{1}},\quad
\|v\|_{2}\le\dfrac{\lambda\|m\|_{2}}{c_{2}}
\end{equation}
and
\begin{equation}\label{Linfest}
(1 +\alpha v)u\le M_{1},\quad
(1 +\alpha u)v\le M_{2}\ \ \mbox{in}\ \Omega,
\end{equation}
where $M_{1}$ and $M_{2}$ are positive constants depending on
$(\lambda, \|m\|_{\infty}, b_{1}, \Omega)$ and
$(\lambda, \|m\|_{\infty}, c_{2}, \Omega)$, 
respectively.
\end{lem}

\begin{proof}
Integrating \eqref{SKT1} over $\Omega$, we use the divergence theorem
to see
$$
\displaystyle\int_{\partial\Omega}
\dfrac{\partial }{\partial\nu}[\,(1+\alpha v)u\,]
+\displaystyle\int_{\Omega}u(\lambda m(x)-b_{1}u-c_{1}v)=0.
$$
Here 
the surface integral over $\partial\Omega $ is nonpositive by 
the boundary conditions on $u$ and $v$.
Then we obtain
$$
b_{1}\|u\|^{2}_{2}\le\displaystyle\int_{\Omega}u(\lambda m(x)-c_{1}v)\le
\displaystyle\lambda\int_{\Omega}m(x)u\le 
\lambda\|m\|_{2}\|u\|_{2},$$
thereby, $\|u\|_{2}\le \lambda\|m\|_{2}/b_{1}$.
The same procedure for \eqref{SKT2} yields 
the latter estimate in \eqref{L2est}.

In order to derive \eqref{Linfest},
we express \eqref{SKT1} as
\begin{equation}\label{inv}
(1+\alpha v)u=(-\Delta )^{-1}[\,u(\lambda m(x)-b_{1}u-c_{1}v)\,].
\end{equation}
Here,
$(-\Delta )^{-1}$ represents the inverse operator of $-\Delta$ with the
homogeneous Dirichlet boundary condition on $\partial\Omega$.
Thanks to the boundary condition and the maximum principle,
$(-\Delta )^{-1}$ can be regarded as a bounded linear operator
from $C(\overline{\Omega})$ to $C_{0}(\overline{\Omega}):=\{\,
w\in C(\overline{\Omega})\,:\,w=0\ \mbox{on}\ \partial\Omega\,\}$ 
with the monotone property 
such that,
if $f_{1}\le f_{2}$ in $\Omega$,
then $(-\Delta)^{-1}f_{1}\le (-\Delta)^{-1}f_{2}$ in $\Omega$.
Since $u\ge 0$ and $v\ge 0$ in $\Omega$, then
$$
u(\lambda m(x)-b_{1}u-c_{1}v)\le u(\lambda m(x)-b_{1}u)\le 
\dfrac{\lambda^{2}m(x)^{2}}{4b_{1}}
\quad\mbox{for any}\ x\in \Omega.
$$
For \eqref{inv}, we use the monotone property of $(-\Delta)^{-1}$ to see
$$
(1+\alpha v)u\le
\dfrac{\lambda^{2}}{4b_{1}}(-\Delta )^{-1}m^{2}
\le
\dfrac{\lambda^{2}}{4b_{1}}\|(-\Delta )^{-1}m^{2}\|_{\infty}=:M_{1}
\quad\mbox{in}\ \Omega.$$
The same procedure for \eqref{SKT2} leads us to
the latter estimate in \eqref{Linfest}.
\end{proof}

\subsection{Compactness of positive solutions in the full
cross-diffusion limit}
Based on a priori estimate of positive solutions of
\eqref{SKT}, we begin with the following lemma for
the proof of Theorem \ref{convthm}.
\begin{lem}\label{wzconvlem}
Let $\{(u_{n}, v_{n})\}$ be any sequence of positive solutions
of \eqref{SKT} with $\alpha=\alpha_{n}\to\infty$.
Then there exists $w^{*}\in X$
for any $p>1$ 
such that
$$
\lim_{n\to\infty}(u_{n}, v_{n})=(w_{+}^{*}, w_{-}^{*})
\ \ \mbox{uniformly in}\ \overline{\Omega},
$$
passing to a subsequence if necessary,
and moreover, $w^{*}$ is a solution of \eqref{LS2}.
\end{lem}

\begin{proof}
In \eqref{SKT},
we employ the transformation as
\begin{equation}\label{wzdef}
\varepsilon:=\dfrac{1}{\alpha},\quad
w:=u-v,\quad z:= (\varepsilon +v)u.
\end{equation}
Hence $(u,v)$ is conversely expressed as
\begin{equation}\label{uvdef}
(u,v)
=
\dfrac{1}{2}
\left(\sqrt{(w-\varepsilon )^{2}+4z}+w-\varepsilon,\,
\sqrt{(w-\varepsilon )^{2}+4z}-w-\varepsilon\right).
\end{equation}
Then \eqref{SKT} is 
reduced to the following semilinear system with unknown fucntions
$w$ and $z$:
\begin{equation}\label{wzeq}
\begin{cases}
\Delta w +\lambda m(x) w-u(b_{1}u+c_{1}v)+v(b_{2}u+c_{2}v)=0
\quad&\mbox{in}\ \Omega,\\
\Delta z+\varepsilon u(\lambda m(x)-b_{1}u-c_{1}v)=0
\quad&\mbox{in}\ \Omega,\\
w=z=0\quad&\mbox{on}\ \partial\Omega,
\end{cases}
\end{equation}
where $(u,v)$ is considered as the function of $(w,z, \varepsilon )$
defined by \eqref{uvdef}. 
Here we recall \eqref{Linfest} to note that
\begin{equation}\label{unvnbdd}
0< u_{n}<M_{1},\quad
0< v_{n}<M_{2}\ \ \mbox{in}\ \Omega
\end{equation}
for all $n\in\mathbb{N}$.
Then there exists a positive constant 
$M_{3}=M_{3}(\lambda, b_{i}, c_{i}, \|m\|_{\infty}, \Omega)$ such that
$$
\|\lambda m(x) w_{n}-u_{n}(b_{1}u_{n}+c_{1}v_{n})+
v_{n}(b_{2}u_{n}+c_{2}v_{n})\|_{\infty}\le M_{3},\quad
\|u_{n}(\lambda m(x)-b_{1}u_{n}-c_{1}v_{n})\|_{\infty}\le M_{3}
$$
for all $n\in\mathbb{N}$.
Then \eqref{wzeq} with $\varepsilon=\varepsilon_{n}:=1/\alpha_{n}$ 
implies that 
$$\|\Delta w_{n}\|_{\infty}\le M_{3},\quad
\|\Delta z_{n}\|_{\infty}\le
\dfrac{M_{3}}{\alpha_{n}},
$$
where
$w_{n}:=u_{n}-v_{n}$ and
$z_{n}:=(\varepsilon_{n}+v_{n})u_{n}$.
Therefore, for any $p\in (1,\infty)$,
the elliptic estimate (e.g., \cite{GT}),
enables us to find a positive constant
$M_{4}(\lambda, b_{i}, c_{i}, \|m\|_{\infty}, \Omega )$
independent of $n$, such that
$$
\|w_{n}\|_{W^{2,p}}\le M_{4},\quad
\|z_{n}\|_{W^{2,p}}\le \dfrac{M_{4}}{\alpha_{n}}
$$ 
for all $n\in\mathbb{N}$.
Since $X$ is compactly embedded in $C^{1}(\overline{\Omega })$ 
if $p>N$, then $\{(w_{n},z_{n})\}$ contains a convergence subsequence,
which is relabeled as 
$\{(w_{n},z_{n})\}$ for simplicity, such as
\begin{equation}\label{wzconv2}
\lim_{n\to\infty}(w_{n}, z_{n})=(w^{*}, 0)
\ \ \mbox{weakly in}\ \boldsymbol{X}\ \mbox{and strongly in}\ 
C^{1}(\overline{\Omega})^{2}
\end{equation}
with some $w^{*}\in X$.
Then it follows from \eqref{uvdef} that
\begin{equation}\label{uvconv}
\lim_{n\to\infty}(u_{n}, v_{n})=
\dfrac{1}{2}(|w^{*}|+w^{*}, |w^{*}|-w^{*})
=(w_{+}^{*}, w_{-}^{*})
\ \ \mbox{uniformly in}\ \overline{\Omega}.
\end{equation}
Here we remark that \eqref{Linfest} implies 
$$
\lim_{n\to\infty}u_{n}v_{n}=0
\ \ \mbox{uniformly in}\ \overline{\Omega}.
$$
Consequently, setting
$n\to\infty$ in \eqref{wzeq} with 
$(w,z,u,v, \varepsilon)=(
w_{n},z_{n},u_{n},v_{n}, \varepsilon_{n})$, 
we obtain
\begin{equation}\label{wzstareq}
\begin{cases}
\Delta w^{*} +\lambda m(x) w^{*}
-b_{1}(w_{+}^{*})^{2}+c_{2}(w_{-}^{*})^{2}=0
\quad&\mbox{in}\ \Omega,\\
w^{*}=0\quad&\mbox{on}\ \partial\Omega.
\end{cases}
\end{equation}
The proof of Lemma \ref{wzconvlem} is complete.
\end{proof}
From Lemma \ref{wzconvlem}, 
if $w^{*}$ is a sing-changing solution, then
$(u_{n},v_{n})$ with large $\alpha=\alpha_{n}$ 
can be expected to realize
the segregation of two competing species as
the solution of \eqref{SKT}.
However, it still remains a possibility that
$w^{*}$ is a positive or a negative or the trivial solution.
The next lemma excludes the case where
$w^{*}$ is a positive or a negative solution.

\begin{lem}\label{zerolem}
Let $\{(u_{n},v_{n})\}$ be any sequence of positive solutions of  \eqref{SKT}
with $\alpha=\alpha_{n}\to\infty$.
If
$u_{n}\to 0$ or
$v_{n}\to 0$ almost everywhere in $\Omega$, 
then 
$$
\lim_{n\to\infty}(u_{n},v_{n})=(0,0)
\ \ \mbox{uniformly in}\ \overline{\Omega}.$$
\end{lem}

\begin{proof}
Assume that
$u_{n}\to 0$ almost everywhere in $\Omega$.
By repeating the proof of Lemma \ref{wzconvlem},
one can verify that
\eqref{uvconv} holds with
$w^{*}_{+}=0$ in $\Omega$,
subject to a subsequence if necessary.
It follows from Lemma \ref{wzconvlem} that
$w^{*}=-w^{*}_{-}\le 0$ in $\Omega$,
and moreover,
$$
\lim_{n\to\infty}(u_{n}, v_{n})=(0,|w^{*}|)
\ \ \mbox{uniformly in}\ \overline{\Omega}.
$$
Then \eqref{wzstareq} implies that
$\Delta w^{*}+\lambda m(x)w^{*}+c_{2}(w^{*})^{2}=0$
in 
$\Omega$.
Therefore, we see that
$-c_{2}w^{*}$ is a nonnegative solution of
the following Dirichlet problem of the logistic
equation:
\begin{equation}\label{logi}
\begin{cases}
\Delta\theta+\theta (\lambda m(x)-\theta )=0\quad&\mbox{in}\ \Omega,\\
\theta=0\quad &\mbox{on}\ \partial\Omega.
\end{cases}
\end{equation}
The global bifurcation structure of positive solutions of 
\eqref{logi} was obtained by Cantrell and Cosner
\cite{CC}.
According to their result, the set of all positive solutions
of \eqref{logi} forms a simple curve in $C^{1}(\overline{\Omega})$
parameterized by $\lambda\in (\lambda_{1},\infty)$, and moreover,
the curve bifurcates from the trivial solution at
$\lambda=\lambda_{1}$.
In what follows, the unique positive solution of \eqref{logi}
for each $\lambda\in (\lambda_{1}, \infty)$ 
will be denoted by $\theta_{\lambda}$.
Since $\{(u_{n},v_{n})\}$ is a sequence of positive solutions
of \eqref{SKT} with $\alpha=\alpha_{n}$, then
Lemma \ref{exlem} ensures $\lambda\in (\lambda_{1},\infty)$.
It follows that
$w^{*}=-\theta_{\lambda }/c_{2}$ or $w^{*}=0$.

Suppose for contradiction that
$w^{*}=-\theta_{\lambda}/c_{2}$.
Hence the second equation of \eqref{wzeq} implies that
$
\Delta [(\varepsilon_{n}+v_{n})u_{n}]
+\varepsilon_{n}u_{n}(\lambda m(x)-b_{1}u_{n}-c_{1}v_{n})=0$ in $\Omega$.
Then, for any $p\in (1, \infty )$, 
multiplying this equation by $1/\|u_{n}\|_{p}$,
we see that
$\widehat{u}_{n}:=u_{n}/\|u_{n}\|_{p}$ satisfies
$$
\Delta[(\varepsilon_{n}+v_{n})\widehat{u}_{n}]+
\varepsilon_{n}\widehat{u}_{n}(\lambda m(x)-b_{1}u_{n}-c_{1}v_{n})=0
\quad\mbox{in}\ \Omega.
$$
Hence \eqref{unvnbdd} ensures the uniform boundedness of
$\{\,\|\widehat{u}_{n}(\lambda m(x)-b_{1}u_{n}-c_{1}v_{n})\|_{p}\,\}$.
Together with the elliptic estimate,
we can find a positive constant
$M_{5}$ independent of $n$ such that
$$
\|(\varepsilon_{n} +v_{n})\widehat{u}_{n}\|_{W^{2,p}}\le
\dfrac{M_{5}}{\alpha_{n}}
\quad\mbox{for all}\ n\in\mathbb{N}.
$$
Then the Sobolev embedding theorem gives a positive constant $M_{6}$
such that
\begin{equation}\label{M6}
0<(\varepsilon_{n} +v_{n}(x))\widehat{u}_{n}(x)
<\dfrac{M_{6}}{\alpha_{n}},
\end{equation}
namely,
$$
0<\widehat{u}_{n}(x)<
\dfrac{M_{6}}{1+\alpha_{n}v_{n}(x)}=:h_{n}(x)
\ \ 
\mbox{for any}\ x\in\overline{\Omega}.
$$
By the assumption that $v_{n}\to\theta_{\lambda}/c_{2}>0$
uniformly in $\overline{\Omega}$, we see that
$$
\lim_{n\to\infty}h_{n}=0
\ \ \mbox{uniformly in any compact subset of $\Omega$.}
$$
Together with the fact that 
$h_{n}(x)\le M_{6}$ for any $x\in\overline{\Omega}$,
we see that
$\widehat{u}_{n}\to 0$ in $L^{p}(\Omega )$ for any 
$p\in (1, \infty )$.
However, this contradicts $\|\widehat{u}_{n}\|_{p}=1$
for any $n\in\mathbb{N}$.
Consequently, we obtain $w^{*}=0$ in $\Omega$, that is,
$$
\lim_{n\to\infty}v_{n}=0\ \ \mbox{uniformly in $\overline{\Omega}$}.
$$
Therefore,
we deduce that $(u_{n}, v_{n})\to (0,0)$
uniformly in $\overline{\Omega}$
if $u_{n}\to 0$ almost everywhere in
$\Omega$.

Obviously, the other assumption that $v_{n}\to 0$
almost everywhere in $\Omega$ leads us to the uniform decays of
$u_{n}$ and $v_{n}$ in the same manner.
Thus we complete the proof of Lemma \ref{zerolem}.
\end{proof}

From Lemma \ref{zerolem}, 
we know that, 
in the full cross-diffusion limit as $\alpha=\alpha_{n}\to\infty$, 
decays of species necessarily occur simultaneously 
for both species. 
The following lemma asserts that,
if such a scenario occurs,
both species decay with the same order.

\begin{lem}\label{ratiolem}
Let $\{(u_{n}, v_{n})\}$ be any sequence of positive solutions of
\eqref{SKT} with $\alpha=\alpha_{n}\to\infty$ satisfying
$$
\lim_{n\to\infty}(u_{n}, v_{n})=(0,0)
\ \ \mbox{uniformly in}\ \overline{\Omega}.$$
Then there exists $\delta\in (0,1)$ independent of $n$ such that
\begin{equation}\label{quoest}
\delta<\dfrac{\|u_{n}\|_{\infty}}{\|v_{n}\|_{\infty}}<\dfrac{1}{\delta}
\quad
\mbox{for all}\ n\in\mathbb{N}.
\end{equation}.
\end{lem}

\begin{proof}
Suppose for contradiction that 
\begin{equation}\label{asum}
\lim_{n\to\infty}\dfrac{\|u_{n}\|_{\infty}}{\|v_{n}\|_{\infty}}=0
\end{equation}
with some subsequence, 
which is denoted by $\{(u_{n}, v_{n})\}$ again
for simplicity.
Hence the first equation of \eqref{wzeq} implies that
$$
\Delta (u_{n}-v_{n})+\lambda m(x)(u_{n}-v_{n})-u_{n}(b_{1}u_{n}+c_{1}v_{n})
+v_{n}(b_{2}u_{n}+c_{2}v_{n})=0
\ \ \mbox{in}\ \Omega.$$
By multiplying this equation by $1/\|v_{n}\|_{\infty}$, we see that
$\widetilde{v}_{n}:=v_{n}/\|v_{n}\|_{\infty}$ satisfies
\begin{equation}\label{tiln}
-\Delta \biggl(
\dfrac{u_{n}}{\|v_{n}\|_{\infty}}
-\widetilde{v}_{n}
\biggr)=
\lambda m(x)
\biggl(
\dfrac{u_{n}}{\|v_{n}\|_{\infty}}-\widetilde{v}_{n}
\biggr)
-\dfrac{u_{n}}{\|v_{n}\|_{\infty}}(b_{1}u_{n}+c_{1}v_{n})
+\widetilde{v}_{n}(b_{2}u_{n}+c_{2}v_{n})
\ \ \mbox{in}\ \Omega.
\end{equation}
Since $0<u_{n}<M_{1}$ and $0<v_{n}<M_{2}$ in $\Omega$
by \eqref{Linfest}.
then the assumption \eqref{asum} implies that
the right-hand side of \eqref{tiln}
is uniformly bounded in $L^{\infty}(\Omega )$.
Then the elliptic regularity ensures the uniform boundedness of
$$\biggl\{\,\widetilde{v}_{n}-
\dfrac{u_{n}}{\|v_{n}\|_{\infty}}\,\biggr\}
$$
in $X$ for any $p\in (1,\infty)$.
By the Sobolev embedding theorem, 
we can find $\widetilde{v}\in X$
for $p>N$ such that
$$
\lim_{n\to\infty}\biggl(
\widetilde{v}_{n}-
\dfrac{u_{n}}{\|v_{n}\|_{\infty}}\biggr)
=\widetilde{v}
\ \ \mbox{weakly in $X$ and strongly in 
$C^{1}(\overline{\Omega })$},
$$
subject to a subsequence if necessary.
Therefore, setting $n\to\infty$ in \eqref{tiln}, we see that
$$
-\Delta\widetilde{v}=\lambda 
m(x)\widetilde{v}\ \ \mbox{in}\ \Omega,\quad
\widetilde{v}=0\ \ \mbox{on}\ \partial\Omega.
$$
Together with $\|\widetilde{v}\|_{\infty}=1$,
the strong maximum principle implies $\widetilde{v}>0$ in $\Omega$.
Then it follows that $\lambda =\lambda_{1}$.
However, this is impossible because Lemma \ref{exlem} says that
there is no positive solution of \eqref{SKT} with 
$\lambda=\lambda_{1}$.
Consequently we obtain the 
lower estimate in \eqref{quoest} by the contradiction argument.
The upper estimate can be shown by the same manner.
\end{proof}

\begin{lem}\label{UVaprlem}
Suppose that 
$\lambda>\lambda_{1}$ and
$\lambda\neq \lambda_{j}$ for any $j\ge 2$.
Let $\{(u_{n}, v_{n})\}$ be any sequence of positive solutions of
\eqref{SKT} with $\alpha=\alpha_{n}\to\infty$ satisfying
\begin{equation}\label{decayass}
\lim_{n\to\infty}(u_{n}, v_{n})=(0,0)
\ \ \mbox{uniformly in}\ \overline{\Omega}.
\end{equation}
Then there exists a positive constant $M_{7}$ independent of
$n\in\mathbb{N}$ such that
$$
\alpha_{n}\|u_{n}\|_{\infty}\le M_{7},\quad
\alpha_{n}\|v_{n}\|_{\infty}\le M_{7}
\quad
\mbox{for all}\ n\in\mathbb{N}.
$$
\end{lem}

\begin{proof}
We argue by contradiction.
Suppose that $\alpha_{n}\|v_{n}\|_{\infty}\to\infty$
passing to a subsequence.
By the same procedure to get \eqref{tiln},
we see that 
$$
\widetilde{u}_{n}:=\dfrac{u_{n}}{\|u_{n}\|_{\infty}},\quad
\widetilde{v}_{n}:=\dfrac{v_{n}}{\|v_{n}\|_{\infty}},\quad
r_{n}:=\dfrac{\|u_{n}\|_{\infty}}{\|v_{n}\|_{\infty}}
$$
satisfy
\begin{equation}\label{tiln2}
-\Delta (
r_{n}
\widetilde{u}_{n}
-\widetilde{v}_{n}
)=
\lambda
m(x)
(
r_{n}
\widetilde{u}_{n}
-\widetilde{v}_{n}
)
-r_{n}\widetilde{u}_{n}(b_{1}u_{n}+c_{1}v_{n})
+\widetilde{v}_{n}(b_{2}u_{n}+c_{2}v_{n})
\ \ \mbox{in}\ \Omega,
\end{equation}
Here we recall Lemma \ref{ratiolem} to note that
$\delta<r_{n}<1/\delta$ for all $n\in\mathbb{N}$.
Therefore, thanks to the assumption \eqref{decayass} and 
the normalization procedure,
the right-hand side of 
\eqref{tiln2} is uniformly bounded
in $L^{\infty}(\Omega )$.
Then the elliptic regularity ensures the uniform boundedness of
$\{r_{n}\widetilde{u}_{n}-\widetilde{v}_{n}\}$ in $X$,
and hence,
the Sobolev embedding theorem gives 
$\widetilde{w}^{*}
\in X$ such that
\begin{equation}\label{tilconv}
\lim_{n\to\infty}(r_{n}\widetilde{u}_{n}-\widetilde{v}_{n})=
\widetilde{w}^{*}\ \ \mbox{weakly in $X$ and
strongly in $C^{1}(\overline{\Omega})$},
\end{equation}
passing to a subsequence if necessary.
By \eqref{decayass}, we set $n\to\infty$ in \eqref{tiln2}
to obtain
$$
-\Delta\widetilde{w}^{*}= \lambda
m(x)\widetilde{w}^{*}\ \ \mbox{in}\ \Omega,
\quad \widetilde{w}^{*}=0\ \ \mbox{on}\ \partial\Omega.
$$
Since $\lambda\neq\lambda_{j}$ for any $j\in\mathbb{N}$
by the assumption,
we deduce that
$\widetilde{w}^{*}=0$ in $\Omega$.

We now shift our focus on the following expression of 
$\widetilde{v}_{n}$ 
using \eqref{uvdef}:
\begin{equation}\label{root}
\widetilde{v}_{n}=
\dfrac{v_{n}}{\|v_{n}\|_{\infty}}=
\dfrac{1}{2}\biggl(
\sqrt{
\biggl(
\dfrac{w_{n}-\varepsilon_{n}}{\|v_{n}\|_{\infty}}
\biggr)^{2}+\dfrac{4z_{n}}{\|v_{n}\|^{2}_{\infty}}}
-\dfrac{w_{n}+\varepsilon_{n}}{\|v_{n}\|_{\infty}}
\biggr)
\ \ \mbox{with}\ \
\varepsilon_{n}=\dfrac{1}{\alpha_{n}}.
\end{equation}
From \eqref{tilconv} and the assumption 
$\alpha_{n}\|v_{n}\|_{\infty}\to\infty$, we know that
\begin{equation}\label{rootl}
\dfrac{w_{n}\mp\varepsilon_{n}}{\|v_{n}\|_{\infty}}
=
\dfrac{u_{n}-v_{n}\mp\alpha_{n}^{-1}}{\|v_{n}\|_{\infty}}
=
r_{n}\widetilde{u}_{n}-\widetilde{v}_{n}\mp\dfrac{1}
{\alpha_{n}\|v_{n}\|_{\infty}}\to \widetilde{w}^{*}
\end{equation}
weakly in $X$ and strongly in $C^{1}(\overline{\Omega})$.
Furthermore, following
the argument to get \eqref{M6},
we can verify that $z_{n}:=(\varepsilon_{n}+v_{n})u_{n}$ satisfies
$$
0<\dfrac{z_{n}}{\|v_{n}\|_{\infty}}
=r_{n}(\varepsilon_{n}+v_{n})\widetilde{u}_{n}<
r_{n}\dfrac{\widetilde{M}_{6}}{\alpha_{n}}
\ \ \mbox{in}\ \Omega
$$
with some positive constant $\widetilde{M}_{6}$.
In addition to the assumption $\alpha_{n}\|v_{n}\|_{\infty}\to\infty$,
using the fact that
$\delta<r_{n}<1/\delta$ for all $n\in\mathbb{N}$,
one can see that
\begin{equation}\label{rootr}
0<\dfrac{z_{n}}{\|v_{n}\|^{2}_{\infty}}<\dfrac{\widetilde{M}_{6}}
{\delta\alpha_{n}\|v_{n}\|_{\infty}}\to 0\quad\mbox{as}\ n\to\infty.
\end{equation}
By virtue of \eqref{rootl} and \eqref{rootr},
we set $n\to\infty$ in \eqref{root} to obtain
$$
\lim_{n\to\infty}\widetilde{v}_{n}=
\dfrac{|\widetilde{w}^{*}|-\widetilde{w}^{*}}{2}=\widetilde{w}^{*}_{-}=0
\quad\mbox{uniformly in}\ \overline{\Omega}.
$$
Obviously, this contradicts the fact that 
$\|\widetilde{v}_{n}\|_{\infty}=1$
for all $n\in\mathbb{N}$.
Consequently, this contradiction enables us to
conclude that $\{\alpha_{n}\|v_{n}\|_{\infty}\}$ 
is uniformly bounded.

It follows from Lemma \ref{ratiolem} that
$$
\delta<\dfrac{\alpha_{n}\|u_{n}\|_{\infty}}{\alpha_{n}\|v_{n}\|_{\infty}}<\dfrac{1}{\delta}\quad
\mbox{for all}\  n\in\mathbb{N}.
$$
Then the uniform boundedness of $\{\alpha_{n}\|v_{n}\|_{\infty}\}$ ensures
that of $\{\alpha_{n}\|u_{n}\|_{\infty}\}$.
Then the proof of Lemma \ref{UVaprlem} is complete.
\end{proof}
With the above lemmas, we are now ready to prove Theorem \ref{convthm}.
\begin{proof}[Proof of Theorem \ref{convthm}]
Suppose that $\lambda>\lambda_{1}$ and 
$\lambda\neq \lambda_{j}$ for any $j\ge 2$.
Let $\{(u_{n}, v_{n})\}$ be any sequence of positive solutions
of \eqref{SKT} with $\alpha=\alpha_{n}\to\infty$.
Then Lemma \ref{wzconvlem} gives 
$X$
for any $p\in (1,\infty )$ 
such that
$$
\lim_{n\to\infty}(u_{n}, v_{n})=(w_{+}^{*}, w_{-}^{*})
\ \ \mbox{uniformly in}\ \overline{\Omega},
$$
passing to a subsequence if necessary,
and moreover, $w^{*}$ is a solution of \eqref{LS2}.
Therefore, if both $w^{*}_{+}$ and $w^{*}_{-}$ are not identically equal to zero, 
then 
$w^{*}$ is a sign-changing solution of \eqref{LS2}.

Suppose that 
$w^{*}_{+}=0$ (that is, $u_{n}\to 0$ uniformly in $\overline{\Omega}$)
or $w^{*}_{-}=0$ (that is, $v_{n}\to 0$ uniformly in $\overline{\Omega}$).
Then, 
Lemma \ref{zerolem} implies
\begin{equation}\label{decay2}
\lim_{n\to\infty}(u_{n},v_{n})=(0,0)
\ \ \mbox{uniformly in}\ \overline{\Omega}.
\end{equation}
It follows from Lemma \ref{UVaprlem} that
$\{\alpha_{n}\|u_{n}\|_{\infty}\}$ and $\{\alpha_{n}\|v_{n}\|_{\infty}\}$
are uniformly bounded.
Multiplying \eqref{SKT} by $\alpha_{n}$, one can easily verify that
$$(U_{n}, V_{n}):=\alpha_{n}(u_{n}, v_{n})=\dfrac{1}{\varepsilon_{n}}(u_{n}, v_{n})$$ 
satisfies
$$
\begin{cases}
-\Delta [\,(1+V_{n})U_{n}\,]=U_{n}(\lambda m(x)-b_{1}u_{n}-c_{1}v_{n})
\ \ &\mbox{in}\ \Omega,\\
-\Delta [\,(1+U_{n})V_{n}\,]=V_{n}(\lambda m(x)-b_{2}u_{n}-c_{2}v_{n})
\ \ &\mbox{in}\ \Omega,\\
U_{n}=V_{n}=0\ \ &\mbox{on}\ \partial\Omega.
\end{cases}
$$
By employing the change of variables
\begin{equation}\label{WnZn1}
W_{n}:=U_{n}-V_{n},\quad
Z_{n}=(1+{V}_{n})U_{n},
\end{equation}
one can see that $(W_{n}, Z_{n})$ satisfies
\begin{equation}\label{WnZn2}
\begin{cases}
-\Delta W_{n}=\lambda m(x)W_{n}-U_{n}(b_{1}u_{n}+c_{1}v_{n})
+V_{n}(b_{2}u_{n}+c_{2}v_{n})
\quad&\mbox{in}\ \Omega,\\
-\Delta Z_{n}=U_{n}(\lambda m(x)-b_{1}u_{n}-c_{1}v_{n})
\quad&\mbox{in}\ \Omega,\\
W_{n}=Z_{n}=0
\quad&\mbox{on}\ \partial\Omega.
\end{cases}
\end{equation}
Since the right-hand sides of elliptic equations of \eqref{WnZn2}
are uniformly bounded in
$L^{\infty}(\Omega )$,
then the compactness argument using the elliptic regularity and
the Sobolev embedding theorem ensure
$(W^{*}, Z^{*})\in \boldsymbol{X}$ such that
\begin{equation}\label{WnZn3}
\lim_{n\to\infty}(W_{n}, Z_{n})=(W^{*}, Z^{*})
\ \ \mbox{weakly in $\boldsymbol{X}$ and 
strongly in $C^{1}(\overline{\Omega })^{2}$}.
\end{equation}
Together with \eqref{decay2} and
Lemma \ref{UVaprlem}, setting $n\to\infty$
in the first equation of \eqref{WnZn2}, we obtain
$$
-\Delta W^{*}=\lambda m(x)W^{*}\ \ \mbox{in}\ \Omega,\quad
W^{*}=0\ \ \mbox{on}\ \partial\Omega.
$$
Since $\lambda\neq \lambda_{j}$ for all $j\in\mathbb{N}$
by the assumption, then
one can see that
$W^{*}=0$ in $\Omega$.
It follows from \eqref{WnZn1} that
$$
U_{n}=\dfrac{1}{2}(
\sqrt{(1-W_{n})^{2}+4Z_{n}}-1+W_{n}),\quad
V_{n}=\dfrac{1}{2}
(\sqrt{(1-W_{n})^{2}+4Z_{n}}-1-W_{n}).
$$
Due to \eqref{WnZn3} and $W^{*}=0$, setting $n\to\infty$,
we have
\begin{equation}\label{uni}
\lim_{n\to\infty}U_{n}=\lim_{n\to\infty}V_{n}=
\dfrac{1}{2}(\sqrt{4Z^{*}+1}-1)(=:U^{*})
\ \ \mbox{in}\ C^{1}(\overline{\Omega}).
\end{equation}
Then setting $n\to\infty$ in $Z_{n}=(1+V_{n})U_{n}$,
we obtain
$Z^{*}=  (1+U^{*})U^{*}$.
Owing to \eqref{uni}, setting
$n\to\infty$ in the second equation of \eqref{WnZn2},
we see that $U^{*}$ is a nonnegative solution of
$$
-\Delta [\,(1+U)U\,]=\lambda m(x)U\ \ \mbox{in}\ \Omega,\quad
U=0\ \ \mbox{on}\ \partial\Omega.
$$

It remains to show $U^{*}>0$ in $\Omega$.
Suppose by contradiction that $U^{*}=0$.
Multiplying the second equation of \eqref{WnZn2} by $1/\|U_{n}\|_{\infty}$,
we see that
$$\widetilde{Z}_{n}:=(1+V_{n})\widetilde{U}_{n}
\quad\mbox{with}\quad
\tilde{U}_{n}:=\dfrac{U_{n}}{\|U_{n}\|_{\infty}}
$$
satisfies
\begin{equation}\label{norZ}
-\Delta\widetilde{Z}_{n}
=\dfrac{\widetilde{Z}_{n}}{1+V_{n}}
(\lambda m(x)-b_{1}u_{n}-c_{1}v_{n})\ \ \mbox{in}\ \Omega,\quad
\widetilde{Z}_{n}=0\ \ \mbox{on}\ \partial\Omega.
\end{equation}
In the same manner to get \eqref{WnZn3}, one can find
$\widetilde{Z}^{*}\in X$ such that
$
\widetilde{Z}_{n}\to\widetilde{Z}^{*}
$
weakly in $X$ and strongly in $C^{1}(\overline{\Omega })$,
passing to a subsequence.
With the assumption that 
$V_{n}\to 0$ uniformly in $\overline{\Omega}$
(see \eqref{uni} with $U^{*}=0$), we set $n\to\infty$ in 
\eqref{norZ} to see
$$
-\Delta\widetilde{Z}^{*}=\lambda m(x)\widetilde{Z}^{*}\ \ \mbox{in}\ \Omega,\quad
\widetilde{Z}^{*}=0\ \ \mbox{on}\ \partial\Omega.
$$
We observe that
$$\|\widetilde{Z}_{n}\|_{\infty}\ge\widetilde{Z}_{n}(x)=
(1+V_{n}(x))\widetilde{U}_{n}(x)
\ge\widetilde{U}_{n}(x)
\quad
\mbox{for all}\ x\in\Omega,
$$ 
and thereby,
$\|\widetilde{Z}_{n}\|_{\infty}\ge 
\|\widetilde{U}_{n}\|_{\infty}=1$
for all $n\in\mathbb{N}$.
It follows that
$\widetilde{Z}^{*}$ is a positive solution.
Therefore, we must deduce $\lambda=\lambda_{1}$.
However, this is a contradiction.
Consequently,
we obtain $U^{*}>0$ in $\Omega$ by contradiction.
The proof of Theorem \ref{convthm} is accomplished.
\end{proof}

\section{Small coexitence}
In the former half of this section, we construct the set of positive solutions
to the limiting system \eqref{LS1}.
In view of the altenative (i) of Theorem \ref{convthm},
one expects that positive solutions of \eqref{LS1}
can realize small coexistence
spatial profiles of both species with large cross-diffusion.
Actually, in the latter half of this section,
a subset of positive solutions
of the original SKT model \eqref{SKT} with 
large $\alpha$ will be constructed by the perturbation 
of the set of positive solutions of \eqref{LS1}.

\subsection{Global bifurcation structure for the 
limiting system of small coexistence}
To get the set of positive solutions of \eqref{LS1},
we employ the change of variables
\begin{equation}\label{Zdef}
Z:=(1+U)U,\quad\mbox{conversely,}\quad
U=\dfrac{1}{2}(\sqrt{4Z+1}-1).
\end{equation}
Then \eqref{LS1} is reduced to the following semilinear equation:
\begin{equation}\label{subell}
\begin{cases}
-\Delta Z=\dfrac{\lambda m(x)}{2}(\sqrt{4Z+1}-1)
\quad &\mbox{in}\ \Omega,\\
Z=0\quad &\mbox{on}\ \partial\Omega.
\end{cases}
\end{equation}
Thanks to the one-to-one correspondence \eqref{Zdef},
in order to know the set of positive solutions of \eqref{LS1},
it suffices to know that of \eqref{subell}.
By regarding $\lambda$ as the bifurcation parameter,
we show the following global bifurcation structure of 
positive solutions of \eqref{subell}.
\begin{prop}\label{subprop}
All positive solutions of \eqref{subell} form a bifurcation curve
parameterized by $\lambda\in (\lambda_{1}, \infty )$ such as
\begin{equation}\label{curve}
\{\,(\lambda, Z_{0}(\,\cdot\,,\lambda))\in (\lambda_{1}, \infty )\times 
X\,\},
\end{equation}
where
$(\lambda_{1}, \infty )\ni \lambda\mapsto Z_{0}(\,\cdot\,,\lambda)\in X$
is of class $C^{1}$.
Furthermore, it holds that
$$
\lim_{\lambda\searrow\lambda_{1}}Z_{0}(\,\cdot\,,\lambda)= 0
\ \ \mbox{in}\ X
$$
and
\begin{equation}\label{sin}
\lim_{\lambda\to\infty}\lambda^{-2}Z_{0}(\,\cdot\,,\lambda)= \zeta_{0}
\ \ \mbox{in}\ C^{1}(\overline{\Omega}),
\end{equation}
where $\zeta_{0}$ is the unique positive solution
to the sublinear elliptic equation
$$
-\Delta\zeta_{0}=m(x)\sqrt{\zeta_{0}}\ \ \mbox{in}\ \Omega,\quad
\zeta_{0}=0\ \ \mbox{on}\ \partial\Omega.
$$
\end{prop}
For the sake of the proof of Proposition \ref{subprop},
we begin with the range of $\lambda$ for the nonexistence of
positive solutions of \eqref{subell}:
\begin{lem}\label{noexlem2}
If $\lambda\le\lambda_{1}$, then \eqref{subell} 
does not admit any positive solution.
\end{lem}
\begin{proof}
Suppose that \eqref{subell} admits a positive solution $Z$.
Let $\varPhi_{1}$ be the 
$L^{\infty}$ normalized positive eigenfunction of \eqref{ev} with 
$\lambda=\lambda_{1}$. It follows that
\begin{equation}\label{egnor}
-\Delta\varPhi_{1}=\lambda_{1}m(x)\varPhi_{1},\ \ \varPhi_{1}>0\ \ \mbox{in}\  
\Omega,\quad
\varPhi_{1}=0\ \ \mbox{on}\ \partial\Omega,\quad
\|\varPhi_{1}\|_{\infty}=1.
\end{equation}
Multiplying \eqref{subell} by $\varPhi_{1}$ and integrating the resulting 
expression over $\Omega$, we have
$$
-\displaystyle\int_{\Omega}
\varPhi_{1}\Delta Z=
\dfrac{\lambda}{2}\int_{\Omega}
m(x)(\sqrt{4Z+1}-1)\varPhi_{1}
=
\int_{\Omega}
\dfrac{2\lambda m(x)Z}{\sqrt{4Z+1}+1}\varPhi_{1}<\lambda
\int_{\Omega}m(x)Z\varPhi_{1}.
$$
Hence the integration by parts implies
$$
(\lambda-\lambda_{1})
\int_{\Omega }m(x)Z\varPhi_{1}>0.
$$
It follows that $\lambda>\lambda_{1}$, in other words,
there is no positive solution of \eqref{subell} if
$\lambda\le\lambda_{1}$.
\end{proof}

Here we show the following a priori estimate of 
all positive solutions to \eqref{subell}.
\begin{lem}\label{aprlem2}
For any $p\in (1,\infty )$,
there exists
$M_{8}=M_{8}(\lambda, \|m\|_{\infty}, p, \Omega )>0$
such that any positive solution $Z$ of \eqref{subell} satisfies
$\|Z\|_{X}\le M_{8}$.
\end{lem}

\begin{proof}
In what follows,
we denote
$\sigma_{1}:=\lambda_{1}(1)$ for \eqref{ev},
namely,
$\sigma_{1}$ represents the least eigenvalue of $-\Delta $ with the
homogeneous Dirichlet boundary condition on $\partial\Omega$.
Hence \eqref{subell} can be expressed as
$$
Z
=\dfrac{1}{2}
\biggl(-\Delta-\dfrac{\sigma_{1}}{2}\biggr)^{-1}
\left[
\lambda m(x)(\sqrt{4Z+1}-1)-\sigma_{1}Z\right].
$$
By the maximum principle,
$(-\Delta -\sigma_{1}/2)^{-1}$ can be regarded as a bounded linear operator
from $C(\overline{\Omega})$ to $C_{0}(\overline{\Omega})$ 
with the monotone property 
such that,
if $f_{1}\le f_{2}$ in $\Omega$,
then $(-\Delta -\sigma_{1}/2)^{-1}f_{1}\le (-\Delta-\sigma_{1}/2)^{-1}f_{2}$ in $\Omega$.
Here we set
$$h(\xi):=\lambda\|m\|_{\infty}
(\sqrt{4\xi +1}-1)-\sigma_{1}\xi
\quad\mbox{for}\ \xi>0.
$$
There obviously exists a positive constant $M_{9}$,
which depends on $\lambda$ and $\|m\|_{\infty}$,
such that
$h(\xi )<M_{9}$ for all $\xi>0$.
Then the above-mentioned monotone property implies that 
any positive solution $Z$ of \eqref{subell} satisfies
$$
0<Z(x)< \dfrac{1}{2}
\biggl(-\Delta-\dfrac{\sigma_{1}}{2}\biggr)^{-1}M_{9}
\quad\mbox{for all}\ x\in\Omega.
$$
With the aid of the elliptic regularity, we obtain the desired bound $M_{8}$.
\end{proof}
Next we construct a local branch of positive solutions
bifurcating from the trivial solution at $\lambda=\lambda_{1}$.
For use of the local bifurcation theorem
from simple eigenvalue established by Crandall and Rabinowitz
\cite[Theorem 1.7]{CR}, 
we define an operator $F:\mathbb{R}\times X\to Y$
associated with \eqref{subell} by
\begin{equation}\label{Fdef}
F(\lambda, Z):=-\Delta Z-\dfrac{\lambda m(x)}{2}(\sqrt{4Z+1}-1).
\end{equation}
\begin{lem}\label{biflem}
In a neighborhood of $(\lambda, Z)=(\lambda_{1}, 0)\in\mathbb{R}\times X$,
all positive solutions of $F(\lambda, Z)=0$ form a 
bifurcation curve parameterized as
$$
\varGamma_{\delta}:=
\{\,(\lambda, Z)=(\lambda (s), 
s(\varPhi_{1}+\zeta (s)))\in\mathbb{R}\times X\,:\,
0<s<\delta\,\}$$
with some small $\delta >0$,
where $\lambda (s)$ is a function of class $C^{1}$ satisfying 
$\lambda(+0)=\lambda_{1}$ and
$\lambda (s)>\lambda_{1}$ for all $s\in (0, \delta )$;
$\varPhi_{1}$ is the eigenfunction defined by \eqref{egnor};
$\zeta (s)$ is an $X$-valued function of class $C^{1}$
satisfying  $\zeta (+0)=0$ and
$\int_{\Omega}\zeta (s)\varPhi_{1}=0$
for all $s\in (0, \delta )$.
\end{lem}

\begin{proof}
In order to find a bifurcation point of positive solutions to
$F(\lambda, Z)=0$ on the branch of the trivial solution,
we introduce the linearized operator
$L(\lambda ):=F_{Z}(\lambda, 0)\in\mathcal{L}(X,Y)$
where $\mathcal{L}(X,Y)$ denotes the Banach space of 
bonded linear operators from $X$ to $Y$.
It is easily verified that
\begin{equation}\label{Fz}
F_{Z}(\lambda, Z)\phi=
-\Delta\phi-\dfrac{\lambda m(x)}{\sqrt{4Z+1}}\phi,
\end{equation}
and therefore,
$$
L(\lambda )\phi=-\Delta\phi-\lambda m(x)\phi.$$
Then, $\mbox{Ker}\,L(\lambda )$ is the set of solutions to
$$-\Delta\phi =\lambda m(x)\phi\ \ \mbox{in}\ \Omega,\quad
\phi=0\ \ \mbox{on}\ \partial\Omega.
$$
It follows that $\mbox{Ker}\,L(\lambda )$ is nontrivial
if and only if $\lambda =\lambda_{j}$ with some $j\in\mathbb{N}$.
By virtue of the Krein-Rutman theorem,
there exists a positive
$\phi\in\mbox{Ker}\,L(\lambda )$ if and only if $\lambda=\lambda_{1}$.
Since $L(\lambda )$ is a formally self-adjoint operator,
the Fredholm alternative theorem implies that
\begin{equation}\label{Ran}
\mbox{Ker}\,L(\lambda_{1})=\mbox{Span}\,\{\varPhi_{1}\}
\ \ \mbox{and}\ \ 
\mbox{Ran}\,L(\lambda_{1})=\biggl\{
\,\psi\in Y\,:\,
\int_{\Omega}\varPhi_{1}\psi =0\,\biggr\}.
\end{equation}
Then, in order to apply the local bifurcation theorem
\cite[Theorem 1.7]{CR}, we need to check the 
transversality condition:
$$F_{Z,\lambda}(\lambda_{1}, 0)
\varPhi_{1}\not\in\mbox{Ran}\,L(\lambda_{1}).$$
Differentiating \eqref{Fz} by $\lambda$, one can verify that
$$F_{Z,\lambda}(\lambda_{1},0)\varPhi_{1}=
m(x)\varPhi_{1}.$$
By taking the $L^{2}$-inner product with $\varPhi_{1}$, we see that
$$\int_{\Omega}
F_{Z,\lambda}(\lambda_{1},0)\varPhi_{1}^{2}
=
\int_{\Omega}
m(x)\varPhi_{1}^{2}>0.
$$
In view of \eqref{Ran}, we deduce that
$F_{Z,\lambda}(\lambda_{1}, 0)\varPhi_{1}\not\in\mbox{Ran}\,L(\lambda_{1})$.
Therefore, we can use the local bifurcation theorem to know that,
near $(\lambda,Z)=(\lambda_{1}, 0)\in\mathbb{R}\times X$,
all solutions to $F(\lambda,Z)=0$ consist of
the trivial solution branch and
the local curve
$$
\{\,(\lambda,Z)=(\lambda(s), 
s(\varPhi_{1}+\zeta (s)))\in\mathbb{R}\times X\,:\,
s\in (-\delta, \delta )\,\}$$
with some $\delta >0$, where
$\zeta (s)$ is an $X$-valued function
of class $C^{1}$
satisfying 
$\int_{\Omega}\zeta (s)\varPhi_{1}=0$
for all $s\in (-\delta, \delta )$ and $\zeta (0)=0$.
Then one can conclude that the subset with $s$ restricted on 
$s\in (0, \delta )$ 
forms the bifurcation curve of positive solutions from
$(\lambda,Z)=(\lambda_{1}, 0)$.
Together with Lemma \ref{noexlem2}, we know that
$\lambda (s)>\lambda_{1}$ for any $s\in (0,\delta )$.
The proof of Lemma \ref{biflem} is complete.
\end{proof}
Our aim is to extend $\varGamma_{\delta}$ 
(obtained in Lemma \ref{biflem}) to the range
$\lambda >\lambda_{1}$ 
as a simple curve of positive solutions of \eqref{subell}.

\begin{proof}[Proof of Proposition \ref{subprop}]
By virtue of the embedding $X\subset C^{1}(\overline{\Omega })$
for $p>N$, we introduce the following positive cone
$$P=\{\,Z\in X\,:\,Z>0\mbox{ in }\Omega\mbox{ and }
\partial_{\nu}Z <0\mbox{ on }\partial\Omega\,\}.$$
In view of the expression of $\varGamma_{\delta}$ obtained in 
Lemma \ref{biflem},
we may assume $\varGamma_{\delta}\subset (\lambda_{1}, \infty )\times P$.
Let $\varGamma$ be the maximal extension of 
$\varGamma_{\delta}$ in $(\lambda_{1}, \infty )\times P$
as the connected set of
positive solutions of \eqref{subell}.
By Lemma \ref{noexlem2}, we see that
the projection of $\varGamma$ on the $\lambda$-axis,
which is denoted by $\mbox{proj}\,\varGamma$, 
is contained in $(\lambda_{1},\infty)$.
Then we define $\overline{\lambda}\in (\lambda_{1}, \infty]$
by
$$\overline{\lambda}:=\sup\{\,\lambda\,:\,
(\lambda_{1}, \lambda )\subset\mbox{proj}\,\varGamma\,\}.$$
The following contradiction argument will show 
$\overline{\lambda}=\infty$.
Suppose $\overline{\lambda}\in (\lambda_{1}, \infty)$ by conradiction.
There exists a sequence 
$\{(\lambda_{n}, Z_{n})\}\subset\varGamma$ such that
$\lambda_{n}\nearrow\overline{\lambda}$.
It follows from Lemma \ref{aprlem2} that
$\|Z_{n}\|_{W^{2,p}}\le M_{8}$ for all $n\in\mathbb{N}$.
By the Sobolev embedding theorem, we can find a subsequence,
which is relabeled by $\{(\lambda_{n}, Z_{n})\}$ again,
such that
\begin{equation}\label{lim}
\lim_{n\to\infty}Z_{n}=\overline{Z}
\ \ \mbox{weakly in}\ X\ \mbox{and strongly in}\  
C^{1}(\overline{\Omega })
\end{equation}
with some nonnegative solution $\overline{Z}\in X$ to \eqref{subell}
with $\lambda=\overline{\lambda}$.
 
Assume the possibility that
$\varGamma$ attains a boundary point of $P$
at $(\overline{\lambda},\overline{Z})$, that is,
there exists $x^{*}\in\Omega$ such that $\overline{Z}(x^{*})=0$
or
there exists $x_{*}\in\partial\Omega$ such that
$\partial_{\nu}\overline{Z}(x_{*})=0$.
By the strong maximum principle or the Hopf boundary lemma,
one can see that
$\overline{Z}=0$ in $\Omega$.
It follows that $(\overline{\lambda}, 0)$ is a bifurcation point
of positive solutions to \eqref{subell}
from the trivial solution.
As in the proof of Lemma \ref{biflem}, 
one must deduce $\overline{\lambda}=\lambda_{1}$.
However, this contradicts the fact that $\overline{\lambda}>\lambda_{1}$.

Then, under the assumption $\overline{\lambda}\in (\lambda_{1}, \infty)$,
we know that the limit $\overline{Z}$ in \eqref{lim}
is a positive solution to \eqref{subell} with $\lambda=\overline{\lambda}$.
We now observe that
$$
-\Delta\overline{Z}
=\dfrac{\overline{\lambda}m(x)}{2}(\sqrt{4\overline{Z}+1}-1)
=\overline{\lambda}\dfrac{2m(x)}{\sqrt{4\overline{Z}+1}+1}
\overline{Z}\ \ \mbox{in}\ \Omega,
\quad \overline{Z}=0\ \ \mbox{on}\ \partial\Omega.
$$
By virtue of the Krein-Rutman theorem,
one can see
\begin{equation}\label{bar1}
\overline{\lambda}=
\lambda_{1}\biggl(\dfrac{2m}{\sqrt{4\overline{Z}+1}+1}\biggr),
\end{equation}
where
$\lambda_{1}(h)$
represents the least eigenvalue
of \eqref{ev} with $m$ replaced by 
$h$.
Here we note
the variational characterization of $\lambda_{1}(h)$
(e.g., \cite{CC, Ni})
as follows:
\begin{equation}\label{var}
\lambda_{1}(h)=\inf_{\phi\in H^{1}_{0}\setminus\{0\}}
\dfrac{\|\nabla \phi\|^{2}_{2}}{\int_{\Omega}h(x)\phi^{2}}.
\end{equation}
Therefore, we know from \eqref{bar1} that
\begin{equation}\label{var2}
\begin{split}
\overline{\lambda}=
\lambda_{1}\biggl(\dfrac{2m}{\sqrt{4\overline{Z}+1}+1}\biggr)
&=
\inf_{\phi\in H^{1}_{0}\setminus\{0\}}
\dfrac{\|\nabla \phi\|^{2}_{2}}
{\int_{\Omega }2m(x)\phi^{2}/(\sqrt{4\overline{Z}+1}+1)}\\
&<
\inf_{\phi\in H^{1}_{0}\setminus\{0\}}
\dfrac{\|\nabla \phi\|^{2}_{2}}
{\int_{\Omega }m(x)\phi^{2}/\sqrt{4\overline{Z}+1}}
=
\lambda_{1}\biggl(
\dfrac{m}{\sqrt{4\overline{Z}+1}}\biggr),
\end{split}
\end{equation}

Using this fact,
we aim to show that
the linearized operator 
$F_{Z}(\overline{\lambda}, \overline{Z})$ is an isomorphism
from $X$ to $Y$.
For this end, we consider
the equation 
$F_{Z}(\lambda, \overline{Z})\phi =0$, which is reduced to 
\begin{equation}\label{Feq}
-\Delta\phi=\lambda\dfrac{m(x)}{\sqrt{4\overline{Z}+1}}\phi
\ \ \mbox{in}\ \Omega,\quad\phi=0\ \ \mbox{on}\ \partial\Omega
\end{equation}
by \eqref{Fz}.
Hence \eqref{Feq} admits nontrivial solutions if and only if
\begin{equation}\label{Fz2}
\lambda=\lambda_{j}\biggl(
\dfrac{m}{\sqrt{4\overline{Z}+1}}\biggr)
\ \ \mbox{for some}\ j\in\mathbb{N}.
\end{equation}
It follows from \eqref{var2} and \eqref{Fz2} that
\eqref{Feq} with $\lambda=\overline{\lambda}$ 
cannot admit any nontrivial solution.
Hence we deduce that
$F_{Z}(\overline{\lambda}, \overline{Z})$
is an isomorphism from $X$ to $Y$.

Therefore, we can use the implicit function theorem to
see that, in a neighborhood of 
$(\overline{d},\overline{Z})\in
\mathbb{R}\times X$, all solutions of $F(\lambda,Z)=0$ form a curve
$$\{ (\lambda,Z(\lambda))\in (\overline{\lambda}-\delta_{1}, 
\overline{\lambda}+\delta_{1})\times X\,\},$$
where $(\overline{\lambda}-\delta_{1}, 
\overline{\lambda}+\delta_{1})\ni \lambda\mapsto Z(\lambda )\in X$ 
is a function of class $C^{1}$.
This fact implies that $\varGamma$ can be extended as the set of
positive solutions of $F(\lambda,Z)=0$ 
in $\lambda\in [\overline{\lambda},\overline{\lambda}+\delta_{1})$.
Obviously, this contradicts the definition of $\overline{\lambda}$.
Consequently, the contradiction argument enables us to
conclude $\overline{\lambda}=\infty$, namely,
$\varGamma$ can be extended as the set of positive solutions of \eqref{subell}
over $\lambda\in (\lambda_{1},\infty)$.

Next we take $(\lambda^{*}, Z^{*})\in\varGamma$ arbitrarily.
Then repeating the above argument, 
one can see that $F_{Z}(\lambda^{*}, Z^{*})\in\mathcal{L}(X,Y)$
is an isomorphism.
The implicit function theorem ensures that, 
in a neighborhood of $(\lambda^{*}, Z^{*})\in\mathbb{R}\times X$,
all positive solutions of $F(\lambda,Z)=0$ form a curve
parameterized as
$\{\,(\lambda, Z(\lambda))\in 
(\lambda^{*}-\delta^{*}, \lambda^{*}+\delta^{*})\times X\,\}$
with some $\delta^{*}>0$.
Therefore, by the usual patchwork procedure,
we know that
$\varGamma$ forms a curve 
\eqref{curve}
parameterized by $\lambda\in (\lambda_{1},\infty)$.
 
We need to show that 
all positive solutions of \eqref{subell}
are exhausted by the bifurcation curve $\varGamma$.
To to so, it suffices to show the uniqueness
of positive solutions of \eqref{subell}
for each $\lambda\in (\lambda_{1}, \infty)$.
In view of the right-hand side of \eqref{subell},
we set
$$
h_{1}(Z):=\sqrt{4Z+1}-1.
$$
Hence \eqref{subell} is represented as
$$-\Delta Z=\dfrac{\lambda m(x)}{2}h_{1}(Z)\quad\mbox{in}\ \Omega,\quad
Z=0\quad\mbox{on}\ \partial\Omega.
$$
We observe that, for any fixed $\lambda\in (\lambda_{1}, \infty)$,
$$
\dfrac{h_{1}(Z)}{Z}=\dfrac{4}{\sqrt{4Z+1}+1}
$$
is monotone decreasing with respect to $Z>0$.
Then, thanks to the result by Br\'ezis and Osward
\cite[Theorem 1]{BO},
there is at most one positive solution of \eqref{subell}
for each $\lambda\in (\lambda_{1}, \infty)$.
Consequently, we deduce that
all positive solutions of \eqref{subell} form the unique simple curve
expressed as 
\eqref{curve}.

In order to show \eqref{sin}, 
we set $\zeta(\lambda ):=\lambda^{-2}Z(\lambda)$ 
for $Z(\lambda)\in\varGamma$.
It is easy to see that
\begin{equation}\label{zetaeq}
-\Delta \zeta=\dfrac{m(x)}{2}(\sqrt{4\zeta+\lambda^{-2}}-
\lambda^{-1})=
\dfrac{2  m(x)}{\sqrt{4\zeta+\lambda^{-2}}+\lambda^{-1}}\zeta
\ \ \mbox{in}\ \Omega,\quad
\zeta=0\ \ \mbox{on}\ \partial\Omega.
\end{equation}
It follows from Lemma \ref{aprlem2} that
$\|\zeta (\lambda )\|_{W^{2,p}}\le \lambda^{-2}M_{8}=:\widetilde{M}_{8}$
for all $\lambda\in (\lambda_{1},\infty)$.
Furthermore, by virtue of the proof of Lemma \ref{aprlem2},
it is possible to verify that
$\widetilde{M}_{8}$ is independent of 
$\lambda\in (\lambda_{1},\infty )$.
By the Sobolev embedding theorem,
there exist a sequence $\{\mu_{n}\}$
with $\lim_{n\to\infty}\mu_{n}=\infty$ and 
$\zeta_{0}\in X$ such that $\zeta_{n}:=\zeta (\mu_{n})$ satisfy
$$
\lim_{n\to\infty}\zeta_{n}=\zeta_{0}
\ \ \mbox{weakly in $X$ 
and strongly in $C^{1}(\overline{\Omega })$.}
$$
Since $\zeta_{n}$ is a positive solution of \eqref{zetaeq}
with $\lambda=\mu_{n}$ for each $n\in\mathbb{N}$, 
then \eqref{var} implies
\begin{equation}\label{var3}
1=\lambda_{1}\biggl(
\dfrac{2m}{\sqrt{4\zeta_{n}+\mu_{n}^{-2}}+\mu_{n}^{-1}}\biggr)
=\inf_{\phi\in H^{1}_{0}\setminus\{0\}}
\dfrac{\|\nabla \phi\|^{2}_{2}}
{\int_{\Omega }2m(x)\phi^{2}/(\sqrt{4\zeta_{n}+\mu_{n}^{-2}}+\mu_{n}^{-1})}.
\end{equation}
Suppose that $\zeta_{0}=0$ in $\Omega$.
Then, by the variational characterization in \eqref{var3},
one can see that 
$\lambda_{1}(2m/(\sqrt{4\zeta_{n}+\mu_{n}^{-2}}+\mu_{n}^{-1}))\to 0$
as $n\to\infty$.
This obviously contradicts the first equality in \eqref{var3}.
Then we see that $\zeta_{0}\ge (\not\equiv ) \,0$ satisfies
$-\Delta\zeta_{0}=m(x)\sqrt{\zeta_{0}}$ in $\Omega$ and
$\zeta_{0}=0$ on $\partial\Omega$, and therefore,
the strong maximum principle implies that $\zeta_{0}>0$ in $\Omega$. 
Thanks to the uniqueness of positive solutions of the sublinear
elliptic problem (e.g., \cite{BO}), one can verify that 
the family $\{\zeta (\lambda )\}_{\lambda>\lambda_{1}}$
itself
tends to $\zeta_{0}$ 
as $\lambda\to\infty$.  
The proof of Proposition \ref{subprop} is accomplished. 
\end{proof}
By the change of variables \eqref{Zdef},
Proposition \ref{subprop}
gives the global bifurcation curve $\mathcal{C}_{\infty}$ of 
positive solutions of \eqref{LS1} stated in Theorem \ref{thm13}:

\begin{proof}[Proof of Theorem \ref{thm13}]
All assertions immediately follow from Proposition \ref{subprop}
and \eqref{Zdef}.
\end{proof}

\subsection{Perturbation 
for the limiting system of small coexistence}
In this subsection,
we construct a subset of positive solutions of \eqref{SKT}
as the perturbation of $\mathcal{C}_{\infty}$ when $\alpha$
is sufficiently large.
By the changes of variables
\begin{equation}\label{uvUV}
(u,v)=\varepsilon (U,V)\quad\mbox{with}\quad\varepsilon=
\dfrac{1}{\alpha}
\end{equation}
and 
\begin{equation}\label{WZdef}
W=U-V,\quad
Z=(1+V)U
\end{equation}
in \eqref{SKT},
we know that $(W,Z)$ satisfies
the following system of semilinear elliptic equations:
\begin{equation}\label{WZeq}
\begin{cases}
\Delta W +\lambda m(x)W-\varepsilon \{\,
U(b_{1}U+c_{1}V)-V(b_{2}U+c_{2}V)\,\}=0
\ \ &\mbox{in}\ \Omega,\\
\Delta Z+\lambda m(x)U-\varepsilon U(b_{1}U+c_{1}V)=0,
\quad Z\ge 0
\ \ &\mbox{in}\ \Omega,\\
W=Z=0\ \ &\mbox{on}\ \partial\Omega,
\end{cases}
\end{equation}
where $(U,V)$ is considered as the
function of $(W,Z)$ defined by \eqref{WZdef} such as
\begin{equation}\label{UVdef}
(U(W,Z),V(W,Z))=
\dfrac{1}{2}(
\sqrt{(1-W)^{2}+4Z}-1+W,
\sqrt{(1-W)^{2}+4Z}-1-W).
\end{equation}
Thanks to the one-to-one correspondence \eqref{uvUV}-\eqref{WZdef}
between $(u,v)$ and $(W,Z)$, 
we have a policy of proving Theorem \ref{Cathm} 
by constructing a subset of solutions to \eqref{WZeq}
as the perturbation of the bifurcation curve \eqref{curve} 
of solutions to
\eqref{subell} obtained in Proposition \ref{subprop}.
\begin{prop}\label{perprop}
For any large $\varLambda>0$,
there exists a small $\underline{\delta}=\underline{\delta}(\varLambda )>0$ 
such that,
if $|\varepsilon |<\underline{\delta}$, then 
there exists a subset of solutions to \eqref{WZeq},
which consists of a bifurcation curve
$\widetilde{\mathcal{C}}_{\varepsilon, \varLambda }
(\subset (\lambda_{1}, \varLambda\,]\times\boldsymbol{X})$ paramterized by 
$\lambda\in (\lambda_{1}, \varLambda\,]$ as follows:
$$\widetilde{\mathcal{C}}_{\varepsilon, \varLambda }
=\{\,(\lambda ,W(\lambda,\varepsilon ), Z(\lambda,\varepsilon ))\in 
(\lambda_{1}, \varLambda\,]\times\boldsymbol{X}\,:\,|\varepsilon|<\,
\underline{\delta}\,\},
$$
where
$$
(\lambda_{1},\varLambda\,]\times (-\underline{\delta},\underline{\delta})\ni
(\lambda,\varepsilon )\mapsto
(W(\lambda,\varepsilon ), Z(\lambda,\varepsilon ))
\in\boldsymbol{X}
$$
is of class $C^{1}$ satisfying
$$
Z(\lambda,\varepsilon )>0\quad\mbox{in}\  \Omega
\quad\mbox{for any}\ 
(\lambda,\varepsilon )\in (\lambda_{1}, \varLambda\,]
\times
(-\underline{\delta}, \underline{\delta})$$
and
\begin{equation}\label{1stbifpt}
\lim_{\lambda\searrow\lambda_{1}}(W(\lambda,\varepsilon ), 
Z(\lambda,\varepsilon ))=(0,0)
\quad\mbox{in}\ \boldsymbol{X}\ \ \mbox{for any} 
\ \varepsilon\in (-\underline{\delta}, \underline{\delta})
\end{equation}
and
$$(W(\lambda, 0), Z(\lambda, 0))=(0,Z_{0}(\lambda ))\ \ 
\mbox{for any}\ \lambda\in (\lambda_{1},\varLambda\,]$$
with the solution $Z_{0}(\lambda )$ of \eqref{subell} 
obtained in Proposition \ref{subprop}.
Furthermore,
$\widetilde{\mathcal{C}}_{\varepsilon, \varLambda}$ can be extended 
to the range $\lambda\in (\varLambda, \infty)$
as a
connected subset of solutions of \eqref{WZeq} with $Z>0$ in $\Omega$.
\end{prop}

\begin{proof}
We define a nonlinear operator 
$G\,:\, \mathbb{R}_{+}\times \boldsymbol{X}\times \mathbb{R}\to \boldsymbol{Y}$
associated with \eqref{WZeq}
by
\begin{equation}\label{Gdef}
G(\lambda,W,Z,\varepsilon )=-
\left[
\begin{array}{l}
\Delta W +\lambda m(x)W\\
\Delta Z+\lambda m(x)U
\end{array}
\right]
+\varepsilon \left[
\begin{array}{l}
U(b_{1}U+c_{1}V)-V(b_{2}U+c_{2}V)\\
U(b_{1}U+c_{1}V)
\end{array}
\right],
\end{equation}
where $(U,V)$ is defined by \eqref{UVdef}.
In view of Proposition \ref{subprop},
one can see that
$$
G(\lambda, 0, Z_{0}(\lambda ), 0)=
\biggl[
\begin{array}{l}
0\\
F(\lambda,Z_{0}(\lambda ))
\end{array}
\biggr]
=
\biggl[
\begin{array}{l}
0\\
0
\end{array}
\biggr]
\ \ \mbox{for any}\ \lambda\in (\lambda_{1},\infty ),
$$
where $F$ is the operator defined by \eqref{Fdef}.
Then 
the curve 
\begin{equation}\label{Z0}
\{\,(\lambda,W,Z)\in (\lambda_{1},\infty)\times \boldsymbol{X}\,:\,
W=0,\ Z=Z_{0}(\lambda )\ \}
\end{equation}
gives a subset of solutions of $G(\lambda,W,Z,0)=0$.

Our aim is to construct a subset of solutions of
$G(\lambda,W,Z,\varepsilon )=0$ with small $|\varepsilon |$
as a perturbation of the curve expressed as \eqref{Z0}.
For an application of the implicit function theorem,
we differentiate $G$ by $(W,Z)$ as follows:
\begin{equation}\label{Gwz}
\begin{split}
&G_{(W,Z)}(\lambda,W,Z,\varepsilon )\left[
\begin{array}{c}
\phi\\
\psi
\end{array}
\right]\\
=&-
\biggl[
\begin{array}{l}
\Delta\phi+\lambda m(x)\phi\\
\Delta\psi+\lambda m(x)(U_{W}\phi+U_{Z}\psi)
\end{array}
\biggl]\\
&+\varepsilon 
\biggl[
\begin{array}{ll}
2b_{1}U+(c_{1}-b_{2})V & (c_{1}-b_{2})U-2c_{2}V\\
2b_{1}U+c_{1}V & c_{1}U
\end{array}
\biggr]
\biggl[
\begin{array}{cc}
U_{W} & U_{Z}\\
V_{W} & V_{Z}
\end{array}
\biggr]
\biggl[
\begin{array}{c}
\phi\\
\psi
\end{array}
\biggr],
\end{split}
\end{equation}
where
\begin{equation}\label{Uwz}
\begin{split}
&U_{W}=U_{W}(W,Z)=
\dfrac{1}{2}\biggl(
1-\dfrac{1-W}{\sqrt{(1-W)^{2}+4Z}}\biggr),\quad
U_{Z}=U_{Z}(W,Z)=
\dfrac{1}{\sqrt{(1-W)^{2}+4Z}},\\
&V_{W}=V_{W}(W,Z)=
-\dfrac{1}{2}\biggl(
1+\dfrac{1-W}{\sqrt{(1-W)^{2}+4Z}}\biggr),\quad
V_{Z}=V_{Z}(W,Z)=
\dfrac{1}{\sqrt{(1-W)^{2}+4Z}}.\\
\end{split}
\end{equation}
In view of \eqref{curve},
we next consider whether 
$G_{(W,Z)}(\lambda,0,Z_{0}(\lambda ),0)$ is an isomorphism
from $\boldsymbol{X}$ to $\boldsymbol{Y}$ or not.
It follows from \eqref{Fz} and \eqref{Gwz} that
\begin{equation}
G_{(W,Z)}(\lambda,0,Z_{0}(\lambda ),0 )\left[
\begin{array}{c}
\phi\\
\psi
\end{array}
\right]
=
\biggl[
\begin{array}{l}
-\Delta\phi-\lambda m(x)\phi\\
F_{Z}(\lambda,Z_{0}(\lambda))\psi-\lambda m(x)U_{W}(0,Z_{0}(\lambda))\phi
\end{array}
\biggl]
\quad\mbox{for}\ \lambda\in (\lambda_{1},\infty).
\nonumber
\end{equation}
Then,
to get $\mbox{Ker}\,G_{(W,Z)}(\lambda, 0, Z_{0}(\lambda ), 0)$,
it suffices to solve the following linear system
\begin{equation}\label{Gwz0}
\begin{cases}
-\Delta\phi =\lambda m(x)\phi \quad &\mbox{in}\ \Omega,\\
F_{Z}(\lambda, Z_{0}(\lambda ))\psi =
\dfrac{\lambda m(x)}{2}\biggl(1-\dfrac{1}{\sqrt{4Z_{0}(\lambda )+1}}\biggr)\phi
\quad &\mbox{in}\ \Omega,\\
\phi =\psi =0
\quad &\mbox{on}\ \partial\Omega.
\end{cases}
\end{equation}
In view of the argument around \eqref{Fz2},
we recall that,
if $\lambda\in (\lambda_{1},\infty)$,
then
$F_{Z}(\lambda,Z_{0}(\lambda ))$ is an isomorphism from $X$ to $Y$,
and therefore, the inverse $F_{Z}(\lambda, Z_{0}(\lambda ))^{-1}$
exists as a bounded linear operator from $Y$ to $X$.
If $\lambda\neq \lambda_{j}$ for any $j\in\mathbb{N}$,
then the first equation of \eqref{Gwz0} leads to $\phi=0$,
and then, the second equation gives
$\psi = F_{Z}(\lambda, Z_{0}(\lambda ))^{-1}\,0 =0$.
On the other hand,
if $\lambda=\lambda_{j}$ with some $j\ge 2$,
then the first equation of \eqref{Gwz0}
gives $\phi=s\varPhi_{j}$
for any $L^{\infty}$ normalized eigenfunction of
\eqref{ev} with $\lambda=\lambda_{j}$
and $s\in\mathbb{R}$, and thereby,
the second equation leads to
$
\psi=s\varPsi_{j}$, where
$$
\varPsi_{j} :=
F_{Z}(\lambda_{j}, Z_{0}(\lambda_{j}))^{-1}\biggl[\,
\dfrac{\lambda_{j} m(x)}{2}
\biggl(1-\dfrac{1}{\sqrt{4Z_{0}(\lambda_{j})+1}}\biggr)\varPhi_{j}\,\biggr].
$$
Let 
$E_{j}$ be the eigenspace of \eqref{ev} with $\lambda=\lambda_{j}$, that is,
$$
E_{j}:=\{\,\phi\in X\,:\,
\mbox{$\phi$ satisfies \eqref{ev} with $\lambda=\lambda_{j}$}\,\}.
$$
Consequently, we know that,
for any $\lambda\in (\lambda_{1}, \infty)$,
\begin{equation}\label{ker}
\mbox{Ker}\,G_{(W,Z)}(\lambda,0, Z_{0}(\lambda ), 0)=
\begin{cases}
0\quad &\mbox{if}\ \lambda\neq \lambda_{j},\\
E_{j}\times H_{j}
\quad &\mbox{if}\ \lambda=\lambda_{j},
\end{cases}
\end{equation}
where
$$
H_{j}:=F_{Z}(\lambda_{j}, Z_{0}(\lambda_{j}))^{-1}
\biggl[\,
\dfrac{\lambda_{j}m(x)}{2}
\biggl(1-\dfrac{1}{\sqrt{4Z_{0}(\lambda_{j})+1}}\biggr)
E_{j}\,\biggr].
$$

Take any $\lambda^{*}\in (\lambda_{1}, \infty)\setminus
\cup^{\infty}_{j=2}\{\lambda_{j}\}$.
Then \eqref{ker} enables us to use the implicit function theorem
to find a neighborhood $\mathcal{U}^{*}(\lambda^{*})$ 
of $(\lambda,W,Z,\varepsilon  )=(\lambda^{*}, 0, Z_{0}(\lambda^{*}), 0)
\in\mathbb{R}\times\boldsymbol{X}\times\mathbb{R}$
and a small $\delta^{*}=\delta^{*}(\lambda^{*}) >0$ such that
\begin{equation}\label{cover1}
\begin{split}
&\{\,(\lambda, W,Z,\varepsilon )\in\mathcal{U}^{*}(\lambda^{*})\,:\,
G(\lambda, W,Z,\varepsilon )=0\,\}\\
=&
\{\,(\lambda, W,Z,\varepsilon )\,:\,
(W,Z)=(W(\lambda,\varepsilon ), Z(\lambda, \varepsilon )),\
(\lambda,\varepsilon )\in (\lambda^{*}-\delta^{*}, 
\lambda^{*}+\delta^{*})\times
(-\delta^{*}, \delta^{*})\,\}\,(=:\varGamma(\lambda^{*}))
\end{split}
\end{equation}
with some function 
$
(\lambda^{*}-\delta^{*}, \lambda^{*}+\delta^{*})\times
(-\delta^{*}, \delta^{*})\ni
(\lambda,\varepsilon )\mapsto
(W(\lambda,\varepsilon ), Z(\lambda,\varepsilon ))\in\boldsymbol{X}
$
of class $C^{1}$ satisfying
\begin{equation}\label{per0}
(W(\lambda,0), Z(\lambda, 0))=(0, Z_{0}(\lambda ))
\quad \mbox{for any}\ 
\lambda\in (\lambda^{*}-\delta^{*}, \lambda^{*}+\delta^{*}).
\end{equation}

Assume $\lambda=\lambda_{j}$ with some $j\ge 2$.
Since
$G_{(W,Z)}(\lambda_{j},0,Z_{0}(\lambda_{j}), 0)$ is not
an isomorphism from $\boldsymbol{X}$ to $\boldsymbol{Y}$
by \eqref{ker},
we restrict the domain of $G_{(W,Z)}(\lambda_{j},0,Z_{0}(\lambda_{j}), 0)$
such that $G_{(W,Z)}(\lambda_{j},0,Z_{0}(\lambda_{j}), 0)$ becomes an isomorphism.
To so so, we introduce the closed subsets
$Y_{j}\subset Y$
and
$X_{j}\subset X$
as
$Y_{j}:=\{\,\phi\in Y\,:\,\int_{\Omega}\phi\varPhi_{j}=0\ 
\mbox{for any}\ \varPhi_{j}\in E_{j}\}$
and
$X_{j}:=X\cap Y_{j}$.
Hence it follows that
$L_{j}:=G_{(W,Z)}(\lambda_{j},0,Z_{0}(\lambda_{j}), 0)|_{X_{j}\times X}$
becomes
an isomorphism from $X_{j}\times X$ to $Y_{j}\times Y$.
Then we denote by $G^{j}$ the nonlinear operator $G$
with the domain restricted to $\mathbb{R}_{+}\times X_{j}\times X\times
\mathbb{R}$ such as
\begin{equation}\label{Gjdef}
\mathbb{R}_{+}\times X_{j}\times X\times\mathbb{R}\ni
(\lambda, W,Z,\varepsilon )\mapsto G^{j}(\lambda, W,Z,\varepsilon ):=
G(\lambda, W,Z,\varepsilon )
\in \boldsymbol{Y}.
\end{equation}
Together with the fact that
$G^{j}(\lambda_{j}, 0, Z_{0}(\lambda_{j}), 0)=0$,
we can use the implicit function theorem
for
$
G^{j}\,:\,\mathbb{R}_{+}\times X_{j}\times X\times\mathbb{R}\to \boldsymbol{Y}
$
to obtain a neighborhood $\mathcal{U}^{j}(\lambda_{j})$ of 
$(\lambda_{j}, 0, Z_{0}(\lambda_{j}), 0)$ in 
$\mathbb{R}\times X_{j}\times X\times\mathbb{R}$
and a small $\delta_{j}>0$ such that
\begin{equation}\label{cover2}
\begin{split}
&\{\,(\lambda, W,Z,\varepsilon )\in\mathcal{U}^{j}(\lambda_{j})\,:\,
G^{j}(\lambda, W,Z,\varepsilon )=0\,\}\\
=&
\{\,(\lambda, W,Z,\varepsilon )\,:\,
(W,Z)=(W(\lambda,\varepsilon ), Z(\lambda, \varepsilon )),\
(\lambda,\varepsilon )\in (\lambda_{j}-\delta_{j}, \lambda_{j}+\delta_{j})\times
(-\delta_{j}, \delta_{j})\,\}\,(=:\varGamma_{j})
\end{split}
\end{equation}
with some function 
$$
(\lambda_{j}-\delta_{j}, \lambda_{j}+\delta_{j})\times
(-\delta_{j}, \delta_{j})\ni
(\lambda,\varepsilon )\mapsto
(W(\lambda,\varepsilon ), Z(\lambda,\varepsilon ))\in X_{j}\times X
$$
of class $C^{1}$ satisfying
$(W(\lambda,0), Z(\lambda, 0))=(0, Z_{0}(\lambda ))$
for any $\lambda\in (\lambda_{j}-\delta_{j}, \lambda_{j}+\delta_{j})$.

Assume $\lambda=\lambda_{1}$.
By \eqref{UVdef} and \eqref{Gdef}, we see from \eqref{Gjdef} that
$G^{1}(\lambda_{1},0,0,\varepsilon )=0$
for any $\varepsilon \in\mathbb{R}$.
It follows from \eqref{Gwz} and \eqref{Uwz} that
$G^{1}_{(W,Z)}(\lambda_{1},0,0,\varepsilon )\in 
\mathcal{L}(X_{1}\times X, \boldsymbol{Y})$ is expressed as
$$
G^{1}_{(W,Z)}(\lambda_{1},0,0,\varepsilon )
\biggl[
\begin{array}{c}
\phi\\
\psi
\end{array}
\biggr]
=
-\biggl[
\begin{array}{l}
\Delta\phi +\lambda_{1}m(x)\phi\\
\Delta\psi +\lambda_{1}m(x)\psi
\end{array}
\biggr].
$$
Therefore, for the linear operator
$L_{1}(\varepsilon ):=G^{1}_{(W,Z)}(\lambda_{1},0,0,\varepsilon )
\in\mathcal{L}(X_{1}\times X, \boldsymbol{Y})$,
we see that
$$
\mbox{Ker}\,L_{1}(\varepsilon )=\mbox{span}\biggl\{
\biggl[
\begin{array}{c}
0\\
\varPhi_{1}
\end{array}
\biggr]
\biggr\},
$$
where $\varPhi_{1}$ is the eigenfunction defined by \eqref{egnor}.
Differentiating \eqref{Gwz} by $\lambda$ and setting
$(\lambda, W,Z)=(\lambda_{1}, 0, 0)$, we see 
$$
G^{1}_{(W,Z),\lambda}(\lambda_{1},0,0,\varepsilon )
\biggl[
\begin{array}{c}
\phi\\
\psi
\end{array}
\biggr]
=-m(x)
\biggl[
\begin{array}{c}
\phi\\
\psi
\end{array}
\biggr].
$$
For the use of the bifurcation theorem
from simple eigenvalues
\cite[Theorem 1.7]{CR},
we need to check the transversality condition
\begin{equation}\label{tra2}
G^{1}_{(W,Z),\lambda}(\lambda_{1},0,0,\varepsilon )
\biggl[
\begin{array}{c}
0\\
\varPhi_{1}
\end{array}
\biggr]
=-m(x)
\biggl[
\begin{array}{c}
0\\
\varPhi_{1}
\end{array}
\biggr]
\not\in\mbox{Ran}\,
L_{1}(\varepsilon ).
\end{equation}
Suppose by contradiction that there exists
$\psi\in X$ such that
$$
-\Delta\psi-\lambda_{1}m(x)\psi
=-m(x)\varPhi_{1}
\ \ \mbox{in}\ \Omega,\quad
\psi=0\ \ \mbox{on}\ \partial\Omega.
$$
Multiplying the differential equation by $\varPhi_{1}$ and 
integrating the resulting expression over $\Omega$, we have
$$\displaystyle\int_{\Omega}
m(x)\varPhi_{1}^{2}=0,$$
which is impossible.
This contradiction enables us to conclude that
\eqref{tra2} holds true.
Then by the bifurcation theorem \cite[Theorem 1.7]{CR},
for any fixed small $|\varepsilon |$,
there exist
a neighborhood $\mathcal{U}^{1}_{\varepsilon}(\lambda_{1})$ of
$(\lambda, W,Z)=(\lambda_{1}, 0, 0)$
in
$\mathbb{R}\times X_{1}\times X$,
a small $\delta_{1}>0$
and functions for $s\in (-\delta_{1}, \delta_{1})$
of class $C^{1}$ as follows:
$$
\begin{cases}
\widetilde{w}(\,\cdot\,,s)\in X_{1}\ \ \mbox{with}\ \ 
\widetilde{w}(\,\cdot\,,0)=0,\\
\widetilde{z}(\,\cdot\,,s)\in X\ \ \mbox{with}\ \ 
\widetilde{z}(\,\cdot\,,0)=0,\\
\mu (s)\in\mathbb{R}\ \ \mbox{with}\ \ 
\mu (0)=0
\end{cases}
$$
such that
\begin{equation}\label{bifloc}
\begin{split}
&\{\,(\lambda, W,Z)\in\mathcal{U}^{1}_{\varepsilon }(\lambda_{1})
\,:\,G^{1}(\lambda, W,Z,\varepsilon )=0\,\}\\
=&
\{\,(\lambda,0,0)\in\mathcal{U}^{1}_{\varepsilon }(\lambda_{1})\,\}\cup
\{\,(\lambda, W,Z)=(\lambda_{1}+\mu (s), s\widetilde{w}(\,\cdot\,,s), 
s(\varPhi_{1}+\widetilde{z}(\,\cdot\,,s)))\,:\,
s\in (-\delta_{1}, \delta_{1})\,\}.
\end{split}
\end{equation}
Here we may assume $\mu' (s)>0$
for any $s\in (0,\delta_{1})$ because there is no solution
of \eqref{WZeq} with $Z>0$ in $\Omega$ when $\lambda\le\lambda_{1}$.

For any large $\varLambda>0$, we take
$\lambda^{*}\in [\lambda_{1}+\mu (\delta_{1}/2), \varLambda ]$
such that $\lambda^{*}\neq\lambda_{j}$ for any $j\in\mathbb{N}$.
By virtue of \eqref{cover1}, we introduce a 
compact set
$\widetilde{\varGamma}(\lambda^{*})$ in 
$\mathbb{R}\times\boldsymbol{X}\times\mathbb{R}$ as
\begin{equation}
\begin{split}
\widetilde{\varGamma}(\lambda^{*}):=\{\,
&(\lambda, W,Z,\varepsilon )\in\mathbb{R}_{+}\times
\boldsymbol{X}\times\mathbb{R}\,:\,\\
&(W,Z)=(W(\lambda,\varepsilon ), Z(\lambda, \varepsilon )),\,
(\lambda,\varepsilon )\times [\lambda^{*}-\delta^{*}/2, \lambda^{*}+\delta^{*}/2]
\times [-\delta^{*}/2, \delta^{*}/2]\,\}.
\end{split}
\nonumber
\end{equation}
We set $j_{0}:=\max\{\,j\in\mathbb{N}\,:\,\lambda_{j}\le\varLambda\,\}$.
For each $j\in\{\,2,3,\ldots,j_{0}\,\}$,
we similarly introduce a compact set contained in $\varGamma_{j}$ 
(see \eqref{cover2}) as follows:
\begin{equation}
\begin{split}
\widetilde{\varGamma}_{j}:=\{\,
&(\lambda, W,Z,\varepsilon )\in\mathbb{R}_{+}\times
\boldsymbol{X}\times\mathbb{R}\,:\,\\
&(W,Z)=(W(\lambda,\varepsilon ), Z(\lambda, \varepsilon )),\,
(\lambda,\varepsilon )\times [\lambda_{j}-\delta_{j}/2, 
\lambda_{j}+\delta_{j}/2]
\times [-\delta_{j}/2, \delta_{j}/2]\,\}.
\end{split}
\nonumber
\end{equation}
Then we observe that
$$
\left(\bigcup\limits_{\lambda^{*}\in [\lambda_{1}+
\mu(\delta_{1}/2),
\varLambda]
\setminus\cup^{j_{0}}_{j=2}\{\lambda_{j}\}}
\widetilde{\varGamma}(\lambda^{*})\cup
\bigcup^{j_{0}}_{j=2}\widetilde{\varGamma}_{j}\right)
\subset
\left(
\bigcup\limits_{\lambda^{*}\in (\lambda_{1}+
\mu(\delta_{1}/3), 
\varLambda+\eta)
\setminus\cup^{j_{0}}_{j=2}\{\lambda_{j}\}}
\varGamma(\lambda^{*})\cup
\bigcup^{j_{0}}_{j=2}\varGamma_{j}
\right)
$$
with some small $\eta>0$.
Hence,
and the right-hand side gives an open covering of
the compact set expressed as the left-hand side.
Then we can find a finite number $\{\ell_{k}\}^{k_{0}}_{k=1}$
such that
$$
\lambda_{1}+\mu\biggl(\dfrac{\delta_{1}}{3}\biggr)
<\ell_{1}<\ell_{2}<\cdots<\ell_{k_{0}}<
\varLambda+\eta
$$
and
\begin{equation}\label{patch}
\left(\bigcup\limits_{\lambda^{*}\in [\lambda_{1}+\mu(\delta_{1}/2),
\varLambda]
\setminus\cup^{j_{0}}_{j=2}\{\lambda_{j}\}}
\widetilde{\varGamma}(\lambda^{*})\cup
\bigcup^{j_{0}}_{j=2}\widetilde{\varGamma}_{j}\right)
\subset
\left(
\bigcup^{k_{0}}_{k=1}\varGamma(\ell_{k})
\cup
\bigcup^{j_{0}}_{j=2}\varGamma_{j}
\right).
\end{equation}
We set
$$
\underline{\delta}:=\min_{1\le j\le j_{0},\,
1\le k\le k_{0}}\bigg\{\,
\dfrac{\delta_{j}}{2},\dfrac{\delta^{*}(\ell_{k})}{2}\,\biggr\}>0.
$$
By virtue of \eqref{patch},
for each $\varepsilon\in (-\underline{\delta},\underline{\delta})$,
a usual patchwork procedure produces a simple curve
of solutions to $G(\lambda, W,Z,\varepsilon )=0$ as
$$\{\,
(\lambda, W(\lambda,\varepsilon ), 
Z(\lambda, \varepsilon ), \varepsilon )\,:\,
\lambda\in [\lambda_{1}+\mu(\delta_{1}/2), \varLambda ]\,\}.$$
We know from \eqref{cover1} that,
for each $\varepsilon\in (-\underline{\delta}, \underline{\delta })$,
there exists a neighborhood $\mathcal{U}^{*}_{\varepsilon }
(\lambda_{1}+\delta_{1}/2)$ of
$$
(\lambda, W,Z)=\biggl(\lambda_{1}+\mu\biggl(\dfrac{\delta_{1}}{2}\biggr),
W\biggl(\lambda_{1}+\mu\biggl(\dfrac{\delta_{1}}{2}\biggr), \varepsilon\biggr),
Z\biggl(\lambda_{1}+\mu\biggl(\dfrac{\delta_{1}}{2}\biggr),
\varepsilon\biggr)\biggr)\in\mathbb{R}\times\boldsymbol{X},$$
such that
all solutions to $G(\lambda, W,Z,\varepsilon )=0$ 
in $\mathcal{U}^{*}_{\varepsilon }(\lambda_{1}+\delta_{1}/2)$
form a simple curve
\begin{equation}\label{above}
(\lambda, W(\lambda, \varepsilon ), Z(\lambda, \varepsilon ))
\end{equation}
for $\lambda$ within a neighborhood of $\lambda_{1}+\mu(\delta_{1}/2)$.
On the other hand, we remark that the local bifurcation curve obtained in 
\eqref{bifloc} enters the neighborhood
$\mathcal{U}^{*}_{\varepsilon}(\lambda_{1}+\delta_{1}/2)$, and 
therefore,
the local bifurcation piece must meet the curve \eqref{above}.
Consequently, we know that
\begin{equation}\label{perbra}
\widetilde{\mathcal{C}}_{\varepsilon, \varLambda}
=\{\,(\lambda, W(\lambda,\varepsilon ), Z(\lambda, \varepsilon ))\,:\,
\lambda\in (\lambda_{1}, \varLambda\,]\,\}
\end{equation}
forms a simple curve of the subset of solutions of
$G(\lambda, W,Z,\varepsilon )=0$ if 
$\varepsilon \in (-\underline{\delta},\underline{\delta})$.

Furthremore, by virtue of the global bifurcation theorem by Rabinowitz \cite{Ra},
we know that $\widetilde{\mathcal{C}}_{\varepsilon, \varLambda}$ 
can be extended as 
a connected subset of solutions of 
$G(\lambda, W,Z,\varepsilon )=0$ to $\lambda\in (\varLambda, \infty )$ 
or reaches the 
trivial solution $(W,Z)=(0,0)$ at some $\lambda\ge\lambda_{1}$.
Considering that the bifurcation point from the trivial solution 
is limited to $\lambda=\lambda_{1}$ and that 
$\widetilde{\mathcal{C}}_{\varepsilon, \varLambda}$
actually bifurcates from the trivial solution at $\lambda=\lambda_{1}$, 
we can be sure that the latter possibility is ruled out.
Therefore, we can conclude that the former occurs.
Furthermore, we see that the positivity of the $Z$ 
component on the branch is preserved by the maximum principle.
Then the proof of Proposition \ref{perprop} is complete.
\end{proof}

\begin{proof}[Proof of Theorem \ref{Cathm}]
By translating the result of Proposition \ref{perprop} 
into the properties of the positive solution $(u,v)$ of \eqref{SKT} 
through the variable transformation \eqref{uvUV} and
\eqref{UVdef}, we obtain Theorem 1.4.
\end{proof}

\section{Bifurcation branches of the limit equation 
homeomorphic to eigenspaces}
This section will show that \eqref{WZeq} has solutions that do not
belong to $\widetilde{\mathcal{C}}_{\varepsilon, \varLambda }$
even when $\varepsilon >0$ is sufficiently small.
More precisely, in addition to 
$\widetilde{\mathcal{C}}_{\varepsilon, \varLambda}$,
we construct other branches of solutions of 
\eqref{WZeq} with
$(\lambda, W)$ is near $(\lambda_{j}, s\varPhi_{j})$
for sufficiently small $|s|>0$ and $\varepsilon >0$.
\subsection{Limiting case $\varepsilon =0$}
To do so,
we begin with the limiting system of \eqref{WZeq}
with $\varepsilon= 0$ as follows
\begin{equation}\label{WZeq0}
\begin{cases}
\Delta W+\lambda m(x)W=0\quad&\mbox{in}\ \Omega,\\
\Delta Z+\dfrac{\lambda m(x)}{2}
(\sqrt{(1-W)^{2}+4Z}-1+W)=0,\ \ Z\ge 0\quad&
\mbox{in}\ \Omega,\\
W=Z=0\quad&\mbox{on}\ \partial\Omega.
\end{cases}
\end{equation}
By the first equation,
if $\lambda\neq \lambda_{j}$,
then $W=0$, and moreover,
the set of positive solutions to 
the second equation with $W=0$ forms the
simple curve \eqref{curve}
parameterized by $\lambda\in (\lambda_{1}, \infty)$
obtained in Proposition \ref{subprop}.

On the other hand,
if $\lambda=\lambda_{j}$ with some $j\ge 2$,
all solutions of the first equation of \eqref{WZeq0}
can be expressed as
$W=s\varPhi_{j}$
for $s\in\mathbb{R}$,
where
$\varPhi_{j}\in E_{j}$ denotes 
any $L^{\infty}(\Omega)$ normalized 
eigenfunction of \eqref{ev}
with $\lambda=\lambda_{j}$, that is,
\begin{equation}\label{Phij}
-\Delta\varPhi_{j}=\lambda_{j}m(x)\varPhi_{j}\ \ \mbox{in}\ \Omega,
\quad\varPhi_{j}=0\ \ \mbox{on}\ \partial\Omega,\quad
\|\varPhi_{j}\|_{\infty}=1.
\end{equation}
Then in this case,
\eqref{WZeq0} is reduced to
the single elliptic equation:
\begin{equation}\label{sZeq}
\begin{cases}
\Delta Z+\dfrac{\lambda_{j}m(x)}{2}
h(Z,1-s\varPhi_{j})
=0,
\quad Z\ge 0\quad &\mbox{in}\ \Omega,\\
Z=0\quad &\mbox{on}\ \partial\Omega.
\end{cases}
\end{equation}
where $h(Z, \xi)$ is defined by
$h(Z, \xi):=
\sqrt{4Z+\xi^{2}}-\xi$.
\begin{lem}\label{eigenlem}
Let $j\in\mathbb{N}$ be $j\ge 2$.
For any fixed $s\in\mathbb{R}$ and $\varPhi_{j}\in E_{j}$,
there exists  
a unique positive solution $Z=Z_{j}(s)\in X$ of \eqref{sZeq}.
Furthermore, $\mathbb{R}\ni s\mapsto Z_{j}(s)\in X$
is of class $C^{1}$ with
$Z_{j}(0)=Z_{0}(\lambda_{j})$,
where $Z_{0}(\lambda_{j})$ is the positive solution of \eqref{subell}
with $\lambda=\lambda_{j}$ obtained in Proposition \ref{subprop}.
\end{lem}

\begin{proof}
We first observe that 
$$
\dfrac{h(Z,\xi)}{Z}=
\frac{\sqrt{4Z+\xi^{2}}-\xi}{Z}=\dfrac{4}{\sqrt{4Z+\xi^{2}}+\xi}
$$
is monotone decreasing with respect to $Z>0$.
Then it is well known that
\eqref{sZeq} admits at most one positive solution
for each $\varPhi_{j}\in E_{j}$ and $s\in\mathbb{R}$
(e.g., \cite{BO}).

In order to construct the set of solutions of \eqref{sZeq},
we define the operator
$F^{(j)}(s,Z)\,:\,\mathbb{R}\times X\to Y$ by
$$
F^{(j)}(s,Z):=-\Delta Z-\dfrac{\lambda_{j}m(x)}{2}h(Z,1-s\varPhi_{j}).
$$
It follows from Proposition \ref{subprop} that
$F^{(j)}(0, Z_{0}(\lambda_{j}))=0$
for each $j\in\mathbb{N}$ with $j\ge 2$.
Then we can see that the Fr\'echet derivative
$L_{j}(s):=F^{(j)}_{Z}(s, Z_{0}(\lambda_{j}))\in\mathcal{L}(X,Y)$ 
is represented as
\begin{equation}\label{Ljdef}
L_{j}(s)\phi=
-\Delta\phi-\dfrac{\lambda_{j}m(x)}{\sqrt{4Z_{0}(\lambda_{j})+(1-s\varPhi_{j})^{2}}}\phi.
\end{equation}
Then setting $s=0$,
we see that
\begin{equation}\label{Lj0def}
L_{j}(0)\phi=
-\Delta\phi-\dfrac{\lambda_{j}m(x)}{\sqrt{4Z_{0}(\lambda_{j})+1}}\phi=F_{Z}(\lambda_{j},Z_{0}(\lambda_{j}))\phi,
\end{equation}
where $F_{Z}(\lambda,Z)$ is defined by \eqref{Fz}.
By repeating the argument around \eqref{Feq} with
$(\overline{\lambda},\overline{Z})$ replaced by 
$(\lambda_{j}, Z_{0}(\lambda_{j}))$,
one can verify that
$L_{j}(0)$ is an isomorphism from $X$ to $Y$.
Then we can use the implicit function theorem 
for $F^{(j)}(s,Z)$ 
to find a neighborhood $\mathcal{U}_{j}$ of
$(s,Z)=(0,Z_{0}(\lambda_{j}))\in\mathbb{R}\times X$
and a small $\delta_{j}>0$ such that
$$\{\,(s,Z)\in\mathcal{U}_{j}\,:\,F^{(j)}(s,Z)=0\,\}
=\{\,(s,Z)\,:\,Z=Z_{j}(s),\ s\in (-\delta_{j},\delta_{j})\,\}
$$
with some continuously differentiable
function $(-\delta_{j}, \delta_{j})\ni s\mapsto
Z_{j}(s)\in X$.

Next we show that the local curve $Z_{j}(s)$ can be 
extended as a global curve defined for $s\in\mathbb{R}$
as the set of positive solutions of $F^{(j)}(s,Z)=0$.
To do so, we begin with the a priori estimate
of any positive solutions of \eqref{sZeq}.
Let $Z(s)$ be any positive solution of \eqref{sZeq}.
Then by a straight forward calculation,
one can find a positive constant $\widetilde{M}=
\widetilde{M}(j, \|m\|_{\infty})$,
which is independent of $s\in\mathbb{R}$, such that
$$
-\Delta Z(s)-\dfrac{\lambda_{1}}{2}Z(s)
=\dfrac{1}{2}\left(
\lambda_{j}m(x)h(Z(s), 1-s\varPhi_{j})-\lambda_{1}Z(s)\right)
\le
\widetilde{M}(|s|+1)
$$
for any $s\in\mathbb{R}$.
Then the elliptic regularity ensures a positive constant
$M=M(j, \|m\|_{\infty}, \Omega )>0$, 
which is independent of $s\in\mathbb{R}$, such that
$\|Z(s)\|_{X}\le M(|s|+1)$ for any $s\in\mathbb{R}$.
Here we define 
$\overline{s}\in [\delta_{j},\infty]$ by
the supremum of $\hat{s}$ such that
$Z_{j}(s)$ can be extended as a continuously differentiable
functions on $(-\delta_{j},\overline{s})$
forming a subset of
positive solutions of $F^{(j)}(s, Z)=0$.
Our aim is to show $\overline{s}=\infty$.
Suppose by contradiction that $\overline{s}<\infty$.
Then the a priori estimate and the strong maximum principle
ensure that $Z_{j}(\overline{s}):=\lim_{s\nearrow\overline{s}}Z_{j}(s)$ 
is a positive solution of \eqref{sZeq}
with $s=\overline{s}$. It follows that
\begin{equation}\label{Fj0}
F^{(j)}(\overline{s}, Z_{j}(\overline{s}))=0.
\end{equation}
Then it is easy to check that
\begin{equation}\label{Zjposi}
\begin{cases}
-\Delta Z_{j}(\overline{s})=
\lambda_{j}\dfrac{2m(x)}{\sqrt{4Z_{j}(\overline{s})+(1-\overline{s}\varPhi_{j})^{2}}+1
-\overline{s}\varPhi_{j}}Z_{j}(\overline{s})
\quad&\mbox{in}\ \Omega,\\
Z_{j}(\overline{s})=0\quad&\mbox{on}\ \partial\Omega.
\end{cases}
\end{equation}
Let $\lambda_{1}(q)$ be the least eigenvalue of \eqref{ev}
with $m(x)$ replaced by $q(x)$.
Observing that
$$
\dfrac{2m(x)}{\sqrt{4Z_{j}
(\overline{s})+(1-\overline{s}\varPhi_{j})^{2}}+1
-\overline{s}\varPhi_{j}}>
\dfrac{m(x)}{\sqrt{4Z_{j}(\overline{s})+(1-\overline{s}\varPhi_{j})^{2}}}
\quad\mbox{in}\ \Omega
$$
and
$\lambda_{1}(q)$ 
is monotone decreasing with respect to $q$,
we know from \eqref{Zjposi} that 
\begin{equation}\label{muj}
\lambda_{j}=\lambda_{1}\biggl(
\dfrac{2m}{\sqrt{4Z_{j}(\overline{s})+(1-\overline{s}\varPhi_{j})^{2}}+1
-\overline{s}\varPhi_{j}}\biggr)<
\lambda_{1}\biggl(
\dfrac{m}{\sqrt{4Z_{j}(\overline{s})+(1-\overline{s}\varPhi_{j})^{2}}}
\biggr).
\end{equation}
In view of \eqref{Ljdef}, we know that
\eqref{muj} ensures the invertibility of
$L_{j}(\overline{s})$, that is,
$L_{j}(\overline{s})$ is an isomorphism from $X$ to $Y$.
Together with \eqref{Fj0}, the application 
of the implicit function theorem
to $F^{(j)}(s, Z)$ near  
$(\overline{s}, Z_{j}(\overline{s}))$ 
enables us to see that $Z=Z_{j}(s)$ can be extended as 
a continuously differentiable zero-level curve of
$F^{(j)}(s,Z)$ for $s\in (-\delta_{j}, \overline{s}+\kappa )$
with some $\kappa >0$.
However, this contradicts the definition of $\overline{s}$.
Hence the contradiction argument leads to $\overline{s}=\infty$.
By a similar argument, we can reach a fact that
$Z_{j}(s)$ can be extended globally for $s\in \mathbb{R}$
as the set of positive solutions of \eqref{sZeq}.
The proof of Lemma \ref{eigenlem} is complete.
\end{proof}

In addition to Lemma \ref{eigenlem}, by using the sub-super
solution method,
we can get the following estimate
of $Z_{j}(s)$ in a special case where $|s|$ is sufficiently small.
\begin{lem}\label{subsuperlem}
Let $Z_{j}(s)$ be the positive solution of \eqref{sZeq} 
obtained in Lemma \ref{eigenlem}.
If $|s|>0$ is sufficiently small,
then
$$
(1+|s|)^{2}Z_{0}\biggl(
\dfrac{\lambda_{j}}{1+|s|}\biggr)<
Z_{j}(s)<(1-|s|)^{2}Z_{0}\biggl(
\dfrac{\lambda_{j}}{1-|s|}\biggr)
\quad\mbox{in}\ \Omega.
$$
\end{lem}

\begin{proof}
By the definition of $h(Z,\xi)$, we see that
$$
h_{\xi}(Z,\xi)=
\dfrac{\xi}{\sqrt{4Z+\xi^{2}}}-1\le
\dfrac{\xi}{|\xi|}-1\le 0\ \ \mbox{for any}\ 
Z\ge 0\ \ \mbox{and}\ \ \xi> 0.$$
It follows that, for each fixed $Z>0$,
$h(Z,\xi)$ is monotone decreasing with respect to
$\xi >0$.
Since $1-|s|\le 1-s\varPhi_{j}(x)\le 1+|s|$ 
for any $x\in\overline\Omega$ and $s\in\mathbb{R}$,
then
\begin{equation}\label{horder}
h(Z,1+|s|)\le h(Z, 1-s\varPhi_{j})\le h(Z,1-|s|)
\ \ \mbox{in}\ \Omega
\end{equation}
for any $Z\ge 0$ and $|s|<1$.

We define 
$$\overline{Z}_{j}(s):=
(1-|s|)^{2}Z_{0}\biggl(
\dfrac{\lambda_{j}}{1-|s|}\biggr)
$$
for $|s|<1$
and the solution 
$Z_{0}(\lambda_{j}/(1-|s|))$
of \eqref{subell} with $\lambda=\lambda_{j}/(1-|s|)$.
Then it is easily verified that
$$
-\Delta\overline{Z}_{j}(s)=
\dfrac{\lambda_{j}m(x)}{2}h(\overline{Z}_{j}(s), 1-|s|)\ \ 
\mbox{in}\ \Omega,\quad
\overline{Z}_{j}(s)=0\ \ \mbox{on}\ \partial\Omega.
$$
Together with \eqref{horder}, we see that
$$
-\Delta\overline{Z}_{j}(s)\ge
\dfrac{\lambda_{j}m(x)}{2}h(\overline{Z}_{j}(s), 1-s\varPhi_{j})\ \ 
\mbox{in}\ \Omega,\quad
\overline{Z}_{j}(s)=0\ \ \mbox{on}\ \partial\Omega,
$$
that is,
$\overline{Z}_{j}(s)$ is a super-solution
of \eqref{sZeq}.

We next define 
$$\underline{Z}_{j}(s):=
(1+|s|)^{2}Z_{0}\biggl(
\dfrac{\lambda_{j}}{1+|s|}\biggr)$$
for $s$ such that $|s|<1$ and $\lambda_{j}/(1+|s|)>\lambda_{1}$.
Then it can be checked that
$$
-\Delta\underline{Z}_{j}(s)=
\dfrac{\lambda_{j}m(x)}{2}h(\underline{Z}_{j}(s), 1+|s|)\ \ 
\mbox{in}\ \Omega,\quad
\underline{Z}_{j}(s)=0\ \ \mbox{on}\ \partial\Omega.
$$
Then \eqref{horder} implies that
$$
-\Delta\underline{Z}_{j}(s)\le
\dfrac{\lambda_{j}m(x)}{2}h(\underline{Z}_{j}(s), 1-s\varPhi_{j})\ \ 
\mbox{in}\ \Omega,\quad
\underline{Z}_{j}(s)=0\ \ \mbox{on}\ \partial\Omega,
$$
namely,
$\underline{Z}_{j}(s)$ is a sub-solution of \eqref{sZeq}.

In order to show 
$\underline{Z}_{j}(s)<\overline{Z}_{j}(s)$
in 
$\Omega$
near $s=0$,
we introduce the function $\widehat{Z}_{0}(\xi)$ as
$$
\widehat{Z}_{0}(\xi)=(1+\xi)^{2}Z_{0}
\biggl(\dfrac{\lambda_{j}}{1+\xi}\biggr)
\ \ \mbox{for}\ \ 
\xi\in\biggl(-1,-1+\dfrac{\lambda_{j}}{\lambda_{1}}\biggr).
$$
Hence it follows that
\begin{equation}\label{subsuper}
\underline{Z}_{j}(s)=\widehat{Z}_{0}(|s|),\quad
\overline{Z}_{j}(s)=\widehat{Z}_{0}(-|s|).
\end{equation}
Differentiating $\widehat{Z}_{0}(\xi)$ by $\xi$, we obtain
\begin{equation}\label{deri}
\dfrac{d\widehat{Z}_{0}}{d\xi}\bigg|_{\xi=0}
=
2Z_{0}(\lambda_{j})-\lambda_{j}Z_{0}'(\lambda_{j})
\quad\mbox{in}\ \Omega,
\end{equation}
where the prime symbol represents the derivative by $\lambda$.
Differentiating \eqref{subell} by $\lambda$, 
one can see that
\begin{equation}\label{Z0d}
L_{j}(0)Z_{0}'(\lambda_{j})=
\dfrac{m(x)}{2}\left(
\sqrt{4Z_{0}(\lambda_{j})+1}-1\right),
\end{equation}
where $L_{j}(0)\in\mathcal{L}(X,Y)$ is defined by \eqref{Lj0def}.
In the same manner as the argument getting \eqref{muj},
one can verify 
\begin{equation}\label{muj2}
\lambda_{j}<\lambda_{1}\biggl(
\dfrac{m}{\sqrt{4Z_{0}(\lambda_{j})+1}}\biggr).
\end{equation}
By virtue of the elliptic regularity and
the maximum principle,
\eqref{muj2} implies that
$L_{j}(0)$ is invertible and 
$L_{j}(0)^{-1}$ is monotone in the sense that
if $f_{1}< f_{2}$ in $\Omega$, then
$L_{j}(0)^{-1}f_{1}< L_{j}(0)^{-1}f_{2}$ in $\Omega$.
Therefore,
we know from \eqref{Z0d} that
\begin{equation}\label{pri}
Z_{0}'(\lambda_{j})=
L_{j}(0)^{-1}
\biggl[
\dfrac{m(x)}{2}(\sqrt{4Z_{0}(\lambda_{j})+1}-1)\biggr].
\end{equation}
We observe from \eqref{subell} that
$$
-\Delta Z_{0}(\lambda )-\dfrac{\lambda m(x)}{\sqrt{4Z_{0}(\lambda )+1}}Z_{0}(\lambda )=
\dfrac{\lambda m(x)}{2}(\sqrt{4Z_{0}(\lambda )+1}-1)-
\dfrac{\lambda m(x)}{\sqrt{4Z_{0}(\lambda )+1}}Z_{0}(\lambda ),
$$
and thereby,
\begin{equation}\label{inv2}
Z_{0}(\lambda_{j})=\lambda_{j}L_{j}(0)^{-1}
\biggl[
\dfrac{m(x)}{2}(\sqrt{4Z_{0}(\lambda_{j})+1}-1)
-\dfrac{m(x)}{\sqrt{4Z_{0}(\lambda_{j})+1}}Z_{0}(\lambda_{j})\biggr].
\end{equation}
Substituting \eqref{pri} and \eqref{inv2} into \eqref{deri}, we obtain
\begin{equation}\label{deri2}
\dfrac{d\widehat{Z}_{0}}{d\xi}\bigg|_{\xi=0}
=
L_{j}(0)^{-1}
\biggl[
\dfrac{m(x)}{2}(\sqrt{4Z_{0}(\lambda_{j})+1}-1)
-\dfrac{2m(x)}{\sqrt{4Z_{0}(\lambda_{j})+1}}Z_{0}(\lambda_{j})\biggr].
\end{equation}
By taking account for
\begin{equation}\label{nega}
\begin{split}
&\dfrac{m(x)}{2}(\sqrt{4Z_{0}(\lambda_{j})+1}-1)
-\dfrac{2m(x)}{\sqrt{4Z_{0}(\lambda_{j})+1}}Z_{0}(\lambda_{j})\\
=&
2m(x)Z_{0}(\lambda_{j})\biggl(
\dfrac{1}{\sqrt{4Z_{0}(\lambda_{j})+1}+1}-
\dfrac{1}{\sqrt{4Z_{0}(\lambda_{j})+1}}\biggr)<0
\ \ \mbox{in}\ \Omega
\end{split}
\end{equation}
and the monotone property of $L_{j}(0)^{-1}$ induced by
\eqref{muj2},
we know from \eqref{deri2} and \eqref{nega} that
$$\dfrac{d\widehat{Z}_{0}}{d\xi}\bigg|_{\xi=0}
<0\ \ \mbox{in}\ \Omega.$$
Together with \eqref{subsuper}, one deduces that
$$
\underline{Z}_{j}(s)<\overline{Z}_{j}(s)
\ \ \mbox{in}\ \Omega$$
if $|s|>0$ is sufficiently small.
Consequently, the usual sub-super solution method
ensures a solution $Z_{j}(s)$ of \eqref{sZeq} such that
$$
(\,0<\,)\,\underline{Z}_{j}(s)<Z_{j}(s)<\overline{Z}_{j}(s)
\ \ \mbox{in}\ \Omega$$
if $|s|$ is sufficiently small.
The proof of Lemma \ref{subsuperlem} is accomplished.
\end{proof}
Combining Proposition \ref{subprop}, Lemmas \ref{noexlem2},
\ref{eigenlem} and
\ref{subsuperlem},
we obtain the following result on 
solutions of \eqref{WZeq0} with $Z>0$ in $\Omega$:
\begin{cor}\label{limcor}
If $\lambda\in (0, \lambda_{1}]$, then \eqref{WZeq0} does not
admit any nontrivial solution.
If $\lambda\in (\lambda_{1}, \infty)
\setminus\cup^{\infty}_{j=2}\{\lambda_{j}\}$,
then there exists a unique solution $(W,Z)=(0, Z_{0}(\lambda ))$ of
\eqref{WZeq0} with $Z>0$ in $\Omega$.
If $\lambda=\lambda_{j}$ with $j\ge 2$,
then all solutions of \eqref{WZeq0} with $Z>0$ can be represented as
$$(W,Z)=(s\varPhi_{j}, Z_{j}(s))\quad
\mbox{for any}\ s\in\mathbb{R},$$
where $\varPhi_{j}\in E_{j}$ is any function satisfying \eqref{Phij}.
Here $\mathbb{R}\ni s\mapsto Z_{j}(s)\in X$ is a continuously differentiable 
function,
depending on $\varPhi_{j}$, such that 
$Z_{j}(0)=Z_{0}(\lambda_{j})$ and
$$
(1+|s|)^{2}Z_{0}\biggl(
\dfrac{\lambda_{j}}{1+|s|}\biggr)<
Z_{j}(s)<(1-|s|)^{2}Z_{0}\biggl(
\dfrac{\lambda_{j}}{1-|s|}\biggr)
\quad\mbox{in}\ \Omega
$$
if $|s|>0$ is sufficiently small.
\end{cor}
By the transformation \eqref{UVdef}, 
the limit equation \eqref{WZeq0} with $\varepsilon =0$ for $(W,Z)$ 
can be reduced
to the following equation for $(U,V)$:
\begin{equation}\label{UVlimeq}
\begin{cases}
\Delta [\,(1+V)U\,]+\lambda m(x)U=0\ \ &\mbox{in}\ \Omega,\\
\Delta [\,(1+U)V\,]+\lambda m(x)V=0\ \ &\mbox{in}\ \Omega,\\
U=V=0\ \ &\mbox{on}\ \partial\Omega.
\end{cases}
\end{equation}
Hence,
by \eqref{UVdef},
the set of solutions of \eqref{WZeq0} obtained in Corollary \ref{limcor}
is converted to the set of positive solutions of \eqref{UVlimeq}:
\begin{cor}\label{UVlimcor}
If $\lambda\in (0,\lambda_{1}]$, then \eqref{UVlimeq} does not
admit any positive solution.
If $\lambda\in (\lambda_{1}, \infty)
\setminus\cup^{\infty}_{j=2}\{\lambda_{j}\}$,
then there exists a unique positive solution
$(U,V)=(U(\,\cdot\,,\lambda), U(\,\cdot\,,\lambda))$,
where $U(\,\cdot\,,\lambda)$ is the solution of \eqref{LS1} obtained in
Theorem \ref{thm13}.
If $\lambda=\lambda_{j}$ with $j\ge 2$,
then all positive solutions of \eqref{UVlimeq} can be
represented as
$$
(U,V)=\dfrac{1}{2}
\left(
\sqrt{(1-s\varPhi_{j})^{2}+4Z_{j}(s)}-1+s\varPhi_{j},
\sqrt{(1-s\varPhi_{j})^{2}+4Z_{j}(s)}-1-s\varPhi_{j}
\right)
$$
for any $s\in\mathbb{R}$ and $\varPhi_{j}\in E_{j}$,
where $Z_{j}(s)$ is the function in Corollary \ref{limcor}.
\end{cor}

\subsection{Perturbation to the case where $\varepsilon >0$ is small}
In this subsection, we construct solutions of \eqref{WZeq}
near $(\lambda, W,Z)=(\lambda_{j}, s\varPhi_{j}, Z_{j}(s))$ by regarding
$\varepsilon$ as a perturbation parameter.
\begin{prop}\label{sapprop}
Assume that $\dim E_{j}=1$
with some $j\ge 2$.
For any fixed $s^{*}\neq 0$, 
there exists $\delta^{*}_{j}>0$ such that
if $(s,\varepsilon ) \in 
(s^{*}-\delta_{j}^{*}, s^{*}+\delta_{j}^{*})
\times 
(-\delta_{j}^{*}, \delta_{j}^{*})$, 
then
\eqref{WZeq} admits a solution expressed as
$$(\lambda, W, Z)=(\lambda_{j}^{*}(s,\varepsilon), W_{j}^{*}(s,\varepsilon ), 
Z_{j}^{*}(s,\varepsilon ))
\quad\mbox{with}\quad
\|W_{j}^{*}(s,\varepsilon )\|_{2}=s\|\varPhi_{j}\|_{2},
$$
where the mapping
$$
(s^{*}-\delta^{*}_{j}, s^{*}+\delta^{*}_{j})
\times 
(-\delta^{*}_{j}, \delta^{*}_{j})
\ni (s, \varepsilon )\mapsto 
(\lambda_{j}^{*}(s, \varepsilon ), W_{j}^{*}(s, \varepsilon ), 
Z_{j}^{*}(s, \varepsilon ))
\in\mathbb{R}\times\boldsymbol{X}
$$
is of class $C^{1}$
with
$(\lambda_{j}^{*}(s, 0), W_{j}^{*}(s, 0), Z_{j}^{*}(s, 0))
=(\lambda_{j}, s\varPhi_{j}, Z_{j}(s))$.
\end{prop}

\begin{proof}
We define a nonlinear operator 
$\widetilde{G}\,:\,
\mathbb{R}_{+}\times\boldsymbol{X}\times\mathbb{R}^{2}\to
\boldsymbol{Y}\times\mathbb{R}$
by
\begin{equation}
\widetilde{G}(\lambda, W,Z,s, \varepsilon ):=
\biggl[
\begin{array}{l}
G(\lambda, W,Z,\varepsilon)\\
\|W\|^{2}_{2}-s^{2}\|\varPhi_{j}\|^{2}_{2}
\end{array}
\biggr],
\nonumber
\end{equation}
where $G\,:\,\mathbb{R}_{+}\times\boldsymbol{X}\times\mathbb{R}\to
\boldsymbol{Y}$ is the operator defined by \eqref{Gdef}.
It follows from Corollary \ref{limcor} that
\begin{equation}\label{Gtil0}
\widetilde{G}(\lambda_{j}, s\varPhi_{j}, Z_{j}(s), s, 0)=
\biggl[
\begin{array}{c}
0\\
0
\end{array}
\biggr]
\ \ \mbox{for any}\ s\in\mathbb{R}.
\end{equation}
It will be shown that,
for any fixed $s^{*}\neq 0$,
the Fr\'echet derivative
$$\widetilde{L}^{*}_{j}:=\widetilde{G}_{(\lambda, W,Z)}
(\lambda_{j}, s^{*}\varPhi_{j}, Z_{j}(s^{*}),s^{*}, 0)
$$
is an isomorphism from $\mathbb{R}\times\boldsymbol{X}$ to
$\boldsymbol{Y}\times \mathbb{R}$.
Setting $(\lambda, W, Z, s, \varepsilon )=
(\lambda_{j}, s^{*}\varPhi_{j}, Z_{j}(s^{*}), s^{*}, 0)$
in
\begin{equation}
\widetilde{G}_{(\lambda, W,Z)}(\lambda, W,Z,s,\varepsilon)
\left[
\begin{array}{c}
\mu\\
\phi\\
\psi
\end{array}
\right]
=
\left[
\begin{array}{l}
-\mu m(x)
\biggl[
\begin{array}{c}
W\\
U
\end{array}
\biggr]
+G_{(W,Z)}(\lambda, W,Z,\varepsilon )\left[
\begin{array}{c}
\phi\\
\psi
\end{array}
\right]\\
2\int_{\Omega}W\phi
\end{array}
\right]
\nonumber
\end{equation}
with $G_{(W,Z)}(\lambda, W,Z,\varepsilon )$ defined by \eqref{Gwz},
one can see
\begin{equation}
\begin{split}
&\widetilde{L}^{*}_{j}
\left[
\begin{array}{c}
\mu\\
\phi\\
\psi
\end{array}
\right]=
\widetilde{G}_{(\lambda, W,Z)}(\lambda_{j}, s^{*}\varPhi_{j}, Z_{j}(s^{*}),s^{*}, 0)
\left[
\begin{array}{c}
\mu\\
\phi\\
\psi
\end{array}
\right]\\
=&
\left[
\begin{array}{l}
-\mu m(x)
\biggl[
\begin{array}{l}
s^{*}\varPhi_{j}\\
U(s^{*}\varPhi_{j}, Z_{j}(s^{*}))
\end{array}
\biggr]
-
\biggl[
\begin{array}{l}
\Delta\phi +\lambda_{j}m(x)\phi\\
\Delta\psi +\lambda_{j}m(x)\{\,
U_{W}(s^{*}\varPhi_{j}, Z_{j}(s^{*}))\phi+
U_{Z}(s^{*}\varPhi_{j}, Z_{j}(s^{*}))\psi\,\}
\end{array}
\biggr]\\
2s^{*}\int_{\Omega}\varPhi_{j}\phi
\end{array}
\right].
\end{split}
\nonumber
\end{equation}
Then $(\mu, \phi,\psi)\in\mbox{Ker}\,\widetilde{L}^{*}_{j}$
is equivalent to 
\begin{equation}\label{Ltil}
\begin{cases}
-\Delta\phi=\lambda_{j}m(x)\phi+s^{*}\mu m(x)\varPhi_{j}
\quad&\mbox{in}\ \Omega,\\
-\mu m(x)U(s^{*}\varPhi_{j}, Z_{j}(s^{*}))
-\Delta\psi-
\lambda_{j}m(x)
\{\,U_{W}(s^{*}\varPhi_{j}, Z_{j}(s^{*}))\phi+
U_{Z}(s^{*}\varPhi_{j}, Z_{j}(s^{*}))\psi\,\}=0
\quad&\mbox{in}\ \Omega,\\
\int_{\Omega}\varPhi_{j}\phi =0,\\
\phi=\psi=0\quad&\mbox{on}\ \partial\Omega.
\end{cases}
\end{equation}
Multiplying the first equation of \eqref{Ltil} by $\varPhi_{j}$ 
and integrating the resulting expression, one can see
$s^{*}\mu\int_{\Omega}m(x)\varPhi_{j}^{2}=0$.
Together with $s^{*}\neq 0$, we obtain $\mu =0$.
Using the first equation again, we see $\phi\in E_{j}$.
It follows from the integral constraint that
$\phi\in E_{j}^{\bot}$ if $\dim E_{j}=1$.
Therefore, we obtain $\phi =0$.
Substituting $\mu =0$ and $\phi =0$ into 
the second equation of \eqref{Ltil},
we have
$$
\begin{cases}
\Delta\psi +
\dfrac{\lambda_{j}m(x)}{\sqrt{4Z_{j}(s^{*})+(1-s^{*}\varPhi_{j})^{2}}}\psi =0
\quad &\mbox{in}\ \Omega,\\
\psi=0\quad &\mbox{on}\ \partial\Omega,
\end{cases}
$$
that is,
$L^{*}_{j}\psi=0$
for the operator $L^{*}_{j}:=L_{j}(s^{*})$ defined by \eqref{Ljdef}.
By the argument showing \eqref{muj}, one can see that
$$
\lambda_{j}<
\lambda_{1}\biggl(
\dfrac{m}{\sqrt{4Z_{j}(s^{*})+(1-s^{*}\varPhi_{j})^{2}}}
\biggr).
$$
This fact ensures the invertibility of $L^{*}_{j}$.
Hence it follows that $\psi =0$.

Consequently, we obtain
$\mbox{Ker}\,\widetilde{L}^{*}_{j}=\{\,(\mu, \phi, \psi)=(0,0,0)\,\}$,
and thereby,
$\widetilde{L}^{*}_{j}$ is an isomorphism from
$\mathbb{R}\times\boldsymbol{X}$ to $\boldsymbol{Y}\times\mathbb{R}$
if $s^{*}\ne 0$ and $\mbox{dim}E_{j}=1$.
Together with \eqref{Gtil0},
the application of the implicit function theorem
for $\widetilde{G}(\lambda, W, Z, s, \varepsilon )$
with each fixed $s^{*}\neq 0$
gives a neighborhood $\mathcal{U}_{j}^{*}$
of $(\lambda_{j}, s^{*}\varPhi_{j}, Z_{j}(s^{*}),s^{*}, 0)\in
\mathbb{R}\times\boldsymbol{X}\times\mathbb{R}^{2}$,
a small $\delta_{j}^{*}>0$ and
continuously differentiable functions
$$(\lambda^{*}_{j}(s, \varepsilon ),
W^{*}_{j}(s, \varepsilon ),
Z^{*}_{j}(s, \varepsilon ))
\quad
\mbox{for}\ \  
(s,\varepsilon)\in (s^{*}-\delta^{*}_{j}, s^{*}+\delta^{*}_{j})\times
(-\delta^{*}_{j},\delta^{*}_{j})$$
such that
\begin{equation}
\begin{split}
&\{\,(\lambda, W, Z, s, \varepsilon )\in\mathcal{U}^{*}_{j}\,:\,
\widetilde{G}(\lambda, W,Z,s, \varepsilon )=0\,\}\\
=\,&\{\,(\lambda, W, Z)=(\lambda^{*}_{j}(s, \varepsilon ),
W^{*}_{j}(s, \varepsilon ),
Z^{*}_{j}(s, \varepsilon ))\,:\,
(s,\varepsilon)\in (s^{*}-\delta^{*}_{j}, s^{*}+\delta^{*}_{j})\times
(-\delta^{*}_{j},\delta^{*}_{j})\,\}
\end{split}
\nonumber
\end{equation}
and
$$(\lambda^{*}_{j}(s, 0 ),
W^{*}_{j}(s, 0 ),
Z^{*}_{j}(s, 0 ))=
(\lambda_{j}, s\varPhi_{j}, Z_{j}(s))
\ \ \mbox{for any}\ \  
s\in (s^{*}-\delta^{*}_{j}, s^{*}+\delta^{*}_{j}).
$$
This fact completes the proof of Proposition \ref{sapprop}.
\end{proof}
By \eqref{uvUV} and \eqref{UVdef},
for each $(s,\varepsilon )\in (s^{*}-\delta^{*}_{j}, 
s^{*}+\delta^{*}_{j})$,
the solution to \eqref{WZeq} with $\lambda=\lambda^{*}_{j}(s,\varepsilon )$
is corresponding to a positive solution to
\eqref{SKT} with $\lambda=\lambda^{*}_{j}(s, \varepsilon )$:
\begin{cor}
Suppose that $\dim E_{j}=1$ with some $j\ge 2$.
For any fixed $s^{*}\neq 0$,
there exists $\delta^{*}_{j}>0$ such that,
if $s\in
(s^{*}-\delta^{*}_{j}, s^{*}+\delta^{*}_{j})$
and $\alpha>0$ is sufficiently large,
then \eqref{SKT} admits a positive solution
\begin{equation}
\begin{split}
&(u,v)=\\
&\dfrac{1}{2\alpha}
\left(
\sqrt{(1-W^{*}_{j}(s,\varepsilon ))^{2}+4Z^{*}_{j}(s,\varepsilon )}
-1+W^{*}_{j}(s,\varepsilon ),
\sqrt{(1-W^{*}_{j}(s,\varepsilon ))^{2}+4Z^{*}_{j}(s, \varepsilon )}
-1-W^{*}_{j}(s, \varepsilon )\right)
\end{split}
\nonumber
\end{equation}
with $\lambda=\lambda^{*}_{j}(s,\varepsilon )$,
where
$(\lambda^{*}_{j}(s, \varepsilon ), W^{*}_{j}(s, \varepsilon ),
Z^{*}_{j}(s, \varepsilon ))$
is the solution of \eqref{WZeq} obtained in Proposition \ref{sapprop}.
\end{cor}

\section{Bifurcation from solutions of small coexistence}
In this section, 
we look for
bifurcation points on
the curve $\mathcal{C}_{\alpha, \varLambda}$ 
(obtained in Theorem \ref{Cathm})
of solutions to \eqref{SKT} with large $\alpha$.
By the change of variables \eqref{WZdef},
it is sufficient to find bifurcation points
on the curve $\widetilde{\mathcal{C}}_{\varepsilon, \varLambda}$ 
(obtained in \eqref{perbra}) of solutions to \eqref{WZeq}.

\begin{lem}\label{2ndbiflem}
Assume that $\dim E_{j}=1$ with some $j\ge 2$.
Let $G\,:\,\mathbb{R}_{+}\times
\boldsymbol{X}\times\mathbb{R}\to\boldsymbol{Y}$ 
be the nonlinear operator defined by \eqref{Gdef}.
There exists
a continuously differentiable function
$
(-\delta_{j}, \delta_{j})
\ni\varepsilon\mapsto
\mu_{j}(\varepsilon )\in\mathbb{R}$
with some small $\delta_{j}>0$
such that
\begin{equation}\label{mujtola}
\mu_{j}(0 )=\lambda_{j}
\end{equation}
and 
\begin{equation}\label{degper}
\mbox{{\rm dim\,Ker}}\,
G_{(W,Z)}(\mu_{j}(\varepsilon),
W(\mu_{j}(\varepsilon), \varepsilon),
Z(\mu_{j}(\varepsilon), \varepsilon), \varepsilon)= 1
\end{equation}
for any $(\mu_{j}(\varepsilon),
W(\mu_{j}(\varepsilon), \varepsilon),
Z(\mu_{j}(\varepsilon), \varepsilon))\in\widetilde{\mathcal{C}}_{\varepsilon, \varLambda}$ with $\varepsilon\in (-\delta_{j}, \delta_{j})$.
\end{lem}

For the proof of Lemma \ref{2ndbiflem},
we prepare the following lemma:
\begin{lem}\label{invlem}
For any large $\varLambda>0$,
there exists a small $\delta=\delta(\varLambda )>0$
such that,
if $|\varepsilon |<\delta$,
then for any $(\lambda, W(\lambda,\varepsilon ), Z(\lambda, \varepsilon ))\in 
\widetilde{\mathcal{C}}_{\varepsilon, \varLambda}$,
the following Dirichlet problem of the linear elliptic equation:
\begin{equation}\label{LD}
-\Delta\psi-\dfrac{\lambda m(x)}
{\sqrt{(1-W(\lambda, \varepsilon ))^{2}+4Z(\lambda, \varepsilon)}}\psi=0
\ \ \mbox{in}\ \Omega,\quad
\psi=0\ \ \mbox{on}\ \partial\Omega
\end{equation}
admits only the trivial solution.
\end{lem}

\begin{proof}
Let $(\lambda, Z_{0}(\lambda ))$ be the positive solution to \eqref{subell}
obtained in Proposition \ref{subprop}.
By the same manner as \eqref{var2}, we can see that
$$
\lambda<\lambda_{1}\biggl(
\dfrac{m}{\sqrt{1+4Z_{0}(\lambda )}}\biggr)
\ \ \mbox{for any}\  \lambda\in (\lambda_{1}, \infty).
$$
By virtue of \eqref{per0},
we note that
if $\varepsilon>0$ is sufficiently small,
then 
$(W(\lambda,\varepsilon), Z(\lambda, \varepsilon ))$
is close to 
$(0, Z_{0}(\lambda ))$
in $\boldsymbol{X}$.
By the continuity of $\lambda_{1}(q)$ 
with respect to $q$,
one can see that
$$
\lambda<\lambda_{1}\biggl(
\dfrac{m}{\sqrt{(1-W(\lambda,\varepsilon))^{2}
+4Z(\lambda, \varepsilon )}}\biggr)
$$
if $\varepsilon>0$ is sufficiently small.
This fact implies that
the linear operator
$$
-\Delta -\dfrac{\lambda m}{\sqrt{(1-W(\lambda,\varepsilon))^{2}
+4Z(\lambda, \varepsilon )}}\,:\,X\to Y
$$
is invertible.
Therefore, we see that
\eqref{LD} has only the trivial solution.
The proof of Lemma \ref{invlem} is accomplished.
\end{proof}

\begin{proof}[Proof of Lemma \ref{2ndbiflem}]
For any $(\lambda, W(\lambda,\varepsilon ), Z(\lambda, \varepsilon ))
\in\widetilde{\mathcal{C}}_{\varepsilon, \varLambda}$,
we define 
$L(\lambda,\varepsilon )\in\mathcal{L}(\boldsymbol{X}, \boldsymbol{Y})$
by the Fr\'echet derivative of $G$ by $(W,Z)$ at 
$(W(\lambda,\varepsilon ), Z(\lambda, \varepsilon ))$ as follows:
\begin{equation}\label{Ldedef}
L(\lambda,\varepsilon ):=G_{(W,Z)}(\lambda, W(\lambda, \varepsilon ), Z(\lambda, \varepsilon ),\varepsilon ).
\end{equation}
By virtue of \eqref{Gwz}, we see that
$(\phi, \psi)\in\mbox{Ker}\,L(\lambda,\varepsilon )$ is equivalent to
\begin{equation}\label{Lde}
\begin{cases}
-\Delta\phi=\lambda m(x)\phi-\varepsilon \{\,
h_{11}(\lambda,\varepsilon )\phi+h_{12}(\lambda, \varepsilon)\psi\,\}
\ \ &\mbox{in}\ \Omega,\\
-\Delta\psi-\lambda m(x)U_{Z}(\lambda,\varepsilon )\psi
+\varepsilon 
h_{22}(\lambda, \varepsilon )\psi=
\lambda m(x)U_{W}(\lambda, \varepsilon )\phi
-\varepsilon 
h_{21}(\lambda, \varepsilon )\phi
&\mbox{in}\ \Omega,\\
\phi=\psi=0\quad
&\mbox{on}\ \partial\Omega,
\end{cases}
\end{equation}
where
\begin{equation}\label{hijdef}
\begin{split}
&h_{11}(\lambda, \varepsilon ):=
\{\,2b_{1}U+(c_{1}-b_{2})V\,\}U_{W}+
\{\,(c_{1}-b_{2})U-2c_{2}V\,\}V_{W},\\
&h_{12}(\lambda, \varepsilon ):=
\{\,2b_{1}U+(c_{1}-b_{2})V\,\}U_{Z}+
\{\,(c_{1}-b_{2})U-2c_{2}V\,\}V_{Z},\\
&h_{21}(\lambda, \varepsilon ):=
(2b_{1}U+c_{1}V)U_{W}+c_{1}UV_{W},\\
&h_{22}(\lambda, \varepsilon ):=
(2b_{1}U+c_{1}V)U_{Z}+c_{1}UV_{Z}
\end{split}
\end{equation}
with
\begin{equation}\label{UVWZ}
\begin{split}
&(U,V)=(U(W(\lambda,\varepsilon ), Z(\lambda, \varepsilon )), 
V(W(\lambda, \varepsilon ), Z(\lambda, \varepsilon )))\\
=&\dfrac{1}{2}(
\sqrt{(1-W(\lambda,\varepsilon ))^{2}+4Z(\lambda, \varepsilon )}-1+W(\lambda,\varepsilon ),
\sqrt{(1-W(\lambda,\varepsilon ))^{2}+4Z(\lambda, \varepsilon )}-1-W(\lambda,\varepsilon )
)
\end{split}
\end{equation}
and
$(U_{W}, U_{Z}, V_{W}, V_{Z})$ defined by \eqref{Uwz} with
$(W,Z)=(W(\lambda,\varepsilon ), Z(\lambda, \varepsilon ))$.

We know from Lemma \ref{invlem} that, 
for each $\lambda\in (\lambda_{1}, \infty)$, 
if $|\varepsilon |$ is sufficiently small,
then the operator
\begin{equation}\label{Tdef}
T(\lambda,\varepsilon):=
-\Delta-\lambda m(x)U_{Z}(\lambda,\varepsilon )+\varepsilon h_{22}(\lambda,\varepsilon )
\in\mathcal{L}(X,Y)
\end{equation}
is invertible, so that
$T^{-1}(\lambda,\varepsilon )\in\mathcal{L}(Y,X)$ exists.
Then, the second equation of \eqref{Lde} can be expressed as
$$
\psi=T^{-1}(\lambda, \varepsilon )\{\,\lambda m(x)U_{W}(\lambda,\varepsilon )
-
\varepsilon h_{21}(\lambda, \varepsilon )\,\}\phi.
$$
Substituting this expression into the first equation of \eqref{Lde},
we know that \eqref{Lde} is reduced to
\begin{equation}\label{nonlocal}
\begin{cases}
-\Delta\phi=\lambda m(x)\phi-
\varepsilon h_{11}(\lambda, \varepsilon )\phi-
\varepsilon h_{12}(\lambda, \varepsilon )T^{-1}(\lambda,\varepsilon)
\{\,\lambda m(x)U_{W}(\lambda, \varepsilon )-\varepsilon 
h_{21}(\lambda, \varepsilon )\,\}
\phi\ \ &\mbox{in}\ \Omega,\\
\phi =0\ \ &\mbox{on}\ \partial\Omega.
\end{cases}
\end{equation}
Let $\widetilde{\varPhi}_{j}$ be any $L^{2}(\Omega )$ normalized eigenfunction
of \eqref{ev} with $\lambda=\lambda_{j}$.
Hence,
$\widetilde{\varPhi}_{j}\in E_{j}$ with $\|\widetilde{\varPhi}_{j}\|_{2}=1$.
In order to parameterize all solutions near 
$(\lambda, \phi,\varepsilon)=
(\lambda_{j}, \widetilde{\varPhi}_{j}, 0)$ to \eqref{nonlocal},
we define the operator
$H\,:\,\mathbb{R}_{+}\times X\times\mathbb{R}\to Y\times\mathbb{R}$
by
\begin{equation}
\begin{split}
&H(\lambda,\phi, \varepsilon )\\
=&
\biggl[
\begin{array}{l}
-\Delta\phi-\lambda m(x)\phi+
\varepsilon h_{11}(\lambda, \varepsilon )\phi+
\varepsilon h_{12}(\lambda, \varepsilon )T^{-1}(\lambda,\varepsilon)
\{\,\lambda m(x)U_{W}(\lambda, \varepsilon )-\varepsilon h_{21}(\lambda, \varepsilon )\,\}
\phi\\
\|\phi\|^{2}_{2}-1
\end{array}
\biggr].
\end{split}
\nonumber
\end{equation}
Hence it follows that
\begin{equation}\label{Hj0}
H(\lambda_{j}, \widetilde{\varPhi}_{j}, 0)=0.
\end{equation}
By a straightforward calculation, one can see that
$$
H_{(\lambda,\phi)}(\lambda_{j}, \widetilde{\varPhi}_{j}, 0)
\biggl[
\begin{array}{c}
\mu\\
\varphi
\end{array}
\biggr]
=
\biggl[
\begin{array}{l}
-\mu m(x)\widetilde{\varPhi}_{j}-\Delta\varphi-\lambda_{j}m(x)\varphi\\
2\int_{\Omega }\widetilde{\varPhi}_{j}\varphi
\end{array}
\biggr].
$$
Therefore, any $(\mu, \varphi )\in\mbox{Ker}\,
H_{(\lambda,\phi)}(\lambda_{j}, \widetilde{\varPhi}_{j}, 0)$
satisfies
$$
\begin{cases}
-\Delta\varphi =\lambda_{j}m(x)\varphi+\mu m(x)\widetilde{\varPhi}_{j}
\quad\mbox{in}\ \Omega,\\
\int_{\Omega}\widetilde{\varPhi}_{j}\varphi=0,\\
\phi=0\ \ \mbox{on}\ \partial\Omega.
\end{cases}
$$
Taking the $L^{2}$ inner product of the first equation with
$\widetilde{\varPhi}_{j}$, one can see 
$\mu\int_{\Omega}m(x)\widetilde{\varPhi}_{j}^{\,2}=0$,
which yields $\mu =0$.
Then setting $\mu=0$ in the first equation, we know that
$$
-\Delta\varphi=\lambda_{j}m(x)\varphi\ \ \mbox{in}\ \Omega,
\quad\varphi=0\ \ \mbox{on}\ \partial\Omega.
$$
Hence it follows from $\dim E_{j}=1$ that
$\varphi=s\widetilde{\varPhi}_{j}$ for any $s\in\mathbb{R}$.
Substituting this expression into the second equation,
we obtain $s=0$.
Consequently, we deduce that
$\mbox{Ker}\,H_{(\lambda,\phi)}(\lambda_{j},\widetilde{\varPhi}_{j},0)$
is trivial, and hence,
$H_{(\lambda,\phi)}(\lambda_{j}, \widetilde{\varPhi}_{j}, 0)$
is as isomorphism from $\mathbb{R}\times X$ to
$Y\times\mathbb{R}$.

Together with \eqref{Hj0},
we can apply the implicit function theorem 
for $H(\lambda,\phi, \varepsilon )$ to find
a neighborhood $\mathcal{U}$ of 
$(\lambda,\phi,\varepsilon )=(\lambda_{j},\widetilde{\varPhi}_{j}, 0)\in
\mathbb{R}\times X\times\mathbb{R}$,
a small $\delta_{j}>0$ and functions
$$
(-\delta_{j}, \delta_{j})\ni
\varepsilon\mapsto (\mu_{j}(\varepsilon),\phi_{j}(\varepsilon))
\in\mathbb{R}\times X
$$
of class $C^{1}$ such that
$$
\{\,(\lambda,\phi,\varepsilon )\in\mathcal{U}\,:\,
H(\lambda,\phi,\varepsilon )=0\,\}
=
\{\,(\lambda,\phi,\varepsilon )\,:\,
(\lambda,\phi )=(\mu_{j}(\varepsilon ), \phi_{j}(\varepsilon )),
\ \ \varepsilon\in (-\delta_{j},
\delta_{j})\,\}
$$
and
\begin{equation}\label{kerconti}
(\mu_{j}(0), \phi_{j}(0))=
(\lambda_{j}, \widetilde{\varPhi}_{j}).
\end{equation}
By the definition of $H(\lambda,\phi,\varepsilon )$,
we deduce that, 
if $|\varepsilon |$ and $|\lambda-\lambda_{j}|$ are sufficiently small,
then
\begin{equation}\label{kerL}
\mbox{Ker}\,
L(\lambda, \varepsilon )=
\begin{cases}
\{\,(\phi,\psi)=(0,0)\,\}
\ \ &\mbox{if}\ \lambda\neq \mu_{j}(\varepsilon ),\\
\{\,(\phi,\psi)=s(\phi_{j}(\varepsilon ),
\psi_{j} (\varepsilon ))\,:\,s\in\mathbb{R}\,\}
\ \ &\mbox{if}\ \lambda= \mu_{j}(\varepsilon )
\end{cases}
\end{equation}
with
$$
\psi_{j}(\varepsilon):=
T^{-1}(\mu_{j}(\varepsilon ), \varepsilon )
\{\,\mu_{j}(\varepsilon )m(x)U_{W}(\mu_{j}(\varepsilon), \varepsilon )
-\varepsilon 
h_{21}(\mu_{j}(\varepsilon ),\varepsilon )\,\}
\phi_{j}(\varepsilon).
$$
Then the proof of Lemma \ref{2ndbiflem} is accomplished.
\end{proof}
Our aim is to show that,
for each small $|\varepsilon |$,
$$
(\lambda, W, Z )=
(\mu_{j}(\varepsilon ), 
W(\mu_{j}(\varepsilon ), \varepsilon ),
Z(\mu_{j}(\varepsilon ), \varepsilon ))\in
\widetilde{\mathcal{C}}_{\varepsilon, \varLambda }
$$
gives a bifurcation point from which a curve
of solutions (not in $\widetilde{\mathcal{C}}_{\varepsilon, \varLambda}$)
to $G(\lambda, W, Z, \varepsilon )=0$
bifurcates.
For use of the bifurcation theorem from simple eigenvalues
\cite[Theorem 1.7]{CR},
we need to check the following transversality condition,
in addition to Lemma \ref{2ndbiflem}.

\begin{lem}\label{translem}
Assume that $\dim E_{j}=1$ with some $j\ge 2$.
Let $L(\lambda,\varepsilon )\in\mathcal{L}(\boldsymbol{X}, \boldsymbol{Y})$
be the operator defined by \eqref{Ldedef}.
Then,
there exists $\widetilde{\delta}_{j}\in (0,\delta_{j}]$,
where $\delta_{j}$ is the positive number in
Lemma \ref{2ndbiflem}, such that
if $|\varepsilon | <\widetilde{\delta}_{j}$,
then
\begin{equation}\label{tranper}
L_{\lambda}(\mu_{j}(\varepsilon), \varepsilon )
\biggl[
\begin{array}{l}
\phi_{j}(\varepsilon )\\
\psi_{j}(\varepsilon)
\end{array}
\biggr]\\
\not\in
\mbox{{\rm Ran}}\,
L(\mu_{j}(\varepsilon), \varepsilon)
\end{equation}
for any 
$(\phi_{j}(\varepsilon),\psi_{j}(\varepsilon))\in
{\rm Ker}\, 
L(\mu_{j}(\varepsilon), \varepsilon)$.
\end{lem}

\begin{proof}
Suppose, for contradiction, that
\eqref{tranper} is false.
Then there exist a real sequence $\{\varepsilon_{k}\}$
with
$\varepsilon_{k}\to 0$
and
$(\xi_{j}(\varepsilon_{k}), 
\zeta_{j}(\varepsilon_{k}))\in\boldsymbol{X}$
such that
$$
L(\mu_{j}(\varepsilon_{k}), \varepsilon_{k} )
\biggl[
\begin{array}{l}
\xi_{j}(\varepsilon_{k})\\
\zeta_{j}(\varepsilon_{k})
\end{array}
\biggr]
=
L_{\lambda}(\mu_{j}(\varepsilon_{k}), \varepsilon_{k} )
\biggl[
\begin{array}{l}
\phi_{j}(\varepsilon_{k})\\
\psi_{j}(\varepsilon_{k})
\end{array}
\biggr].
$$
By \eqref{Gwz},
this equation is represented as 
\begin{equation}\label{conju}
\begin{split}
&
-\biggl[
\begin{array}{ll}
\Delta\xi_{j}(\varepsilon_{k})
+\mu_{j}(\varepsilon_{k})m(x)\xi_{j}(\varepsilon_{k})\\
\Delta\zeta_{j}(\varepsilon_{k})
+\mu_{j}(\varepsilon_{k})m(x)(U^{(k)}_{W}\xi_{j}(\varepsilon_{k})+
U^{(k)}_{Z}\zeta_{j}(\varepsilon_{k})
\end{array}
\biggr]\\
&+\varepsilon_{k}
\biggl[
\begin{array}{ll}
2b_{1}U^{(k)}+(c_{1}-b_{2})V^{(k)} & (c_{1}-b_{2})U^{(k)}-2c_{2}V^{(k)}\\
2b_{1}U^{(k)}+c_{1}V^{(k)} & c_{1}U^{(k)}
\end{array}
\biggr]
\biggl[
\begin{array}{cc}
U^{(k)}_{W} & U^{(k)}_{Z}\\
V^{(k)}_{W} & V^{(k)}_{Z}
\end{array}
\biggr]
\biggl[
\begin{array}{c}
\xi_{j}(\varepsilon_{k})\\
\zeta_{j}(\varepsilon_{k})
\end{array}
\biggr]\\
=&
-m(x)
\biggl[
\begin{array}{l}
\phi_{j}(\varepsilon_{k})\\
U^{(k)}_{W}\phi_{j}(\varepsilon_{k})+U^{(k)}_{Z}
\psi_{j}(\varepsilon_{k})+
\mu_{j}(\varepsilon_{k})
\{\,\partial_{\lambda }U^{(k)}_{W}\phi_{j}(\varepsilon_{k}) 
+\partial_{\lambda }U^{(k)}_{Z}\psi_{j}(\varepsilon_{k})\,\}
\end{array}
\biggr]
\bigg|_{\lambda=\mu_{j}(\varepsilon_{k})}\\
&+\varepsilon_{k}
\biggl[
\begin{array}{ll}
\partial_{\lambda }h_{11}(\mu_{j}(\varepsilon_{k}), \varepsilon_{k} ) 
& \partial_{\lambda }h_{12}(\mu_{j}(\varepsilon_{k}), \varepsilon_{k} )\\
\partial_{\lambda }h_{21}(\mu_{j}(\varepsilon_{k}), \varepsilon_{k} ) 
& \partial_{\lambda }h_{22}(\mu_{j}(\varepsilon_{k}), \varepsilon_{k} )
\end{array}
\biggr]
\biggl[
\begin{array}{l}
\phi_{j}(\varepsilon_{k})\\
\psi_{j}(\varepsilon_{k})
\end{array}
\biggr],
\end{split}
\end{equation}
where 
$(U^{(k)}, V^{(k)})$ and
$(U^{(k)}_{W}, U^{(k)}_{Z}, V^{(k)}_{W}, V^{(k)}_{Z})$ 
are defined by \eqref{UVWZ} and
\eqref{Uwz} with
$$(W,Z)=(W(\mu_{j}(\varepsilon_{k}), \varepsilon_{k} ), 
Z(\mu_{j}(\varepsilon_{k}), \varepsilon_{k} )),$$
respectively.
Let $L^{*}(\mu_{j}(\varepsilon_{k}), \varepsilon_{k})$ be the 
conjugate operator of $L(\mu_{j}(\varepsilon_{k}), \varepsilon_{k})$
as follows:
\begin{equation}\label{con}
\begin{split}
&L^{*}(\mu_{j}(\varepsilon_{k}), \varepsilon_{k})
\biggl[
\begin{array}{l}
\phi\\
\psi
\end{array}
\biggr]
=
-\biggl[
\begin{array}{l}
\Delta\phi
+\mu_{j}(\varepsilon_{k})m(x)(\phi+U^{(k)}_{W}\psi)\\
\Delta\psi
+\mu_{j}m(x)U^{(k)}_{Z}\psi
\end{array}
\biggr]\\
&+\varepsilon_{k}
\biggl[
\begin{array}{cc}
U^{(k)}_{W} & V^{(k)}_{W}\\
U^{(k)}_{Z} & V^{(k)}_{Z}
\end{array}
\biggr]
\biggl[
\begin{array}{ll}
2b_{1}U^{(k)}+(c_{1}-b_{2})V^{(k)} & 2b_{1}U^{(k)}+c_{1}V^{(k)}\\
(c_{1}-b_{2})U^{(k)}-2c_{2}V^{(k)} & c_{1}U^{(k)}
\end{array}
\biggr]
\biggl[
\begin{array}{c}
\phi\\
\psi
\end{array}
\biggr].
\end{split}
\end{equation}
Thanks to \eqref{kerL},
the Fredholm alternative theorem gives
$$
\dim\mbox{Ker}\,L^{*}(\mu_{j}(\varepsilon _{k}), \varepsilon_{k})=
\dim\mbox{Ker}\,L(\mu_{j}(\varepsilon _{k}), \varepsilon_{k})=1
\ \ \mbox{for all}\ k\in\mathbb{N}.
$$
In view of \eqref{con},
we employ a similar argument as the proof of 
Lemma \ref{2ndbiflem} to verify that
$$
\mbox{Ker}\,L^{*}(\mu_{j}(\varepsilon _{k}), \varepsilon_{k})=
\{\,(\phi, \psi)=s(\phi_{j}^{*}(\varepsilon_{k}), 
\psi_{j}^{*}(\varepsilon_{k}))\,:\,s\in\mathbb{R}\,\}
$$
with some functions $(\phi_{j}^{*}(\varepsilon_{k}), \psi_{j}^{*}(\varepsilon_{k}))$ satisfying
$$ \lim_{k\to\infty}(\phi_{j}^{*}(\varepsilon_{k}), 
\psi_{j}^{*}(\varepsilon_{k}))=(\widetilde{\varPhi}_{j}, 0)\quad
\mbox{in}\ \boldsymbol{X}.
$$
Together with \eqref{kerconti},
taking the inner product of \eqref{conju} with 
$(\phi_{j}^{*}(\varepsilon_{k}), 
\psi_{j}^{*}(\varepsilon_{k}))$ and setting $k\to\infty$,
one can obtain
$$
\int_{\Omega}m(x)\widetilde{\varPhi}_{j}^{2}=0.
$$
Obviously, this is a contradiction.
Then the proof of Lemma \ref{translem} is complete.
\end{proof}

Owing to Lemmas \ref{2ndbiflem} and \ref{translem},
we can use the bifurcation theorem \cite[Theorem 1.7]{CR}
to construct a local curve of solutions to \eqref{WZeq} with $Z>0$
bifurcating from the point
$$
(\mu_{j}(\varepsilon), 
W(\mu_{j}(\varepsilon), \varepsilon ),
Z(\mu_{j}(\varepsilon), \varepsilon ))\in 
\widetilde{\mathcal{C}}_{\varepsilon, \varLambda}.
$$
\begin{prop}\label{2ndbifprop1}
Assume that
$\dim E_{j}=1$ 
with some $j\ge 2$.
Let $\widetilde{\delta}_{j}$ be the positive number obtained
in Lemma \ref{translem}.
If $|\varepsilon |<\widetilde{\delta}_{j}$,
then there
exist a neighborhood $
\widetilde{\mathcal{U}}_{j,\varepsilon}
(\,\subset\mathbb{R}\times
\boldsymbol{X}\,)$ of 
$$(\lambda, W,Z)=(\mu_{j}(\varepsilon),
W(\mu_{j}(\varepsilon), \varepsilon ),
Z(\mu_{j}(\varepsilon), \varepsilon ))\in
\widetilde{\mathcal{C}}_{\varepsilon,\varLambda}$$
and continuously differentiable functions
$$
(\xi(s,\varepsilon ), \widetilde{w}(s,\varepsilon ),
\widetilde{z}(s,\varepsilon ))\in\mathbb{R}\times\boldsymbol{X}
\ \ \mbox{for}\ s\in (-\delta_{j}, \delta_{j})$$ 
with some small
$\delta_{j}>0$ such that
$
(\xi(0, \varepsilon ), \widetilde{w}(0,\varepsilon ),
\widetilde{z}(0,\varepsilon ))=
(0, 0, 0)
$ 
and
\begin{equation}\label{2ndbifc}
\begin{split}
&\{\,(\lambda, W,Z)\in\widetilde{\mathcal{U}}_{j,\varepsilon}\,:\,G(\lambda, W,Z,\varepsilon )=0\,\}\\
=&
\{\,
(\lambda, W,Z)\in\widetilde{\mathcal{U}}_{j,\varepsilon}
\cap\widetilde{\mathcal{C}}_{\varepsilon, \varLambda}\,\}\,\cup\\
&
\{\,(\lambda, W,Z)=(\mu_{j}(\varepsilon), W(\mu_{j}(\varepsilon), \varepsilon),
Z(\mu_{j}(\varepsilon), \varepsilon))\\
&\ \ \ +
(\xi(s,\varepsilon ), s(\phi_{j}(\varepsilon)+\widetilde{w}(s,\varepsilon)),
s(\psi_{j}(\varepsilon)+\widetilde{z}(s,\varepsilon))),\ \ 
s\in (-\delta_{j}, \delta_{j})\,\}.
\end{split}
\end{equation}
\end{prop}
In view of \eqref{1stbifpt},
we recall that
the curve $\widetilde{\mathcal{C}}_{\varepsilon, \varLambda}$
bifurcates from the trivial solution $(W,Z)=(0,0)$ at
$\lambda=\lambda_{1}$.
Then it can be said that 
$(\mu_{j}(\varepsilon), 
W(\mu_{j}(\varepsilon), \varepsilon ),
Z(\mu_{j}(\varepsilon), \varepsilon ))\in 
\widetilde{\mathcal{C}}_{\varepsilon, \varLambda}$
gives a secondary bifurcation point.

By the change of variables \eqref{uvUV} and \eqref{UVdef},
Proposition \ref{2ndbifprop1} gives multiple bifurcation points on
the curve $\mathcal{C}_{\alpha, \varLambda}$ (obtained in Theorem \ref{Cathm})
of positive solutions of \eqref{SKT} with large $\alpha>0$:
\begin{prop}\label{2ndbifprop2}
Assume that
$\lambda_{j}<\varLambda$ and
$\dim E_{j}=1$ 
with some $j\ge 2$.
Set $\mu_{j, \alpha}:=\mu_{j}(1/\alpha)$
for $\mu_{j}(\varepsilon)$ obtained in Lemma
\ref{2ndbiflem}.
If $\alpha>0$ is sufficiently large,
then
there
exist a neighborhood $\mathcal{U}_{j, \alpha}\,(\,\subset
\mathbb{R}\times
\boldsymbol{X}\,)$ of 
$$(\lambda,u,v)=(\mu_{j, \alpha} ,
u_{0, \alpha}(\mu_{j, \alpha}\, ),
v_{0, \alpha}(\mu_{j, \alpha} ))\in\mathcal{C}_{\alpha, \varLambda}$$
and continuously differentiable functions
$$
(\xi_{j, \alpha}(s), 
\widetilde{u}_{j, \alpha}(s),
\widetilde{v}_{j, \alpha}(s))\in\mathbb{R}\times\boldsymbol{X}
\ \ \mbox{for}\ s\in (-\delta_{j}, \delta_{j})$$ 
with some small
$\delta_{j}>0$ such that
$
(\xi_{j, \alpha}(0), \widetilde{u}_{j, \alpha}(0),
\widetilde{v}_{j, \alpha}(0 ))=
(0, 0, 0)
$ 
and
\begin{equation}\label{2ndbifc2}
\begin{split}
&\{\,(\lambda,u,v)\in\mathcal{U}_{j, \alpha}\,:\,\mbox{$(u,v)$ is a positive solution of \eqref{SKT}}\,\}\\
=&
\{\,
(\lambda,u,v)\in\mathcal{U}_{j, \alpha}\cap\mathcal{C}_{\alpha, \varLambda}\,\}\,\cup\\
&
\{\,(\lambda,u,v)=(\mu_{j, \alpha}, 
u_{0, \alpha}(\mu_{j, \alpha}),
v_{0, \alpha}(\mu_{j, \alpha}))+
(\xi_{j, \alpha }(s), \widetilde{u}_{j, \alpha}(s),
\widetilde{v}_{j, \alpha}(s)),\ \ 
s\in (-\delta_{j}, \delta_{j})\,\}.
\end{split}
\end{equation}
\end{prop}
\begin{rem}\label{shrem}
By virtue of the proof of Lemma \ref{2ndbiflem}
applying the implicit function theorem to $H(\lambda,\phi,\varepsilon )$
around $(\lambda_{j}, \widetilde{\varPhi_{j}}, 0)$,
one can verify that $\mathcal{U}_{j,\alpha}$
does not shrink as $\alpha\to\infty$, that is,
the measure of $\mathcal{U}_{j,\alpha}$
remains positive as $\alpha\to\infty$.
\end{rem}

\begin{proof}
Let $(\lambda, W,Z)$ be any solution of $G(\lambda, W,Z,\varepsilon )=0$
with each $\varepsilon\in (0, \widetilde{\delta}_{j})$
obtained in \eqref{2ndbifc}.
By \eqref{uvUV} and \eqref{UVdef},
we see that
$\lambda$ and
$$
(u(W,Z),v(W,Z))=\dfrac{1}{2\alpha}
(\sqrt{(1-W)^{2}+4Z}-1+W,
\sqrt{(1-W)^{2}+4Z}-1-W)
$$
satisfy \eqref{SKT} with
$\alpha=1/\varepsilon $.
In view of \eqref{2ndbifc}, we set 
$\mu_{j, \alpha}:= \mu_{j}(\varepsilon )$ and
$$
u_{0, \alpha}(\mu_{j, \alpha})
=u(W(\mu_{j, \alpha},\varepsilon ), 
Z(\mu_{j, \alpha}, \varepsilon)),\quad
v_{0, \alpha}(\mu_{j, \alpha})=
v(W(\mu_{j, \alpha}, \varepsilon), 
Z(\mu_{j, \alpha}, \varepsilon ))
$$
with
$\alpha=1/\varepsilon $.
Substituting \eqref{2ndbifc} into
$(u(W,Z),v(W,Z))$, one can verify that
$\xi_{j, \alpha}(s):=\xi (s,\varepsilon)$
and
$$
\widetilde{u}_{j, \alpha}(s):=
u(W(\mu_{j}(\varepsilon), \varepsilon)
+s(\phi_{j}(\varepsilon)+\widetilde{w}(s,\varepsilon)),
Z(\mu_{j}(\varepsilon), \varepsilon)+
s(\psi_{j}(\varepsilon)+\widetilde{z}(s, \varepsilon)))
-
u_{0, \alpha}(\mu_{j, \alpha})
$$
and
$$
\widetilde{v}_{j, \alpha}(s):=
v(W(\mu_{j}(\varepsilon), \varepsilon)+
s(\phi_{j}(\varepsilon)+\widetilde{w}(s,\varepsilon)),
Z(\mu_{j}(\varepsilon), \varepsilon)+
s(\psi_{j}(\varepsilon)+\widetilde{z}(s, \varepsilon)))
-
v_{0, \alpha}(\mu_{j, \alpha})
$$
satisfy \eqref{2ndbifc2}.
Then the proof of Proposition \ref{2ndbifprop2} is complete.
\end{proof}

\section{Complete segregation}
\subsection{Perturbation of solutions of \eqref{LS2}}
In this subsection, we 
perturb $(w_{+}, w_{-})$ for the sign-changing solution
$w$ of \eqref{LS2} to construct the positive solution
$(u,v)$
of \eqref{SKT} with large $\alpha$.
To do so, we construct solutions of \eqref{wzeq} 
with the benefit of
the correspondence between
$(u,v, \alpha)$ and $(w,z,\varepsilon)$
expressed by \eqref{wzdef} and \eqref{uvdef}.
In view of the left-hand side of \eqref{wzeq},
we define the operator
$g\,:\,\mathbb{R}_{+}\times\boldsymbol{X}\times\mathbb{R}\to
\boldsymbol{Y}$ by 
$$
g(\lambda,w,z,\varepsilon):=
-\biggl[
\begin{array}{l}
\Delta w+\lambda m(x)w-b_{1}u^{2}+c_{2}v^{2}-(c_{1}-b_{2})uv\\
\Delta z+\varepsilon u(\lambda m(x)-b_{1}u)-\varepsilon c_{1}uv
\end{array}
\biggr],
$$
where
$(u,v)$ is regarded as the function of $(w,z,\varepsilon)$ 
defined as
\eqref{uvdef}.
In view of Lemma \ref{wzconvlem} and \eqref{wzconv2},
we recall that any sequence $\{(w_{n},z_{n})\}$ of 
solutions to \eqref{wzeq} with $\varepsilon=\varepsilon_{n}\to 0$ 
satisfies
$$
(w_{n}, z_{n})\to (w^{*},0)
\ \ \mbox{weakly in}\ \boldsymbol{X}\ \mbox{and strongly in}\ 
C^{1}(\overline{\Omega})^{2}
$$
and
\begin{equation}\label{uv0}
(u_{n},v_{n}):=(u(w_{n},z_{n}), v(w_{n}, z_{n}))\to
(w_{+}^{*}, w_{-}^{*})
\ \ \mbox{uniformly in}\ \overline{\Omega},
\end{equation}
passing to a subsequence if necessary,
with some $w^{*}\in X$ satisfying
$$
g(\lambda,w^{*},0,0)=
-\biggl[
\begin{array}{l}
\Delta w^{*}+\lambda m(x)w^{*}-b_{1}(w_{+}^{*})^{2}+c_{2}(w_{-}^{*})^{2}\\
0
\end{array}
\biggr]=0.
$$
Therefore, in order to solve the limiting equation
$g(\lambda,w,0,0)=0$, it is essential to consider 
sign-changing solutions to \eqref{LS2}.
It is noted that the function
$\mathbb{R}\ni w\mapsto -b_{1}w^{2}_{+}+c_{2}w^{2}_{-}\in\mathbb{R}$
is continuously differentiable at $w=0$.
Then, in the case where $m\in C(\overline{\Omega })$,
all weak solutions of \eqref{LS2} become 
classical solutions being in class $C^{2}(\overline{\Omega})$ 
thanks to the elliptic regularity.

In the simple case where
$\Omega=(-\ell, \ell )$ and $m(x)=m$
(constant)
$>0$
for all $x\in\overline{\Omega}$,
the usual shooting method  
enables us to classify the set of solutions of
\eqref{LS2} denoted as
$$
\widetilde{\mathcal{S}}_{\infty}:=\{
(\lambda,w)\in\mathbb{R}_{+}\times C^{2}([-\ell, \ell ])\,:\,
\mbox{$(\lambda,w)$ satisfies \eqref{LS2}}\,\}.
$$
In view of Lemma \ref{basiclem},
we recall that
$$\widetilde{\mathcal{S}}_{\infty}=
\bigcup\limits^{\infty}_{j=1}\widetilde{\mathcal{S}}_{j, \infty}\,\cup\,
\{\,(\lambda,w)\,:\,\lambda >0,\ w=0\,\},$$
where each $\widetilde{\mathcal{S}}_{j, \infty}$ forms a 
pitchfork bifurcation curve whose upper and lower branches
$\widetilde{\mathcal{S}}^{\pm}_{j, \infty}$ are parameterized by
$\lambda\in (\lambda_{j}, \infty)$, respectively,
as follows:
$\widetilde{\mathcal{S}}^{+}_{j, \infty}$ and 
$\widetilde{\mathcal{S}}^{-}_{j, \infty}$ 
can be represented as
$$
\widetilde{\mathcal{S}}^{\pm}_{j, \infty}=\{\,
(\lambda, w^{\pm}_{j,0}(\,\cdot\,,\lambda))\,:\,
\lambda\in (\lambda_{j}, \infty )\,\}.$$
In the following, we omit the spatial variable in the representation 
$w^{\pm}_{j,0}(\,\cdot\,,\lambda)$ and denote it by 
$w^{\pm}_{j,0}(\lambda)$.
Hence, it follows that
\begin{equation}\label{g0}
g(\lambda, w^{\pm}_{j,0}(\lambda), 0, 0)=0
\ \ \mbox{for any}\ \ \lambda\in (\lambda_{j}, \infty ).
\end{equation}
Our aim is to construct solutions of $g(\lambda, w, z, \varepsilon )=0$
near $(\lambda, w^{\pm}_{j,0}(\lambda), 0, 0)$ regarding $\varepsilon$ as
the perturbation parameter.
For this end, we calculate the linearized operator
$g_{(w,z)}(\lambda, w^{\pm}_{j,0}(\lambda), 0, 0)$
for any $\lambda\in (\lambda_{j}, \infty )$.

Differentiating $g(\lambda,w,z,\varepsilon )$ by $(w,z)$, one can see
\begin{equation}
\begin{split}
g_{(w,z)}(\lambda,w,z,\varepsilon )\left[
\begin{array}{c}
\phi\\
\psi
\end{array}
\right]
=&-
\biggl[
\begin{array}{l}
\Delta\phi+\lambda m(x)\phi\\
\Delta\psi+\varepsilon \lambda m(x)(u_{w}\phi+u_{z}\psi)
\end{array}
\biggl]\\
&+
\biggl[
\begin{array}{ll}
2b_{1}u+(c_{1}-b_{2})v & (c_{1}-b_{2})u-2c_{2}v\\
\varepsilon (2b_{1}u+c_{1}v) & \varepsilon c_{1}u
\end{array}
\biggr]
\biggl[
\begin{array}{cc}
u_{w} & u_{z}\\
v_{w} & v_{z}
\end{array}
\biggr]
\biggl[
\begin{array}{c}
\phi\\
\psi
\end{array}
\biggr],
\end{split}
\nonumber
\end{equation}
where $(u,v)$ is the function of $(w,z,\varepsilon )$ 
defined by \eqref{uvdef}, and thereby,
\begin{equation}
\begin{split}
&u_{w}(w,z, \varepsilon)=
\dfrac{1}{2}\biggl(
\dfrac{w-\varepsilon }
{\sqrt{(w-\varepsilon )^{2}+4z}}+1
\biggr),\quad
u_{z}(w,z, \varepsilon)=
\dfrac{1}{\sqrt{(w-\varepsilon )^{2}+4z}},\\
&
v_{w}(w,z, \varepsilon)=
\dfrac{1}{2}\biggl(
\dfrac{w-\varepsilon }
{\sqrt{(w-\varepsilon )^{2}+4z}}-1
\biggr),\quad
v_{z}(w,z, \varepsilon)=
\dfrac{1}{\sqrt{(w-\varepsilon )^{2}+4z}}.
\end{split}
\nonumber
\end{equation}
Here we note that
all zeros of $w^{\pm}_{j,0}(x,\lambda)$ 
on $[-\ell, \ell ]$ are not degenerate in the sense that,
if $w^{\pm}_{j,0}(x^{*},\lambda )=0$ with some $x^{*}\in [-\ell, \ell ]$, 
then
$(w^{\pm}_{j,0})_{x}(x^{*},\lambda)\neq 0$.
Owing to the non-degeneracy,
one can verify that $g(\lambda,w,z,\varepsilon )$ is 
differentiable with respect to $(w,z)$ 
at $(\lambda,w^{\pm}_{j,0}(\lambda),0,0)$ and the
derivative can be expressed as follows:
\begin{equation}\label{gwz0}
\begin{split}
&g_{(w,z)}(\lambda,w^{\pm}_{j,0}(\lambda),0,0)\left[
\begin{array}{c}
\phi\\
\psi
\end{array}
\right]\\
&=
\biggl[
\begin{array}{l}
-\phi''-\{\,\lambda m(x)-2(b_{1}(w^{\pm}_{j,0})_{+}+c_{2}
(w^{\pm}_{j,0})_{-})\,\}\phi+
\{\,2(b_{1}\chi_{\{w>0\}}-c_{2}\chi_{\{w<0\}})+c_{1}-b_{2}\,\}
\psi\\
-\psi''
\end{array}
\biggr],
\end{split}
\nonumber
\end{equation}
where $\chi_{\{w>0\}}$
(resp. $\chi_{\{w<0\}}$) denotes
the indicator function of the set where
$w^{\pm}_{j,0}(\lambda)$ takes positive (resp. negative) values.
Then the eigenvalue problem
\begin{equation}\label{gev}
g_{(w,z)}
(\lambda, w^{\pm}_{j,0}(\lambda), 0, 0)\left[
\begin{array}{c}
\phi\\
\psi
\end{array}
\right]
=
\sigma
\left[
\begin{array}{c}
\phi\\
\psi
\end{array}
\right]
\end{equation}
is reduced to the system of linear equations:
\begin{subequations}\label{linev}
\begin{align}
&\phi''+\{\,\lambda m(x)-2(b_{1}(w^{\pm}_{j,0})_{+}
+c_{2}(w^{\pm}_{j,0})_{-})+\sigma\,\}\phi
\nonumber
\\ 
\label{linev1}
=\,& 
\{\,2(b_{1}\chi_{\{w>0\}}-c_{2}\chi_{\{w<0\}})+c_{1}-b_{2}\,\}
\psi
\ \ \mbox{in}\ (-\ell, \ell ),\quad 
\phi(0)=\phi(\ell)=0,\\
\label{linev2}
&-\psi''=\sigma\psi\ \ \mbox{in}\ (-\ell, \ell ),\quad
\psi(0)=\psi(\ell)=0.
\end{align}
\end{subequations}
We show that $g_{(w,z)}
(\lambda, w^{\pm}_{j,0}(\lambda), 0, 0)\in
\mathcal{L}(\boldsymbol{X}, \boldsymbol{Y})$
is non-degenerate in the sense that
$\sigma =0$ is not an eigenvalue of \eqref{gev}.
If
$\sigma =0$, then \eqref{linev2} implies $\psi =0$.
Next we consider \eqref{linev1} with $\psi =0$:
\begin{equation}\label{self}
\phi''+\{\,\lambda m(x)-2(b_{1}(w^{\pm}_{j,0})_{+}
+c_{2}(w^{\pm}_{j,0})_{-})+\sigma\,\}\phi=0
\ \ \mbox{in}\ (-\ell, \ell ),\quad
\phi(0)=\phi(\ell)=0.
\end{equation}
By the self-adjointness of the operator associated with \eqref{self},
we know that all eigenvalues of \eqref{self} are real.
Furthermore, by applying the Sturm-Liouville theory to 
\eqref{LS2} and \eqref{self},
one can see that the number of zeros of $w^{\pm}_{j,0}(\lambda)$ on 
$(-\ell, \ell )$ coincides with
the number of negative eigenvalues of \eqref{self}.
Since the number of zeros of $w^{\pm}_{j,0}(\lambda)$ on 
$(-\ell, \ell )$ is equal to $j-1$, 
then all eigenvalues of \eqref{self} form a sequence
$\{\sigma^{\pm}_{k}(\lambda)\}^{\infty}_{k=1}$ such as
$$
\sigma^{\pm}_{1}(\lambda)<\sigma^{\pm}_{2}(\lambda)<\cdots<
\sigma^{\pm}_{j-1}(\lambda)\,(\,<0<\,)\, 
\sigma^{\pm}_{j}(\lambda)<\sigma_{j+1}^{\pm }(\lambda)<\cdots.
$$
Hence this fact implies that 
$g_{(w,z)}
(\lambda, w^{\pm}_{j,0}(\lambda), 0, 0)$
does not have the zero eigenvalue,
in other words,
$\sigma=0$ belongs to the resolvent set of 
$g_{(w,z)}(\lambda,w^{\pm}_{j,0}(\lambda),0,0)\in\mathcal{L}
(\boldsymbol{X}, \boldsymbol{Y})$.
Consequently, we deduce that $g_{(w,z)}
(\lambda, w^{\pm}_{j,0}(\lambda), 0, 0)\in
\mathcal{L}(\boldsymbol{X},\boldsymbol{Y})$ 
is an isomorphism.
Together with \eqref{g0}, 
for any fixed $\lambda_{*}\in (\lambda_{j}, \infty )$,
the implicit function theorem ensures
a neighborhood $\mathcal{U}_{*}^{\pm}=
\mathcal{U}_{*}^{\pm}(\lambda_{*})$ of 
$(\lambda,w,z,\varepsilon )=(\lambda_{*}, w^{\pm}_{j,0}(\lambda_{*}), 0,0)\in
\mathbb{R}_{+}\times\boldsymbol{X}\times\mathbb{R}$,
a small $\delta^{\pm}_{j}=\delta^{\pm}_{j}(\lambda_{*})>0$ 
and continuously differentiable functions
$$
(\lambda_{*}-\delta^{\pm}_{j}, \lambda_{*}+\delta^{\pm}_{j})\times
(-\delta^{\pm}_{j}. \delta^{\pm}_{j})\ni
(\lambda,\varepsilon )\mapsto
(w^{\pm}_{j}(\lambda,\varepsilon ),
z^{\pm}_{j}(\lambda,\varepsilon ))\in\boldsymbol{X}
$$
such that
\begin{equation}\label{wjzjdef}
\begin{split}
&\{\,(\lambda,w,z,\varepsilon )\in\mathcal{U}_{*}^{\pm}\,:\,
g(\lambda,w,z,\varepsilon )=0\,\}\\
=&\{\,(\lambda, w^{\pm}_{j}(\lambda, \varepsilon ),
z^{\pm}_{j}(\lambda, \varepsilon ), \varepsilon )\,:\,
(\lambda,\varepsilon )\in ( \lambda_{*}-\delta^{\pm}_{j}(\lambda_{*}), \lambda_{*}+\delta^{\pm}_{j}(\lambda_{*}))
\times (-\delta^{\pm}_{j}(\lambda_{*}), \delta^{\pm}_{j}(\lambda_{*}))\,\}
\end{split}
\end{equation}
with
$
(w^{\pm}_{j}(\lambda,0), z^{\pm}_{j}(\lambda,0))
=(w^{\pm}_{j,0}(\lambda), 0)$
for any $\lambda\in (\lambda_{*}-\delta^{\pm}_{j,0}(\lambda_{*}),
\lambda_{*}+\delta^{\pm}_{j,0}(\lambda_{*}))$.

\begin{prop}\label{perCSprop}
Assume that $\Omega=(-\ell, \ell)$ and $m(x)=m$
(constant)
$>0$
for all $x\in\overline{\Omega}$.
For any small $\eta>0$,
large $\varLambda >0$
and $j\in\mathbb{N}$,
there exists a small 
$\underline{\delta}^{\pm}_{j}(\eta, \varLambda )>0$ such that,
for each $\varepsilon\in 
(-\underline{\delta}^{\pm}_{j}, \underline{\delta}^{\pm}_{j})$,
there exists a pair of simple curves 
$\widetilde{\mathcal{S}}^{\pm}_{j,\varepsilon}(\lambda_{j}+\eta, \varLambda )$
consisting of solutions of \eqref{wzeq} parameterized by
$\lambda\in (\lambda_{j}+\eta, \varLambda)$ such as
$$
\widetilde{\mathcal{S}}^{\pm}_{j,\varepsilon}(\lambda_{j}+\eta, \varLambda )=\{\,
(\lambda,w,z)\in\mathbb{R}_{+}\times\boldsymbol{X}\,:\,
(w,z)=(w^{\pm}_{j}(\lambda,\varepsilon ), z^{\pm}_{j}(\lambda,\varepsilon)),
\ \ \lambda\in (\lambda_{j}+\eta, \varLambda )\,\},
$$
with $(w^{\pm}_{j}(\lambda,\varepsilon ), 
z^{\pm}_{j}(\lambda,\varepsilon))$
defined by \eqref{wjzjdef} satisfying
\begin{equation}\label{nakatugi}
(w^{\pm}_{j}(\lambda,0), z^{\pm}_{j}(\lambda,0))
=(w^{\pm}_{j,0}(\lambda), 0)\ \ \mbox{for any}\ 
\lambda\in (\lambda_{j}+\eta, \varLambda ).
\end{equation}
\end{prop}

\begin{proof}
In view of \eqref{wjzjdef}, we set
\begin{equation}
\begin{split}
\varGamma^{\pm}_{j}(\lambda_{*}):=\{\,
&(\lambda,w^{\pm}_{j}(\lambda,\varepsilon ), z^{\pm}_{j}(\lambda,\varepsilon), \varepsilon )\in
\mathbb{R}\times\boldsymbol{X}\times\mathbb{R}\,:\,\\
&(\lambda,\varepsilon )\in [\lambda_{*}-\delta^{\pm}_{j}(\lambda_{*})/2, 
\lambda_{*}+\delta^{\pm}_{j}(\lambda_{*})/2]\times [
-\delta^{\pm}_{j}(\lambda_{*})/2, \delta^{\pm}_{j}(\lambda_{*})/2]\,\}.
\end{split}
\nonumber
\end{equation}
Here we remark that, for each $j\in\mathbb{N}$,
$$
\bigcup_{\lambda_{j}+\eta\le \lambda_{*}\le 
\varLambda}\varGamma^{+}_{j}(\lambda_{*})
\quad\mbox{and}\quad
\bigcup_{\lambda_{j}+\eta\le \lambda_{*}\le 
\varLambda}\varGamma^{-}_{j}(\lambda_{*})
$$
give a pair of 
compact sets.
It follows from \eqref{wjzjdef} that
$$\varGamma^{\pm}_{j}(\lambda_{*})\subset
\mathcal{U}^{\pm}_{j}(\lambda_{*}),$$
and therefore,
$$
\bigcup_{\lambda_{j}+\eta\le \lambda_{*}\le 
\varLambda}\varGamma^{\pm}_{j}(\lambda_{*})
\subset
\bigcup_{\lambda_{j}+\eta\le \lambda_{*}\le 
\varLambda}
\mathcal{U}^{\pm}_{j}(\lambda_{*}).
$$
By the compactness of 
$\cup_{\lambda_{j}+\eta\le \lambda_{*}\le 
\varLambda}\varGamma^{\pm}_{j}(\lambda_{*})$,
we can find a finite number of $\{l_{i}\}_{i=1}^{k}$ such that
$$\lambda_{j}+\eta\le l_{1}<
l_{2}<\cdots<l_{k}\le \varLambda$$
and
$$
\bigcup_{\lambda_{j}+\eta\le \lambda_{*}\le 
\varLambda}\varGamma^{\pm}_{j}(\lambda_{*})
\subset
\bigcup_{i=1}^{k}\,
\mathcal{U}^{\pm}_{j}(l_{i}).
$$
Define
$$
\underline{\delta}^{\pm}_{j}:=\min_{1\le i\le k}
\dfrac{\delta^{\pm}_{j}(l_{i})}{2}.
$$
Here we recall that,
for each $i\in\{1,2,\ldots,k\}$,
all solutions of $g(\lambda,w,z,\varepsilon )=0$ with 
$(\lambda,w,z,\varepsilon )\in\mathcal{U}^{\pm}_{j}(l_{i})$
consist of $\varGamma^{\pm}_{j}(l_{i})$.
Therefore, we see that,
for any fixed $\varepsilon\in 
(-\underline{\delta}^{\pm}_{j},\underline{\delta}^{\pm}_{j})$,
$$
\widetilde{\mathcal{S}}^{\pm}_{j,\varepsilon}(\lambda_{j}+\eta, \varLambda ):=
\{\,(\lambda,w^{\pm}_{j}(\lambda,\varepsilon ), z^{\pm}_{j}(\lambda, \varepsilon ))\,:\,
\lambda_{j}+\eta< \lambda< \varLambda\,\}
$$
form a pair of simple curves of solutions of
$g(\lambda,w,z,\varepsilon )=0$, or equivalently, \eqref{wzeq}.
The proof of Proposition \ref{perCSprop} is complete.
\end{proof}
In the case where $\varepsilon\in (0,\underline{\delta}^{\pm}_{j})$,
by \eqref{uvdef}
one can transform $\widetilde{\mathcal{S}}^{\pm}_{j,\varepsilon}(\lambda_{j}+\eta, \varLambda )$
to a pair of branches of positive solutions of \eqref{SKT}:
\begin{cor}\label{persegcor}
Assume that 
$\Omega=(-\ell, \ell)$ and $m(x)=m$
(constant)
$>0$
for all $x\in\overline{\Omega}$. 
Let
$\underline{\delta}^{\pm}_{j}(\eta, \varLambda )>0$ 
be the positive number obtained in Proposition
\ref{perCSprop}
for any small $\eta>0$,
large $\varLambda >0$
and $j\in\mathbb{N}$.
Then,
for each $\alpha>1/\delta^{\pm}_{j}$,
there exists a pair of simple curves 
$\mathcal{S}^{\pm}_{j,\alpha}(\lambda_{j}+\eta, \varLambda)$
consisting of positive solutions to \eqref{SKT} parameterized by
$\lambda\in (\lambda_{j}+\eta, \varLambda)$ such as
$$
\mathcal{S}^{\pm}_{j,\alpha}(\lambda_{j}+\eta, \varLambda)=\{\,
(\lambda,u,v)\in\mathbb{R}_{+}\times\boldsymbol{X}\,:\,
(u,v)=(u^{\pm}_{j}(\lambda,\varepsilon ), v^{\pm}_{j}(\lambda,\varepsilon)),
\ \ \lambda\in (\lambda_{j}+\eta, \varLambda )\,\},
$$
where 
$\varepsilon=1/\alpha$ and
$$
(u^{\pm}_{j}(\lambda,\varepsilon ), v^{\pm}_{j}(\lambda,\varepsilon))
=
\dfrac{1}{2}(
\sqrt{(w-\varepsilon )^{2}+4z}+w-\varepsilon ,
\sqrt{(w-\varepsilon )^{2}+4z}-w-\varepsilon 
)
$$
with 
$(w,z)=(w_{j}^{\pm}(\lambda, \varepsilon ),
z_{j}^{\pm}(\lambda, \varepsilon ))$.
\end{cor}

\subsection{From small coexistence to complete segregation} 
This subsection is devoted to the proof of Theorem \ref{g2ndbifthm}.
Our strategy of the proof is to
construct a global branch of positive solutions of
\eqref{SKT} bifurcating from a positive solution on 
the branch $\mathcal{C}_{\alpha, \varLambda}$ 
(obtained in Theorem \ref{Cathm})
by connecting the local bifurcation branch \eqref{2ndbifc2} 
with a pair of branches $\mathcal{S}^{\pm}_{j, \alpha}$.

\begin{proof}[Proof of Theorem \ref{g2ndbifthm}]
By virtue of \eqref{2ndbifc2}, we focus on 
the neighborhood $\mathcal{U}_{j, \alpha}$
of 
$$
(\lambda^{*}, u^{*}, v^{*}):=
(\mu_{j,\alpha}, u_{0,\alpha}(\mu_{j, \alpha}), v_{0,\alpha}(\mu_{j,\alpha}))
\in\mathcal{C}_{\alpha, \varLambda}.$$
In view of Proposition \ref{2ndbifprop2},
we recall that
all positive solutions of \eqref{SKT} in $\mathcal{U}_{j, \alpha}$
consist of 
the union of the piece of $\mathcal{C}_{\alpha, \varLambda}$ and
the piece of the bifurcation curve
$$
\varGamma_{j, \alpha}:=
\{\,(\lambda^{*}, u^{*}, v^{*})+
(\xi_{j, \alpha }(s), \widetilde{u}_{j, \alpha}(s),
\widetilde{v}_{j, \alpha}(s)),\ \ 
s\in (-\delta_{j}, \delta_{j})\,\}.
$$
On the other hand, 
in view of Corollary \ref{persegcor},
we focus on the end points at $\lambda=\lambda_{j}+\eta$
of $\mathcal{S}^{\pm}_{j,\alpha}(\lambda_{j}+\eta, \varLambda)$:
$$
(\lambda_{*}, u^{\pm}_{*}, v^{\pm}_{*}):=
(\lambda_{j}+\eta, u_{j}^{\pm}(\lambda_{j}+\eta, \varepsilon ),
v^{\pm}_{j}(\lambda_{j}+\eta, \varepsilon )),
$$
where
\begin{equation}\label{ujvj}
\begin{split}
&(u_{j}^{\pm}(\lambda_{j}+\eta, \varepsilon ),
v_{j}^{\pm}(\lambda_{j}+\eta, \varepsilon ))\\
&=
\lim_{\lambda\searrow\lambda_{j}+\eta}
(u_{j}(\lambda, \varepsilon ), v_{j}(\lambda, \varepsilon ))\\
&=
\dfrac{1}{2}\lim_{\lambda\searrow\lambda_{j}+\eta}
\biggl(
\sqrt{(w^{\pm}_{j}(\lambda, \varepsilon )-\varepsilon )^{2}+
4z^{\pm}_{j}(\lambda, \varepsilon )}+
w^{\pm}_{j}(\lambda, \varepsilon )-\varepsilon ,\\
&\quad\quad\sqrt{(w^{\pm}_{j}(\lambda, \varepsilon )-\varepsilon )^{2}+
4z^{\pm}_{j}(\lambda, \varepsilon )}-
w^{\pm}_{j}(\lambda, \varepsilon )-\varepsilon 
\biggr)
\quad\mbox{with}\ \ \varepsilon =\dfrac{1}{\alpha}.
\end{split}
\end{equation}
It will be shown that
$(\lambda^{*}, u^{*}, v^{*})$ and
$(\lambda_{*}, u_{*}^{\pm}, v_{*}^{\pm})$
are very close if 
$\eta >0 $ is sufficiently small and 
$\alpha$ is sufficiently large.
Actually, we know from \eqref{ualto0} and \eqref{mujtola}
that
$(\lambda^{*}, u^{*}, v^{*})\to (\lambda_{j}, 0, 0)$
in $\mathbb{R}\times\boldsymbol{X}$
as $\alpha\to\infty$. 
We recall \eqref{nakatugi} to note
$$
\lim_{\lambda\searrow\lambda_{j}+\eta,\ \varepsilon\to 0}
(w_{j}^{\pm}(\lambda, \varepsilon ), z_{j}^{\pm}(\lambda, \varepsilon ))
=(w_{j,0}^{\pm}(\lambda_{j}+\eta ), 0).
$$
Furthermore, we see from Lemma \ref{basiclem} that
$w_{j,0}^{\pm}(\lambda_{j}+\eta )\to 0$
as $\eta\searrow 0$.
It follows from \eqref{ujvj} that
$(\lambda_{*}, u^{\pm}_{*}, v^{\pm}_{*})$ are very
close to $(\lambda_{j}, 0, 0)$ if
$\eta >0$ is sufficiently small and
$\alpha$ is sufficiently large.
Hence,
$(\lambda^{*}, u^{*}, v^{*})$ and
$(\lambda_{*}, u_{*}^{\pm}, v_{*}^{\pm})$
are very close in such a situation.
Therefore,
we see that
a pair of end points
$(\lambda_{*}, u^{+}_{*}, v^{+}_{*})$
and
$(\lambda_{*}, u^{-}_{*}, v^{-}_{*})$
are contained in the neighborhood $\mathcal{U}_{j,\alpha}$ 
of $(\lambda^{*},u^{*}, v^{*})$
if $\eta>0$ is sufficiently small and
$\alpha$ is sufficiently large.
Here it should be noted that
$\mathcal{U}_{j,\alpha}$ does not shrink when $\alpha\to\infty$
as stated in Remark \ref{shrem}.
Furthermore, one can verify from \eqref{uv0} that
$(\lambda_{*}, u^{\pm}_{*}, v^{\pm}_{*})\not\in
\mathcal{C}_{\alpha, \varLambda}$.
Therefore, we deduce 
that
$(\lambda_{*}, u^{+}_{*}, v^{+}_{*})$
and
$(\lambda_{*}, u^{-}_{*}, v^{-}_{*})$
connect with the bifurcation curve
$\varGamma_{j, \alpha}$, 
that is, 
$$
(\lambda_{*}, u^{+}_{*}, v^{+}_{*})
=(\lambda^{*}, u^{*}, v^{*})+
(\xi_{j, \alpha }(s^{+}), \widetilde{u}_{j, \alpha}(s^{+}),
\widetilde{v}_{j, \alpha}(s^{+}))
$$
and
$$
(\lambda_{*}, u^{-}_{*}, v^{-}_{*})
=(\lambda^{*}, u^{*}, v^{*})+
(\xi_{j, \alpha }(s^{-}), \widetilde{u}_{j, \alpha}(s^{-}),
\widetilde{v}_{j, \alpha}(s^{-}))
$$
with some $s^{-}<0<s^{+}$.
Hence we have constructed a pair 
of branches $\mathcal{S}^{\pm}_{j,\alpha, \varLambda}$
(stated in Theorem \ref{g2ndbifthm})
bifurcating from $(\lambda^{*}, u^{*}, v^{*})
\in\mathcal{C}_{\alpha, \varLambda}$.
Furthermore we remark that 
$w_{j,0}^{\pm}(x)$ are odd (resp.\,even) functions
on $(-\ell, \ell)$ if $j$ is even (resp.\,odd).
Therefore, if $j$ is even, then
$w_{j,0}^{-}(x)=w^{+}_{j,0}(-x)$ 
for all $x\in (-\ell, \ell)$.
Together with the fact that
$(\widehat{u}(x), \widehat{v}(x)):=(u(-x), v(-x))$
is a solution of \eqref{SKT} when
$(u(x), v(x))$ is a solution,
we can verify \eqref{evensym}.
The proof of Theorem \ref{g2ndbifthm} is complete.
\end{proof}

\section{Numerical results}
\subsection{Numerical bifurcation diagram with parameter $\lambda$}
In this subsection,
we numerically exhibit the bifurcation diagram of solutions of 
\eqref{SKT} 
by using the continuation software \texttt{pde2path}
\cite{BKS, DRUW, Ue, UWR} based on an FEM discretization
of the stationary problem.
It should be noted here that 
our numerical simulations are heavily based on 
the recent paper \cite{BKS} by Breden, Kuehn and Soresina 
that traces the bifurcation branches of stationary solutions 
of the SKT model.

For \eqref{SKT}, our setting of parameters in the numerical simulation 
is as follows:
$$
\Omega=(-0.5,0.5),
\quad \alpha=20,
\quad m(x)=1, 
\quad b_{1}=3, 
\quad b_{2}=2, 
\quad c_{1}=2, 
\quad c_{2}=1.
$$
The numerical bifurcation diagram is shown in Figure 1,
where
the horizontal axis represents the bifurcation parameter $\lambda$, 
and the vertical axis represents the $L^{2}$ norm of the 
$u$ component of positive solutions $(u,v)$ to \eqref{SKT}.
In Figure 1,
the blue curve corresponds to $\mathcal{C}_{20, \varLambda}$ 
(pure mathematically obtained in Theorem \ref{Cathm})
which bifurcates from the trivial solution at $\lambda=\lambda_{1}$.
The theoretical value of the bifurcation point in the
setting \eqref{ev2} is
$\lambda_{1}=\pi^2\,(\,\fallingdotseq 9.8696)$, 
which is naturally consistent with the numerical bifurcation point 
of the blue curve appeared in Figure 1.
It can be seen the blue curve extends 
in the direction of large 
$\lambda$ with small but gradually increasing 
$L^2$ norm,
and exhibits the branch of small coexistence.
Figure 2 shows the profile of a positive 
solution $(u,v)$ 
corresponding to the point at $\lambda=59.8286$
on the blue curve in Figure 1.
It can be observed from Figure 2 that
$u$ and $v$ are small and
very close to each other.

In Figure 1, 
the red curve corresponds to
the upper branch $\mathcal{S}_{2,20,\varLambda}^{+}$
and lower one $\mathcal{S}_{2,20,\varLambda}^{-}$ of 
the pitchfork bifurcation curve $\mathcal{S}_{2,20,\varLambda}$,
which is pure mathematically obtained in Theorem \ref{g2ndbifthm}.
It follows from Theorem \ref{g2ndbifthm}
that the pitchfork bifurcation curve $\mathcal{S}_{2, 20, \varLambda}$
bifurcates from
a solution on the blue curve at $\mu_{2,20}$.
Here it is noted the $L^{2}$ norm of each component
of $\mathcal{S}_{2,20, \varLambda}^{+}$ and 
$\mathcal{S}_{2,20, \varLambda}^{-}$
should be shown overlapped since \eqref{evensym}.
By \eqref{biflim}, the secondary bifurcation point 
$\mu_{2,\alpha}$ theoretically tends to 
$\lambda_{2}=(2\pi)^2\,(\,\fallingdotseq 39.4784)$
as $\alpha\to\infty$.
It can be seen in Figure 1 that
the numerical secondary bifurcation point 
is near the theoretical limit.

In Figure 3, (a) and (b) show the profiles of
solutions on the upper branch $\mathcal{S}^{+}_{2, 20, \varLambda}$ 
and the lower one $\mathcal{S}^{-}_{2, 20, \varLambda}$ 
of the red pitchfork bifurcation curve 
$\mathcal{S}_{2, 20, \varLambda}$
at
$\lambda=40.3421$, respectively.
One can see that $u$ and $v$ are somewhat spatially segregated.
Furthermore, as the solution moves away 
from the bifurcation point on the blue bifurcation curve, 
it is observed from the numerical simulation that 
the segregation between $u$ and $v$ becomes more overt. 
Actually, (c) and (d) in Figure 3 shows the profiles of
solutions on the upper branch $\mathcal{S}^{+}_{2, 20, \varLambda}$ 
and the lower one $\mathcal{S}^{+}_{2, 20, \varLambda}$ 
on the red pitchfork bifurcation curve 
$\mathcal{S}_{2, 20, \varLambda}$ 
at
$\lambda=43.0673$, 
where $u$ and $v$ considerably segregate each other.

\begin{figure}
\begin{center}
{\includegraphics*[scale=.5]{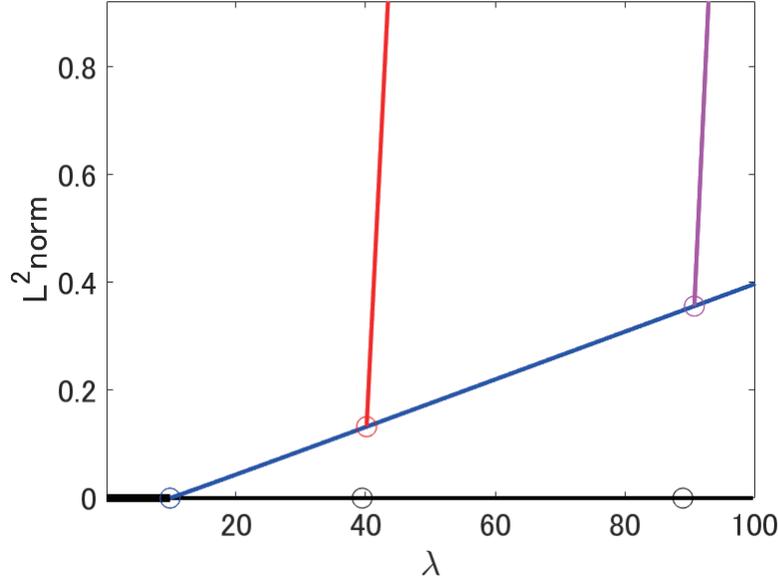}}\\
\caption{Bifurcation diagram of solutions of \eqref{SKT}.}
\end{center}\end{figure}

\begin{figure}
\begin{center}
{\includegraphics*[scale=.4]{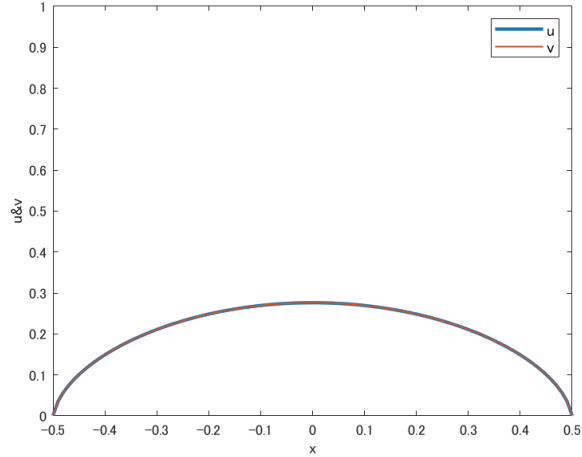}}\\
\caption{Profile of a solution on blue curve at $\lambda=59.8286$.}
\end{center}\end{figure}

\begin{figure}
\centering
\subfigure[Profile of a solution on red upper branch at $\lambda=40.3421$.]{
\includegraphics*[scale=.3]{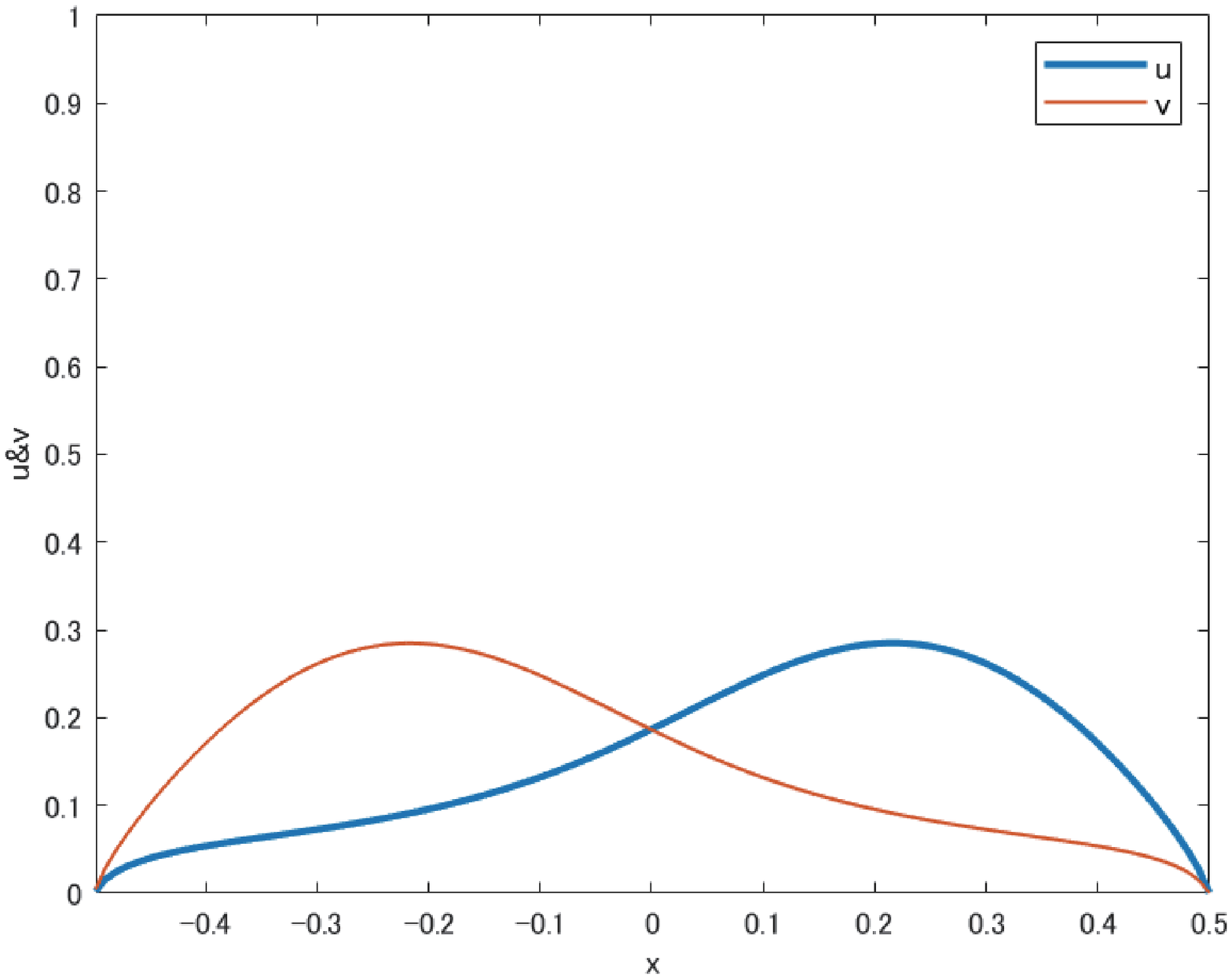}
\label{figa}}
\subfigure[Profile of a solution on red lower branch at $\lambda=40.3421$.]{
\includegraphics*[scale=.3]{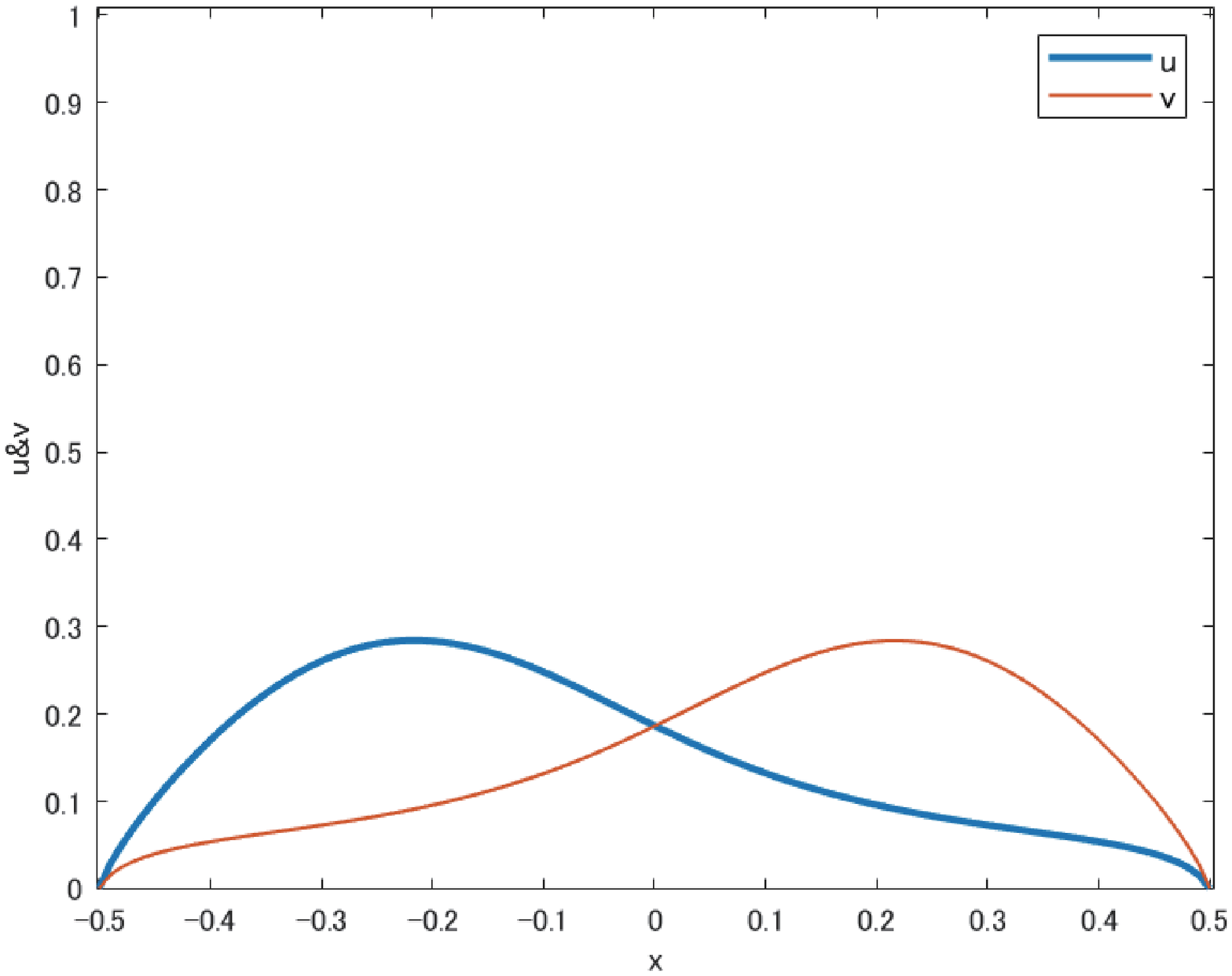}
\label{figb}
}
\centering
\subfigure[Profile of a solution on red upper branch at $\lambda=43.0673$.]{
\includegraphics*[scale=.3]{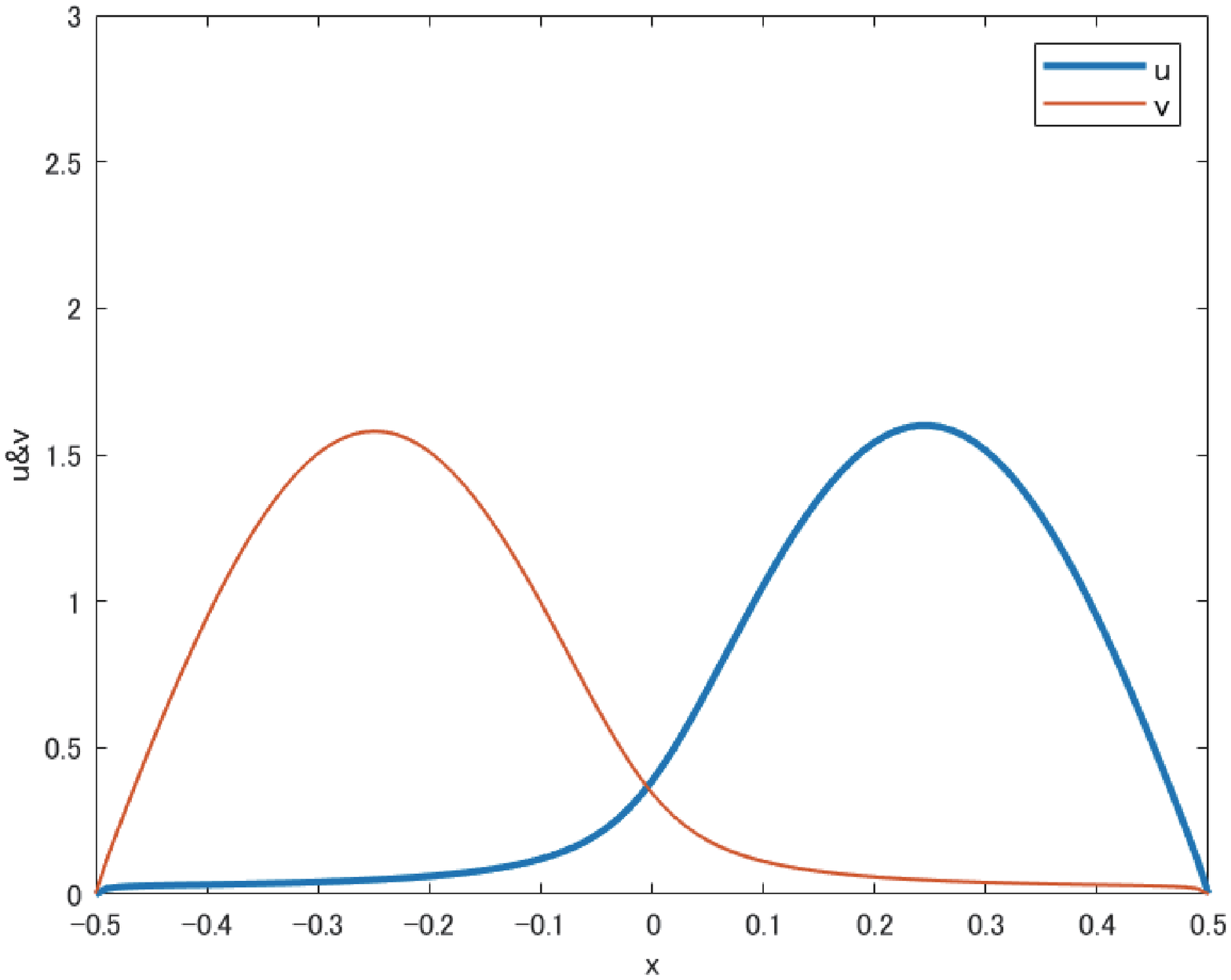}
\label{figc}}
\subfigure[Profile of a solution on red lower branch at $\lambda=43.0673$.]{
\includegraphics*[scale=.3]{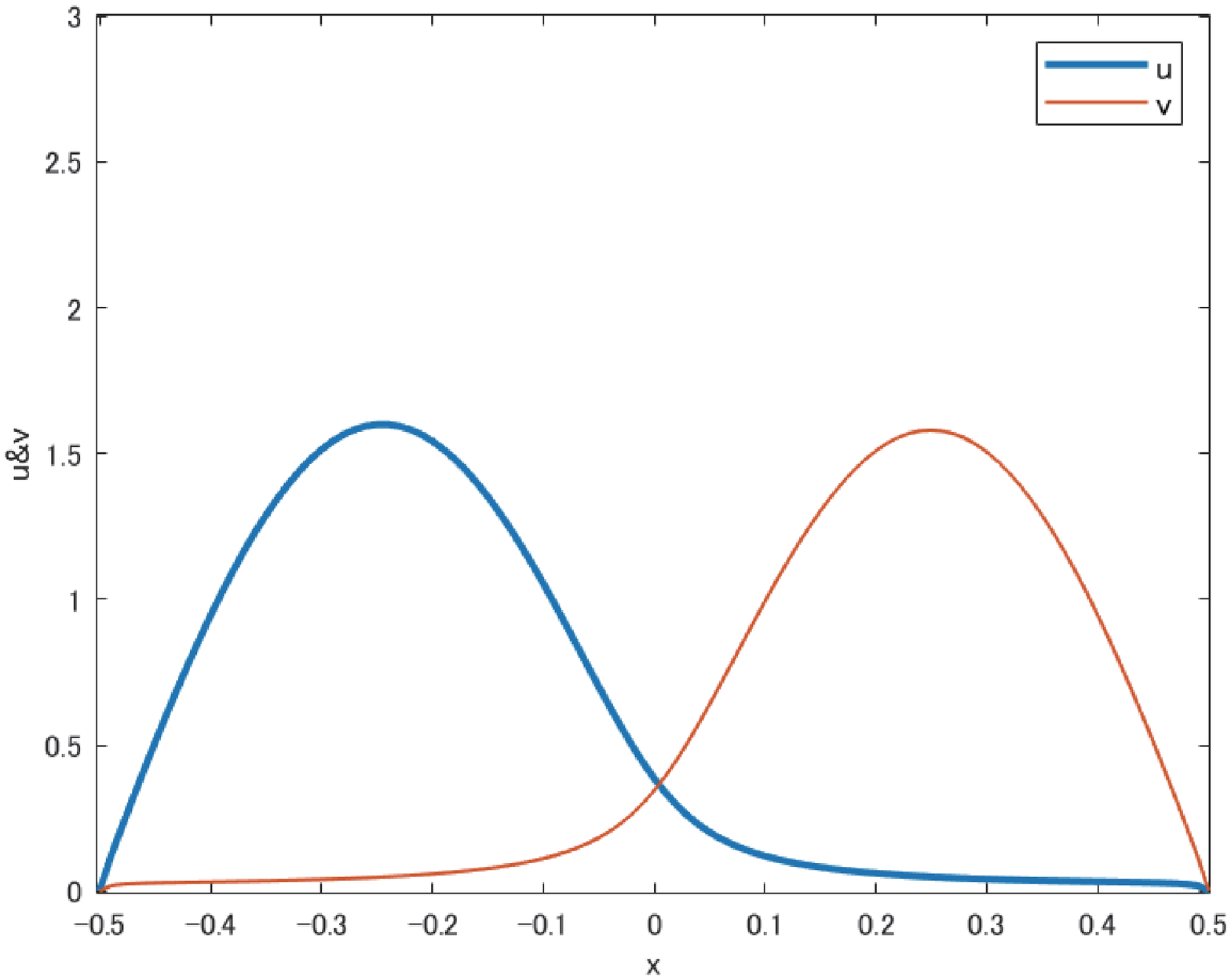}
\label{figd}
}
\centering
\subfigure[Profile of a solution on purple upper branch at $\lambda=91.5836$.]{
\includegraphics*[scale=.3]{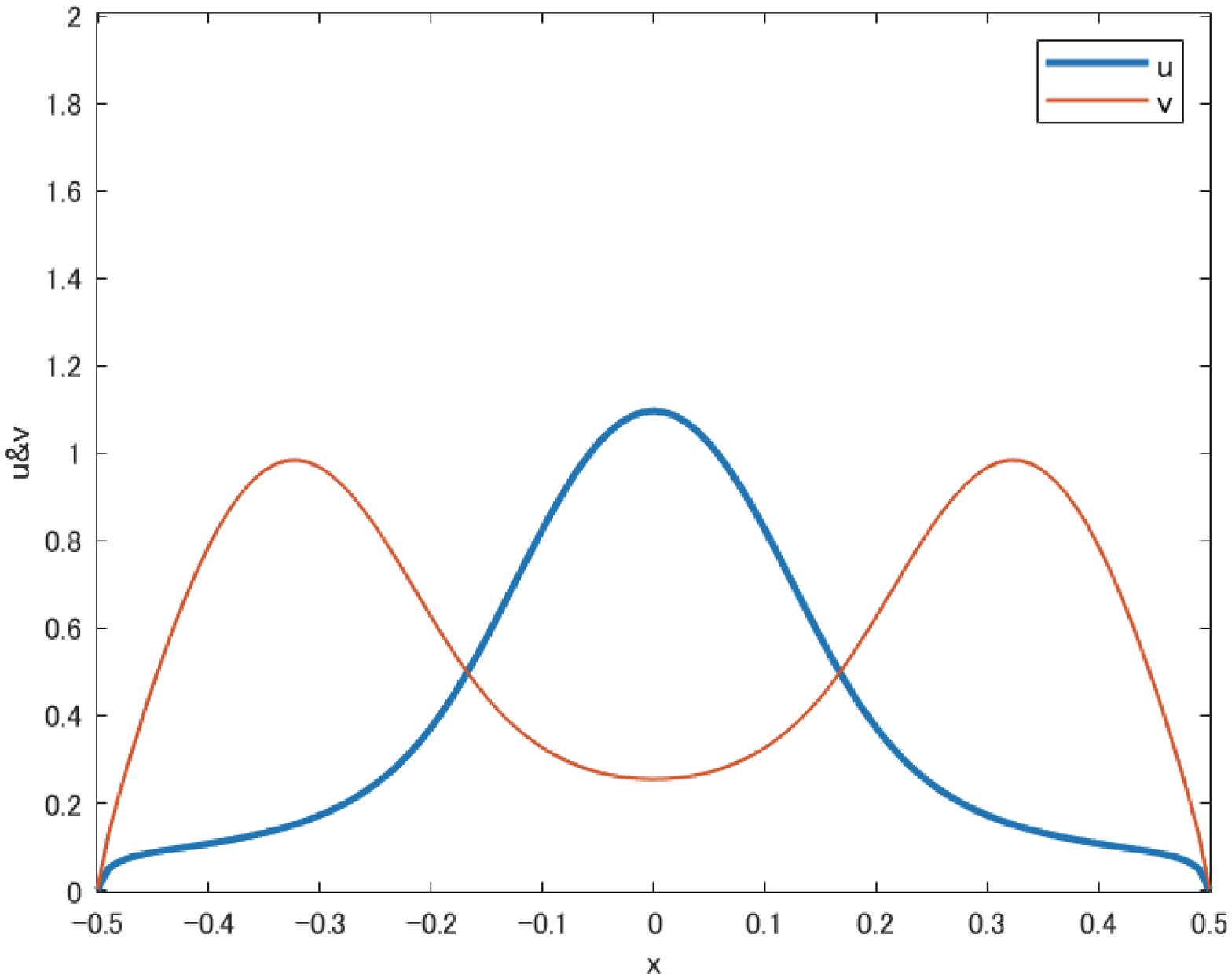}
\label{figc}}
\subfigure[Profile of a solution on purple lower branch at $\lambda=91.5836$.]{
\includegraphics*[scale=.3]{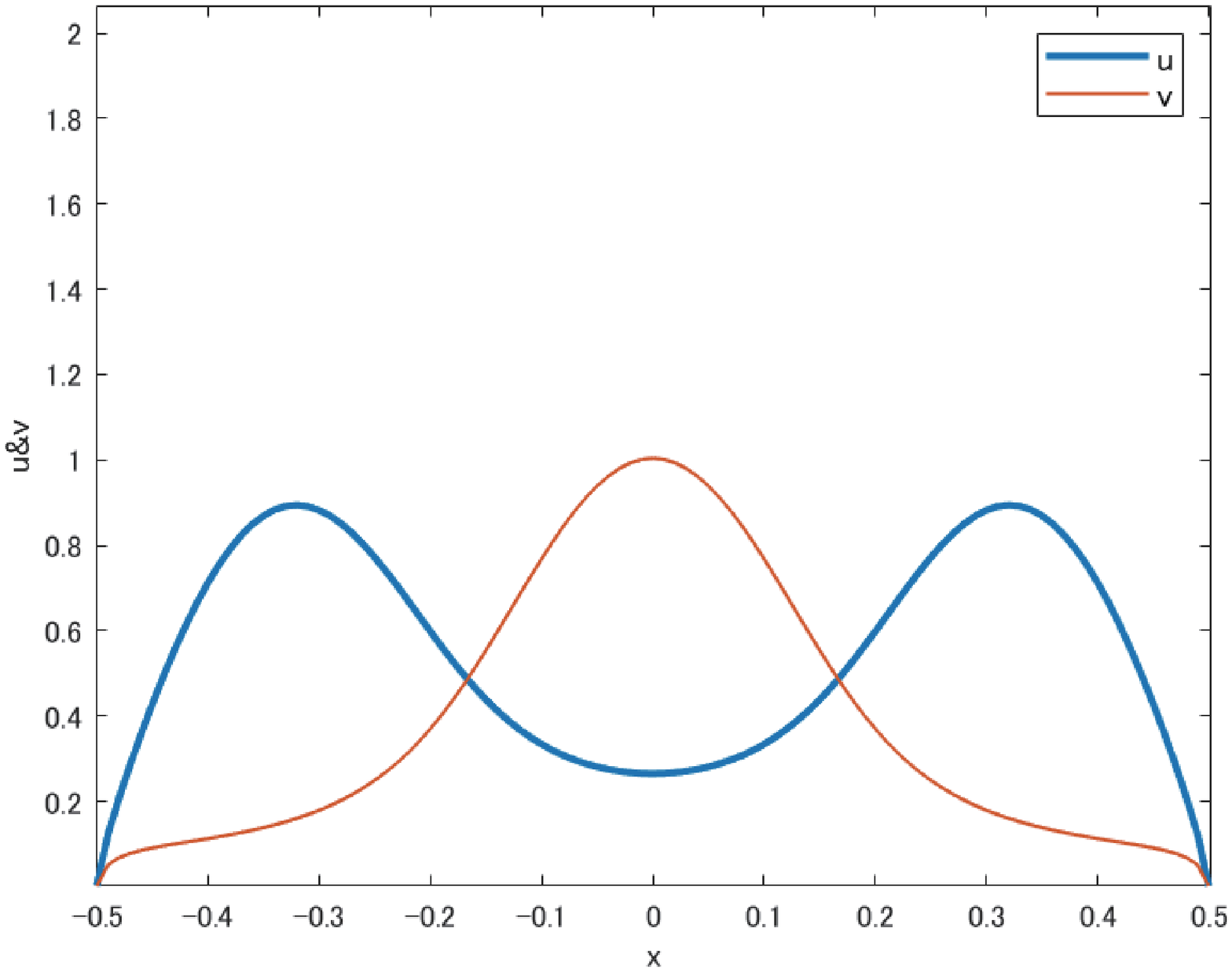}
\label{figd}
}
\caption{
Profiles of solutions on red and purple pitchfork bifurcation branches.}
\label{fig3}
\end{figure}

\begin{figure}
\centering
\subfigure[Bifurcation diagram of solutions of \eqref{SKT}.]{
\includegraphics*[scale=.3]{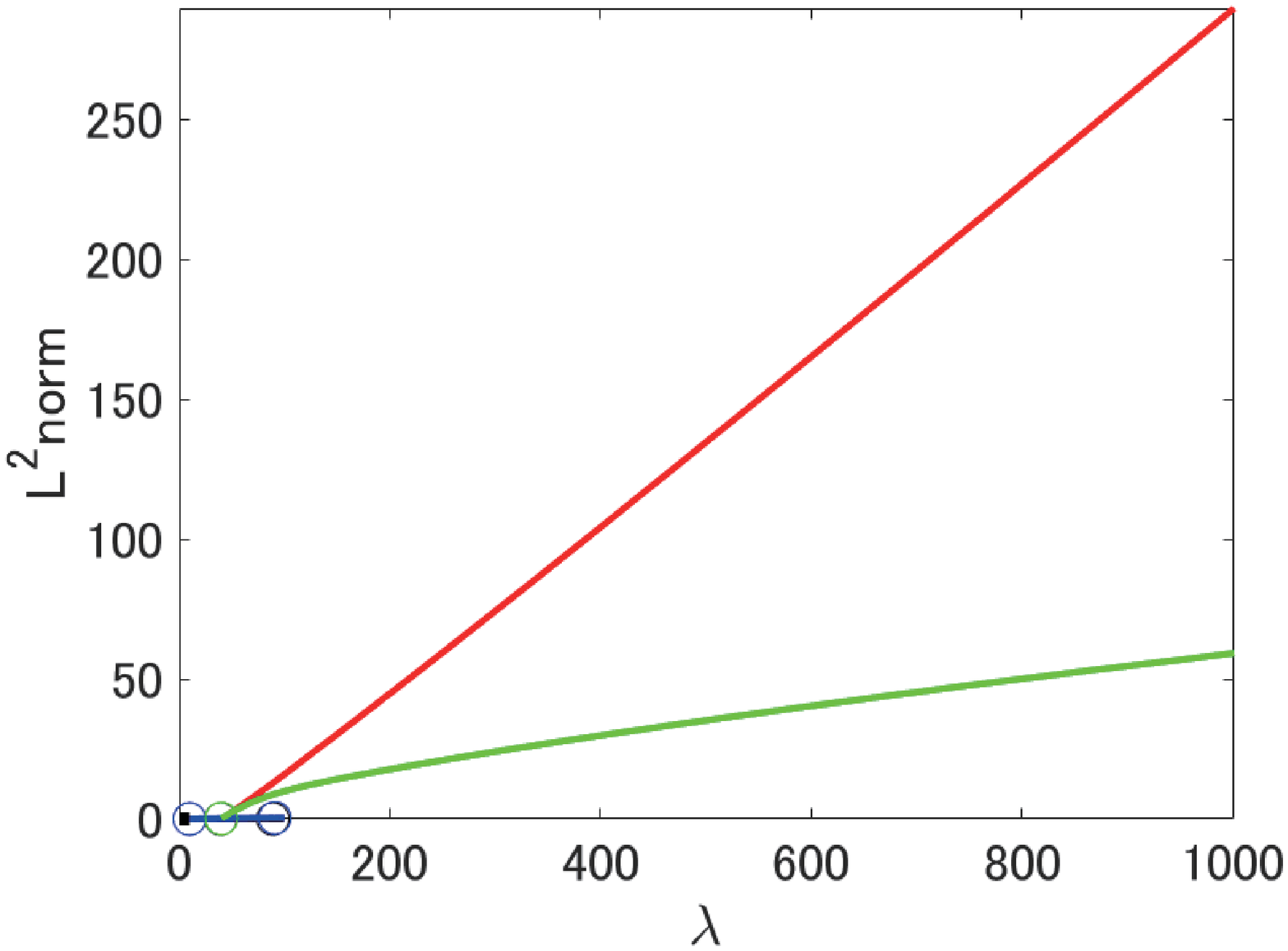}
\label{fig4a}}
\subfigure[Profile of a solution at $\lambda=1000$.]{
\includegraphics*[scale=.3]{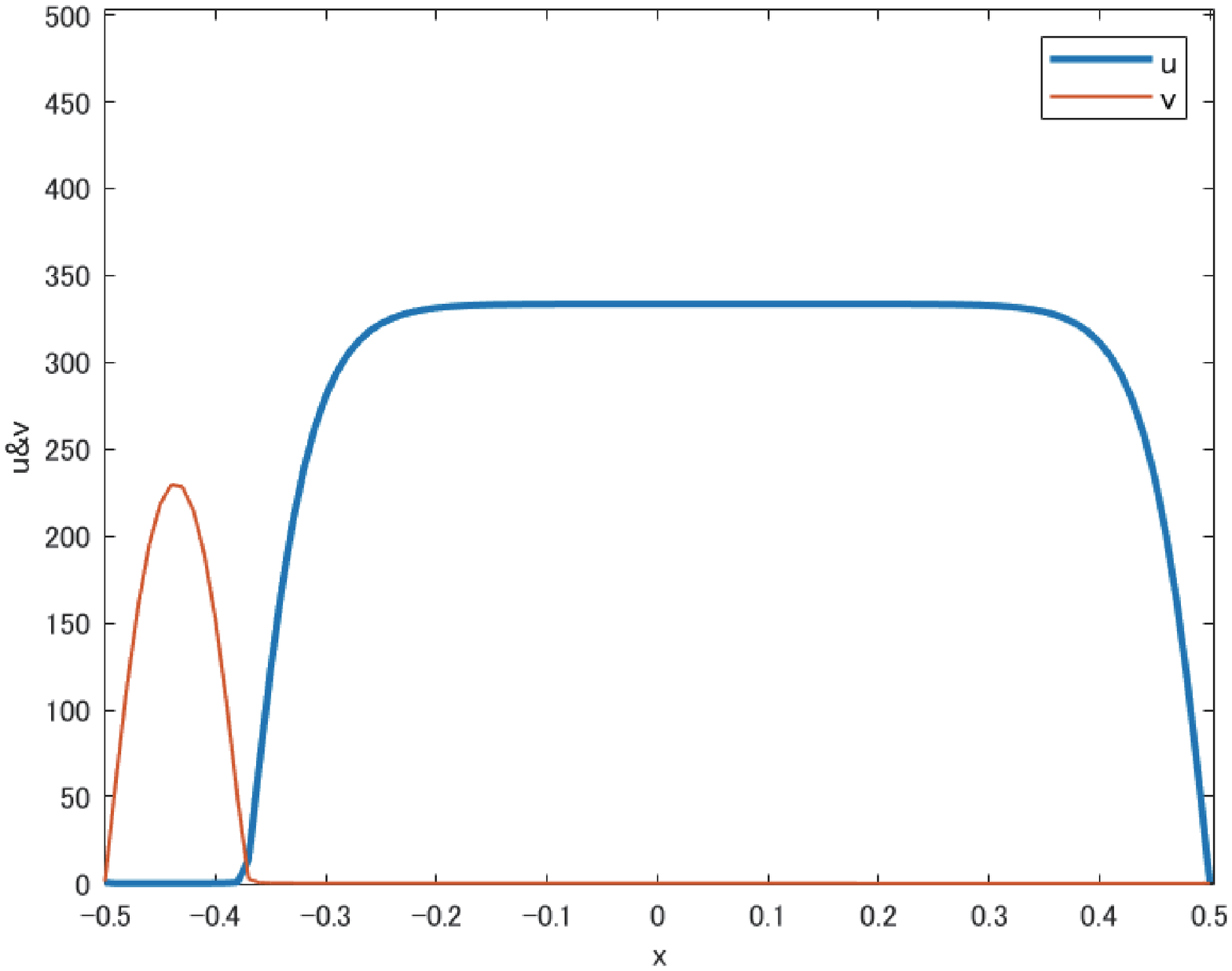}
\label{fig4b}
}
\caption{Bifurcation diagram and
profile of a solution on the branch.}
\label{fig4}
\end{figure}

In Figure 1, 
the purple curve indicates
the upper branch $\mathcal{S}_{3,20,\varLambda}^{+}$
and lower one $\mathcal{S}_{3,20,\varLambda}^{-}$ of 
the pitchfork bifurcation curve $\mathcal{S}_{3,20,\varLambda}$.
These upper and lower branches appear to be almost superimposed because 
they are not far from the bifurcation point.
The purple pitchfork bifurcation curve $\mathcal{S}_{3,20,\varLambda}$ 
bifurcates from
a solution on the blue curve at $\mu_{3,20}$.
By \eqref{biflim}, this bifurcation point 
$\mu_{3,\alpha}$ theoretically tends to 
$\lambda_{3}=(3\pi)^2\,(\,\fallingdotseq 88.8264)$
as $\alpha\to\infty$.
In Figure 3, (e) and (f) show the profiles of
solutions on the upper branch $\mathcal{S}_{3,20,\varLambda}^{+}$
and the lower one $\mathcal{S}_{3,20,\varLambda}^{-}$ of 
the purple pitchfork bifurcation curve, respectively.

In Figure 4(a),
the red (resp.\,green) branch represents the $L^{2}$ norm of 
$u$ (resp.\,$v$)
component of solutions on $\mathcal{S}_{2,20, \varLambda}^{\pm}$
bifurcating from a solution on $\mathcal{C}_{20, \varLambda}$
at $\lambda =\mu_{2,20}$ up to $\lambda=1000$. 
The gap of $L^{2}$ norms of $u$ and $v$ components on 
$\mathcal{S}_{2,20, \varLambda}^{\pm}$
can be observed when $\lambda $ is far from the secondary bifurcation point
at $\lambda =\mu_{2,20}$. 
Figure 4(b) shows the profile of the solution on the upper branch for 
$\lambda=1000$, where $u$ and $v$ are almost completely separated 
from each other.

Apparently, Figures 1-4 numerically support the bifurcation structure 
pure mathematically proved as Theorem \ref{g2ndbifthm} such that 
the branch $\mathcal{C}_{\alpha, \Lambda}$ of small coexistence 
bifurcates from the trivial solution, and moreover,
the pitchfork bifurcation branches $\mathcal{S}_{j, \alpha, \varLambda}^{\pm}$
of complete segregation bifurcates from solutions on 
$\widehat{\mathcal{C}}_{\alpha}$.

\subsection{Transformation to adopt diffusion coefficients as 
the bifurcation parameter}
In this subsection.
we apply a suitable transformation to the solution 
of \eqref{SKT} so that the random diffusion coefficient $d$ can 
be employed as a bifurcation parameter. Then, 
the range of $d$ can be restricted to a bounded interval.
From the viewpoint of 
numerical tracking of the bifurcation branches
the visual bifurcation diagram can be more observable and reasonable.

For any solution $(\widehat{u}, \widehat{v})$ of \eqref{SKT},
we employ the transformation
\begin{equation}\label{transform}
(u, v)=\dfrac{1}{\lambda}(\widehat{u}, \widehat{v}),\qquad
d=\dfrac{1}{\lambda }
\end{equation}
to verify that $(u,v)$ is a solution to 
\begin{equation}\label{SKTd}
\begin{cases}
\Delta [\,(d+\alpha v)u\,]+
u(m(x)-b_{1}u-c_{1}v)=0
\ \ &\mbox{in}\ \Omega,\\
\Delta [\,(d+\alpha u)v\,]+
v(m(x)-b_{2}u-c_{2}v)=0
\ \ &\mbox{in}\ \Omega,\\
u=v=0\ \ &\mbox{on}\ \partial\Omega.
\end{cases}
\end{equation}
Then,
Theorems \ref{Cathm} and \ref{g2ndbifthm} 
obtained as the bifurcation structure 
of solutions of \eqref{SKT} with parameter $\lambda$ 
can be rewritten as the following bifurcation structure 
of solutions of \eqref{SKTd} with parameter $d$ 
via the transformation \eqref{transform}.

\begin{cor}\label{cornum}
Define $d_{j}:=1/\lambda_{j}$.
For any small $\varepsilon >0$,
there exists a large $\overline{\alpha}=\overline{\alpha}
(\varepsilon )$ such that, 
if $\alpha>\overline{\alpha}$,
then there exists a bifurcation curve
$$
\widehat{C}_{\alpha}=
\{\,(d,u_{0,\alpha}(\,\cdot\,,d), v_{0,\alpha}(\,\cdot\,,d))
\in [\varepsilon, d_{1})\times\boldsymbol{X}\,\},
$$
where $[\varepsilon, d_{1})\times
(\overline{\alpha}, \infty )\ni (d, \alpha )\mapsto
(u_{0,\alpha}(\,\cdot\,,d), v_{0,\alpha}(\,\cdot\,,d))\in\boldsymbol{X}$
is of class $C^{1}$ such that
$$
\lim_{d\nearrow d_{1}}
(u_{0,\alpha}(\,\cdot\,,d), v_{0,\alpha}(\,\cdot\,,d))=(0,0)
\quad\mbox{in}\ \boldsymbol{X}
$$ 
and
$$
\lim_{\alpha\to\infty}
\alpha (u_{0,\alpha}(\,\cdot\,,d), v_{0,\alpha}(\,\cdot\,,d))
=(\widehat{U}(\,\cdot\,,d), \widehat{U}(\,\cdot\,,d))$$
with some positive function $\widehat{U}(\,\cdot\,,d)\in X$,
which is continuously differentiable for
$d\in (0,d_{1})$ with
$\lim_{d\nearrow d_{1}}\widehat{U}(\,\cdot\,,d)=0$
in $X$.
Furthermore, $\widehat{C}_{\alpha}$ can 
be extended as a connected subset of positive solutions
of \eqref{SKTd} to
the range $0<d<\varepsilon$.

In particular, if $\Omega=(-\ell, \ell)$ and 
$m(x)=m$
(constant)
$>0$
for all $x\in\overline{\Omega}$,
then, for each $j\in\mathbb{N}$ with $j\ge 2$,
there exists a pitchfork bifurcation curve
$\widehat{\mathcal{S}}_{j, \alpha }$
whose upper branch $\widehat{\mathcal{S}}_{j, \alpha}^{+}$
and lower branch $\widehat{\mathcal{S}}_{j, \alpha}^{-}$ are
parameterized as
$$
\widehat{\mathcal{S}}_{j, \alpha}^{+}=
\{\,(d, u,v)=(\delta_{j, \alpha}(s), 
u_{j, \alpha}(\,\cdot\,,s), 
v_{j,\alpha}(\,\cdot\,,s))
\in\mathbb{R}_{+}\times\boldsymbol{X}
\,:\,s\in (0,T_{j, \alpha, \varepsilon}^{+}\,]\,\}
$$
and 
$$
\widehat{\mathcal{S}}_{j, \alpha}^{-}=
\{\,(d, u,v)=(\delta_{j, \alpha}(s), 
u_{j, \alpha}(\,\cdot\,,s), 
v_{j, \alpha}(\,\cdot\,,s))
\in\mathbb{R}_{+}\times\boldsymbol{X}
\,:\,s\in [\,-T_{j, \alpha, \varepsilon}^{-},0)\,]\,\},
$$
respectively, with some positive numbers
$T_{j, \alpha, \varepsilon}^{\pm}$,
where
$$
[\,-T_{j,\alpha, \varepsilon}^{-}, 
T_{j, \alpha, \varepsilon}^{+}\,]
\ni s\mapsto 
(\delta_{j,\alpha}(s), 
u_{j,\alpha}(\,\cdot\,,s), 
v_{j,\alpha}(\,\cdot\,,s))
\in \mathbb{R}_{+}\times\boldsymbol{X}
$$
is of class $C^{1}$ satisfying 
$$
(\delta_{j,\alpha}(0), 
u_{j,\alpha}(\,\cdot\,,0), 
v_{j,\alpha}(\,\cdot\,,0))
=(\delta_{j,\alpha }^{*}, u_{0,\alpha}(\,\cdot\,, \delta_{j,\alpha}^{*}), 
v_{0,\alpha}(\,\cdot\,, \delta_{j, \alpha}^{*}))
\in\widehat{\mathcal{C}}_{\alpha}$$
with some $\delta_{j,\alpha}^{*}\in [\varepsilon, \lambda_{1})$
and
$\delta_{j,\alpha}(-T^{-}_{j,\alpha,\varepsilon})=
\delta_{j,\alpha}(T^{+}_{j,\alpha,\varepsilon})=\varepsilon$
and
$$
\lim_{\alpha\to\infty}
(\delta_{j, \alpha}(s), 
u_{j, \alpha}(\,\cdot\,,s), 
v_{j, \alpha}(\,\cdot\,,s))=
\begin{cases}
(d, w_{j,0}^{+}(\,\cdot\,,d)_{+},
w_{j,0}^{+}(\,\cdot\,,d)_{-})\quad
&\mbox{if}\ s>0,\\
(d, w_{j,0}^{-}(\,\cdot\,,d)_{+},
w_{j,0}^{-}(\,\cdot\,,d)_{-})\quad
&\mbox{if}\ s<0
\end{cases}
$$
for some $d\in [\varepsilon, d_{1})$ and
sign-changing functions 
$w_{j,0}^{+}(x, d)$ and
$w_{j,0}^{-}(x, d)$.
Here, 
$
\delta_{j,\alpha}^{*}$
is continuous with respect to $\alpha\in (\overline{\alpha}, \infty)$ and satisfies
\begin{equation}\label{jbif}
\lim_{\alpha\to\infty}\delta_{j,\alpha }^{*}=d_{j}.
\end{equation}
Furthermore, if $j$ is even, then
\begin{equation}\label{evensym2}
(\delta_{j, \alpha}(s),
u_{j, \alpha}(x,s),
v_{j, \alpha}(x,s))=
(\delta_{j,\alpha}(-s),
u_{j, \alpha}(-x,-s),
v_{j, \alpha}(-x,-s))
\end{equation}
for any $x\in [-\ell, \ell]$ and 
$s\in [-T^{-}_{j,\alpha,\varepsilon}, T^{+}_{j,\alpha,\varepsilon}]$
with $T^{-}_{j,\alpha,\varepsilon}=T^{+}_{j,\alpha,\varepsilon}$.
\end{cor}

\begin{figure}
\begin{center}
{\includegraphics*[scale=.3]{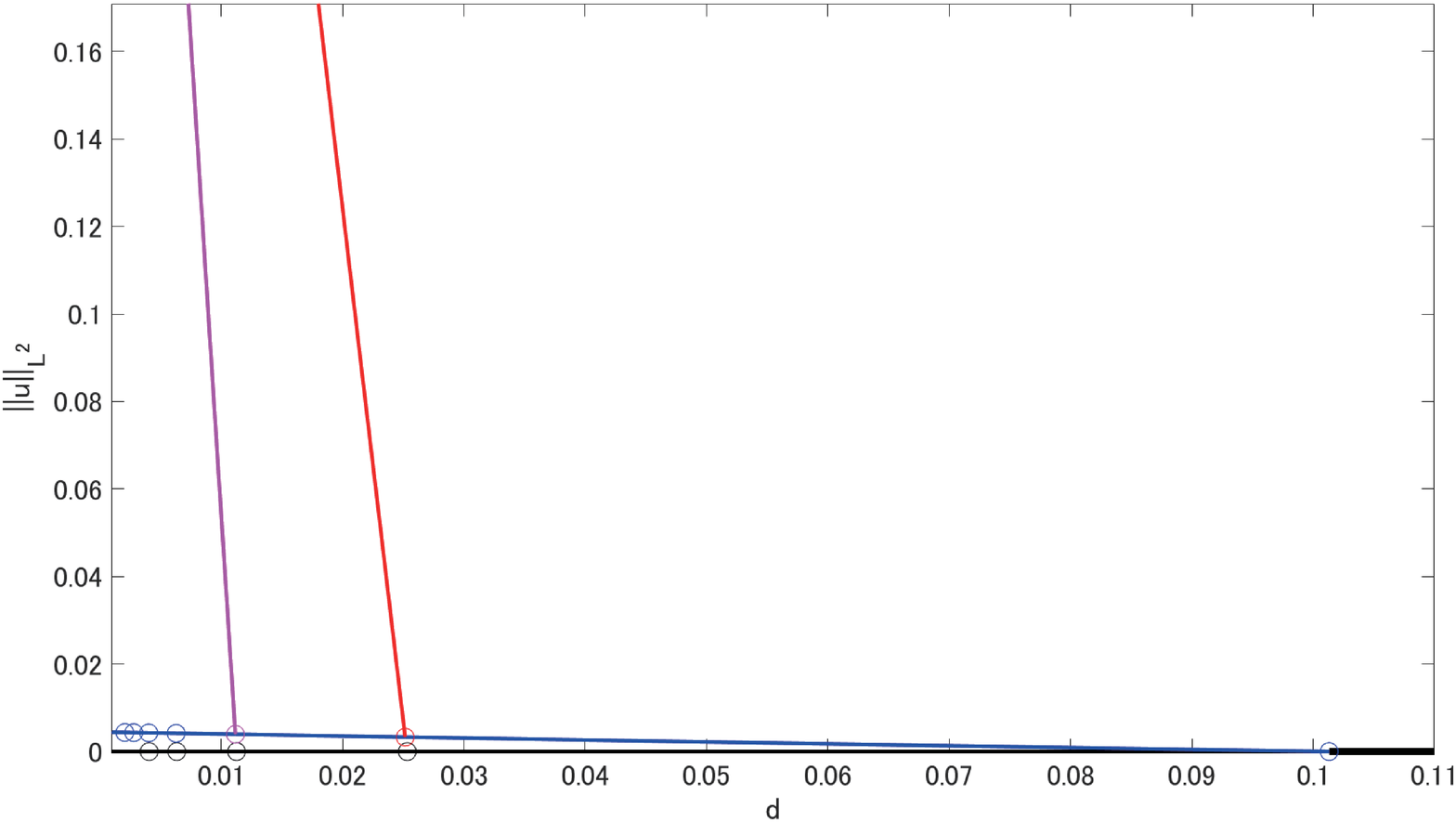}}\\
\caption{Bifurcation diagram of solutions of \eqref{SKTd}.}
\end{center}\end{figure}

\begin{figure}
\begin{center}
{\includegraphics*[scale=.4]{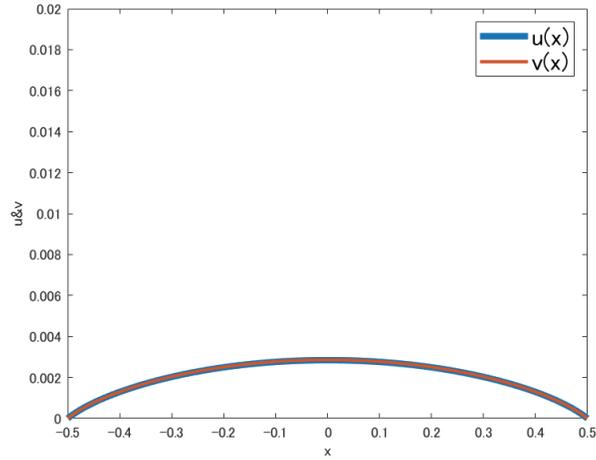}}\\
\caption{Profile of a solution on blue curve at $d=0.05$.}
\end{center}\end{figure}

\begin{figure}
\centering
\subfigure[Profile of a solution on red upper branch at $d=0.025$.]{
\includegraphics*[scale=.3]{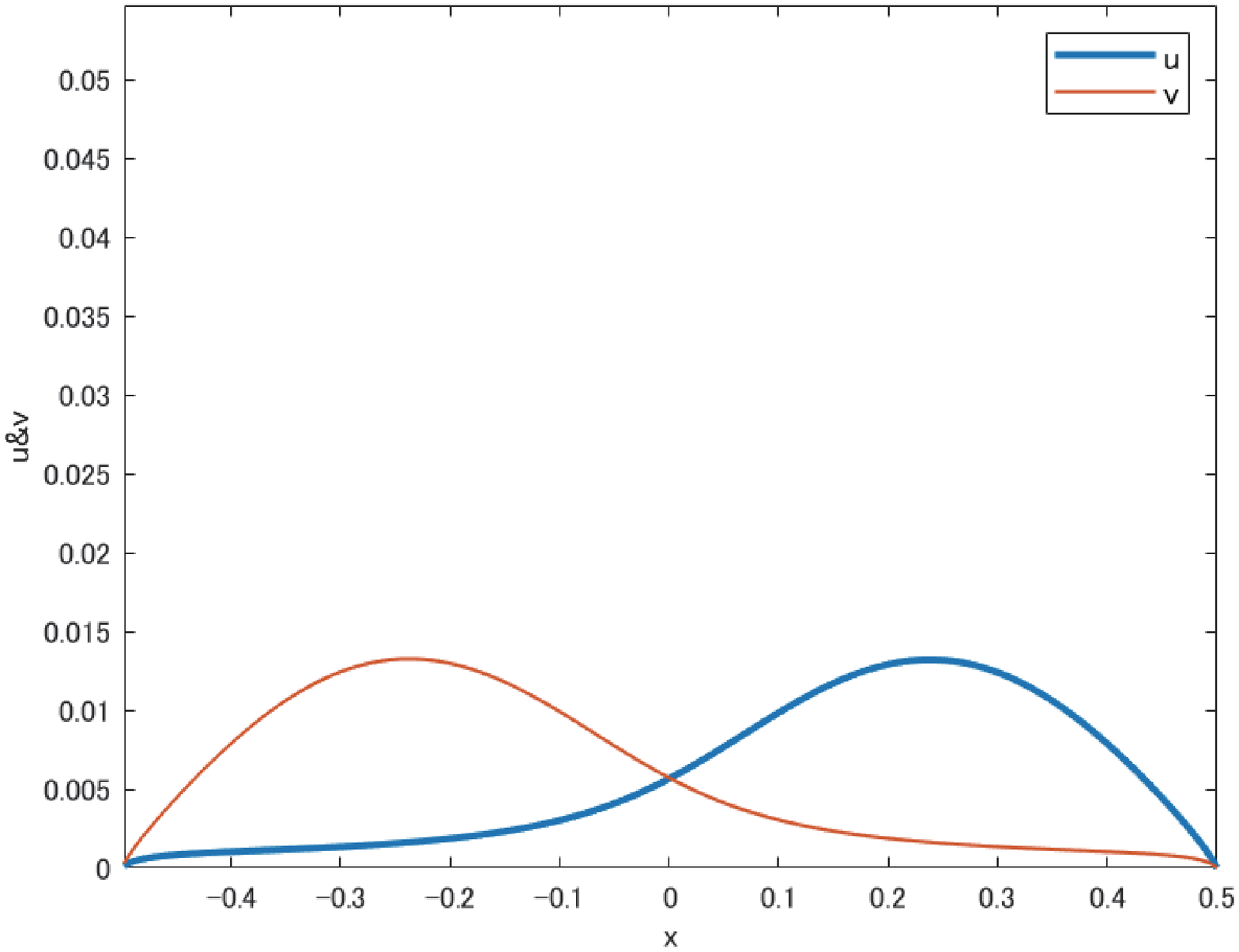}
\label{figa}}
\subfigure[Profile of a solution on red lower branch at $d=0.025$.]{
\includegraphics*[scale=.3]{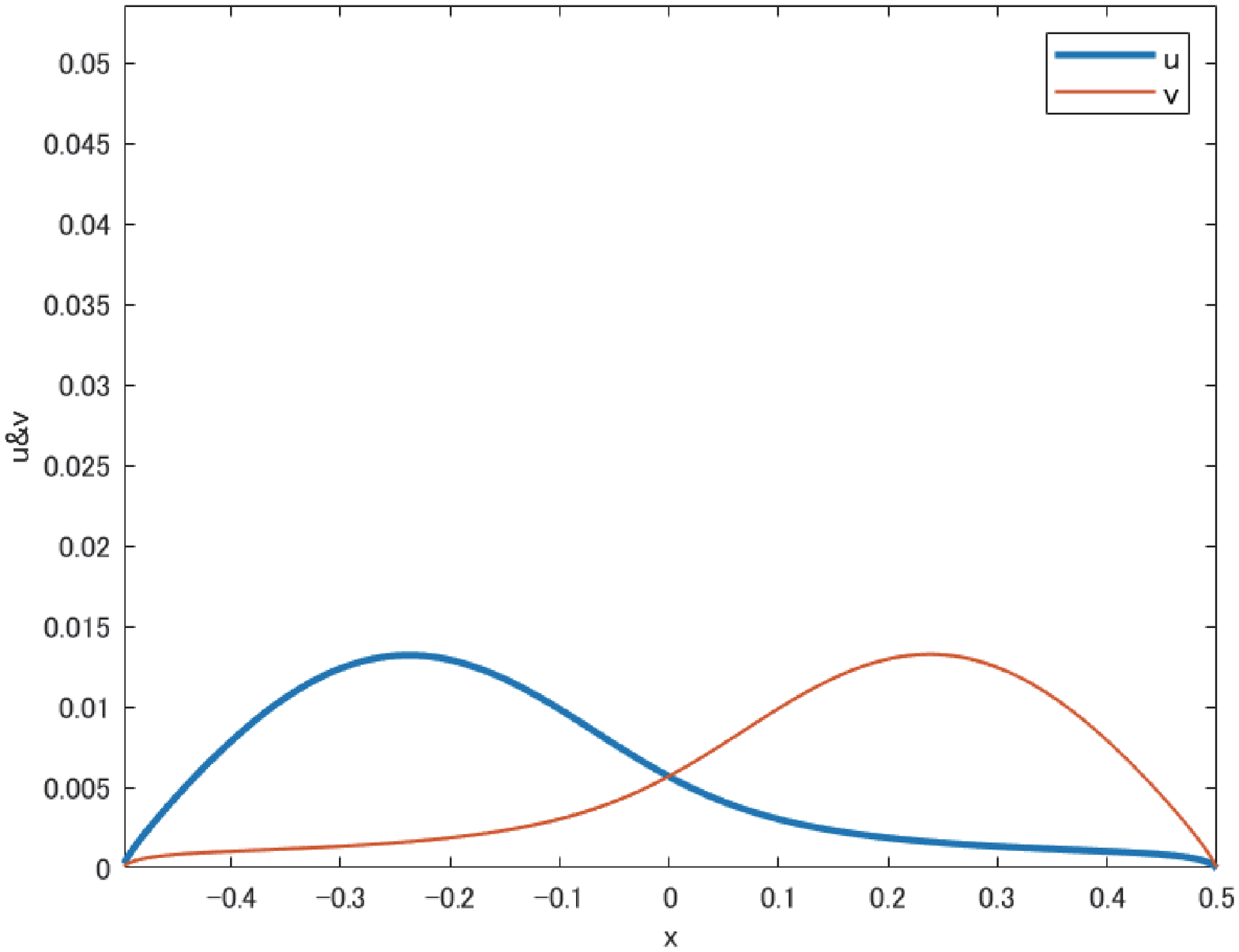}
\label{figb}
}
\centering
\subfigure[Profile of a solution on red upper branch at $d=0.0209$.]{
\includegraphics*[scale=.3]{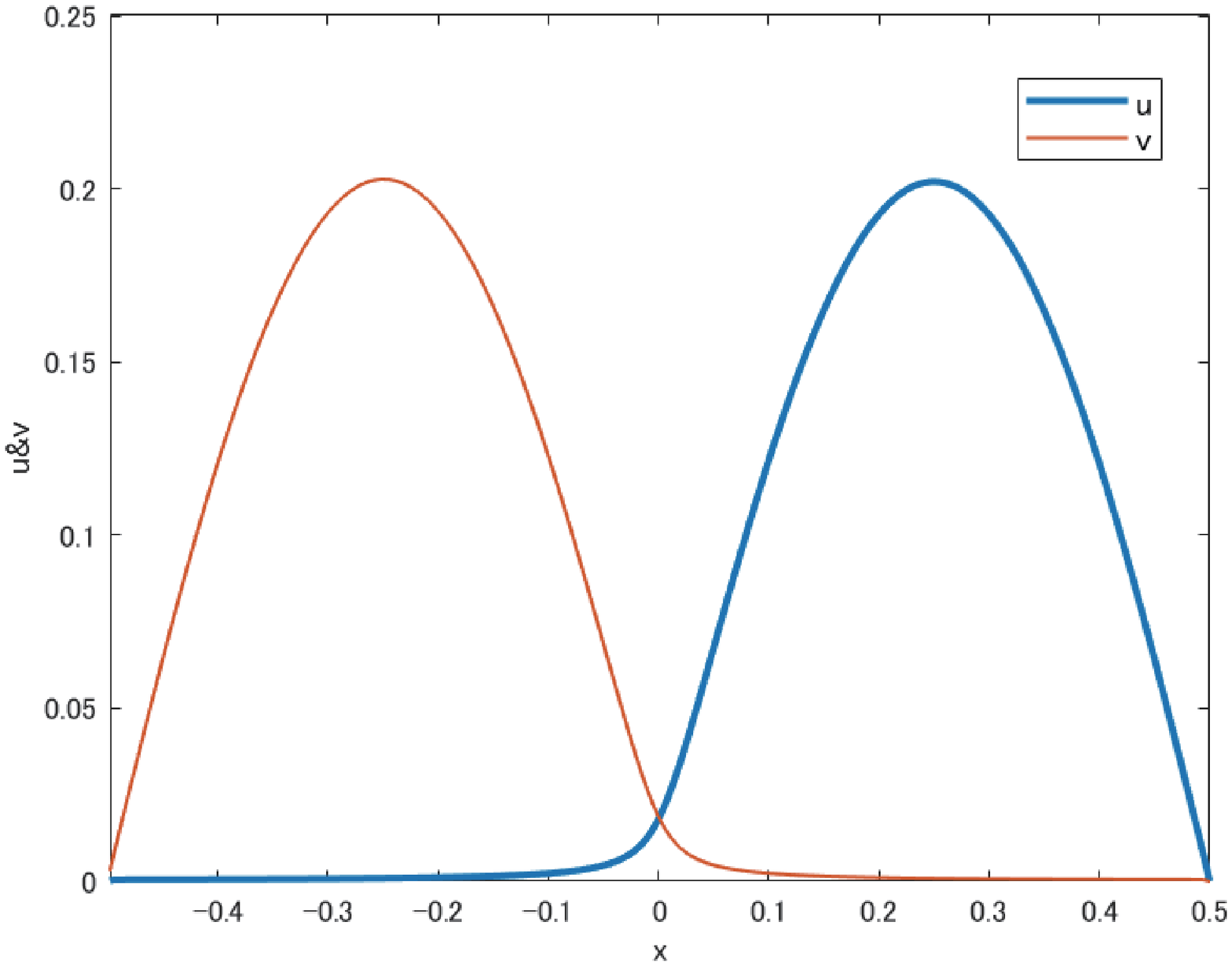}
\label{figc}}
\subfigure[Profile of a solution on red lower branch at $d=0.0209$.]{
\includegraphics*[scale=.3]{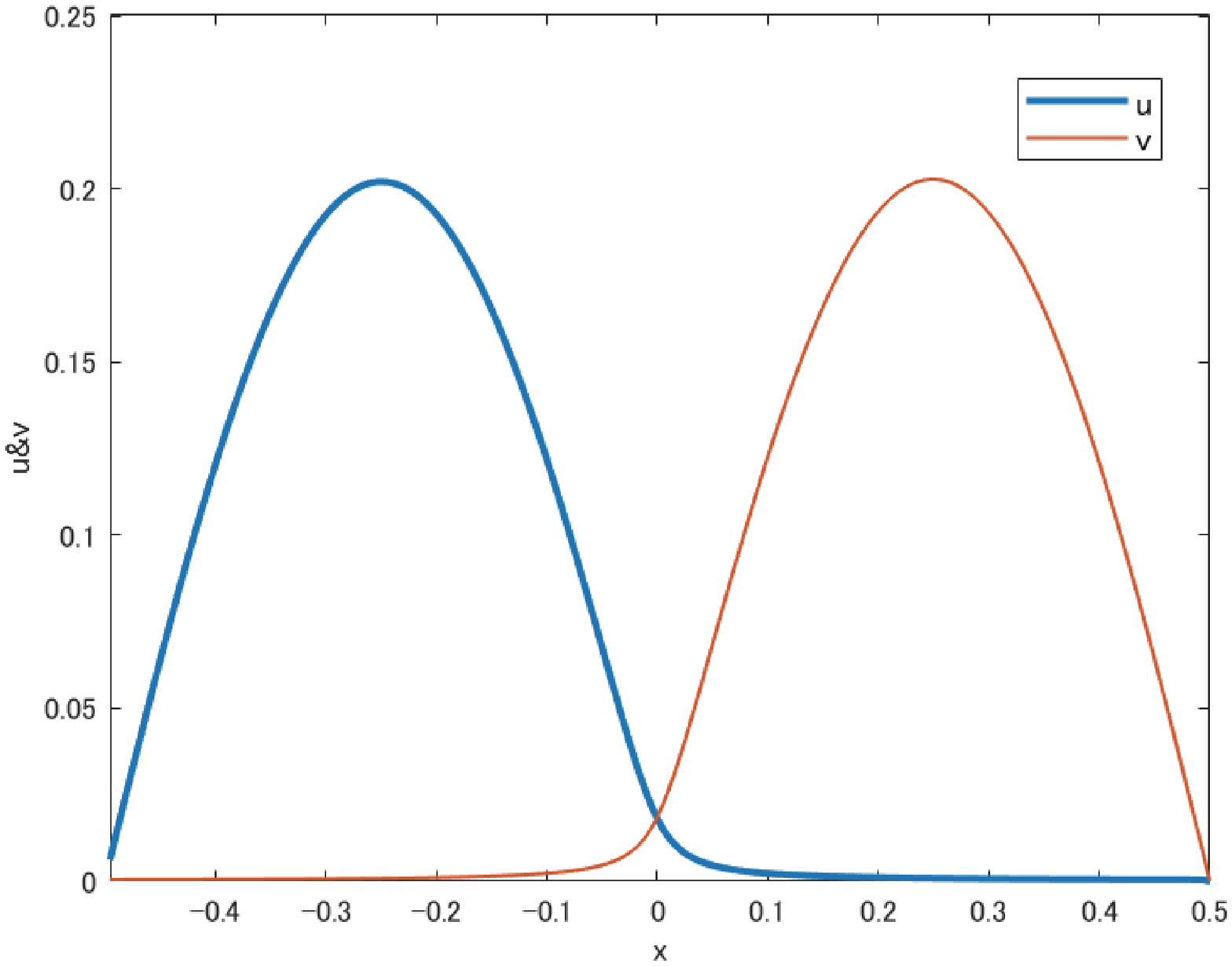}
\label{figd}
}
\centering
\subfigure[Profile of a solution on purple upper branch at $d=0.0074$.]{
\includegraphics*[scale=.3]{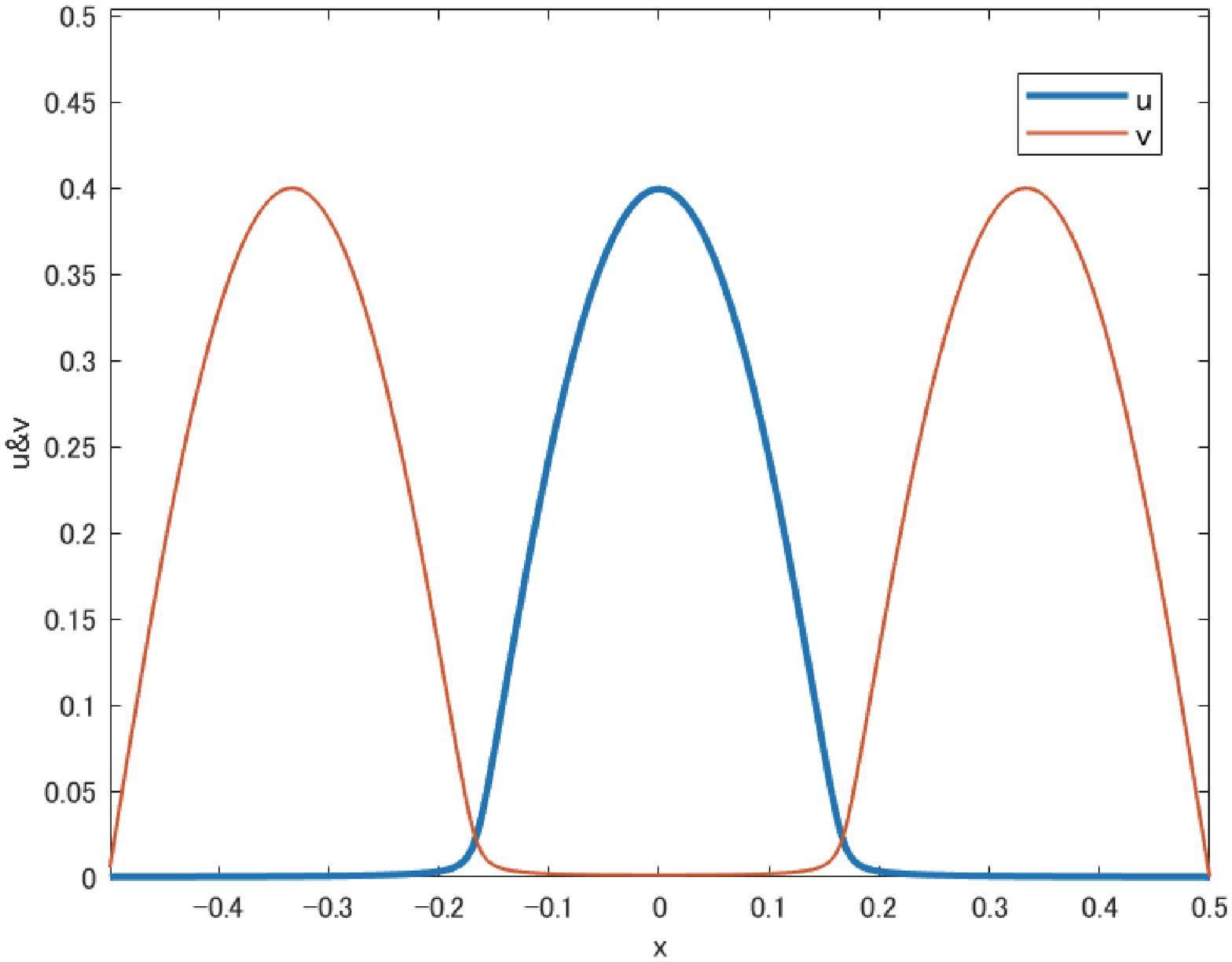}
\label{figc}}
\subfigure[Profile of a solution on purple lower branch at $d=0.0074$.]{
\includegraphics*[scale=.3]{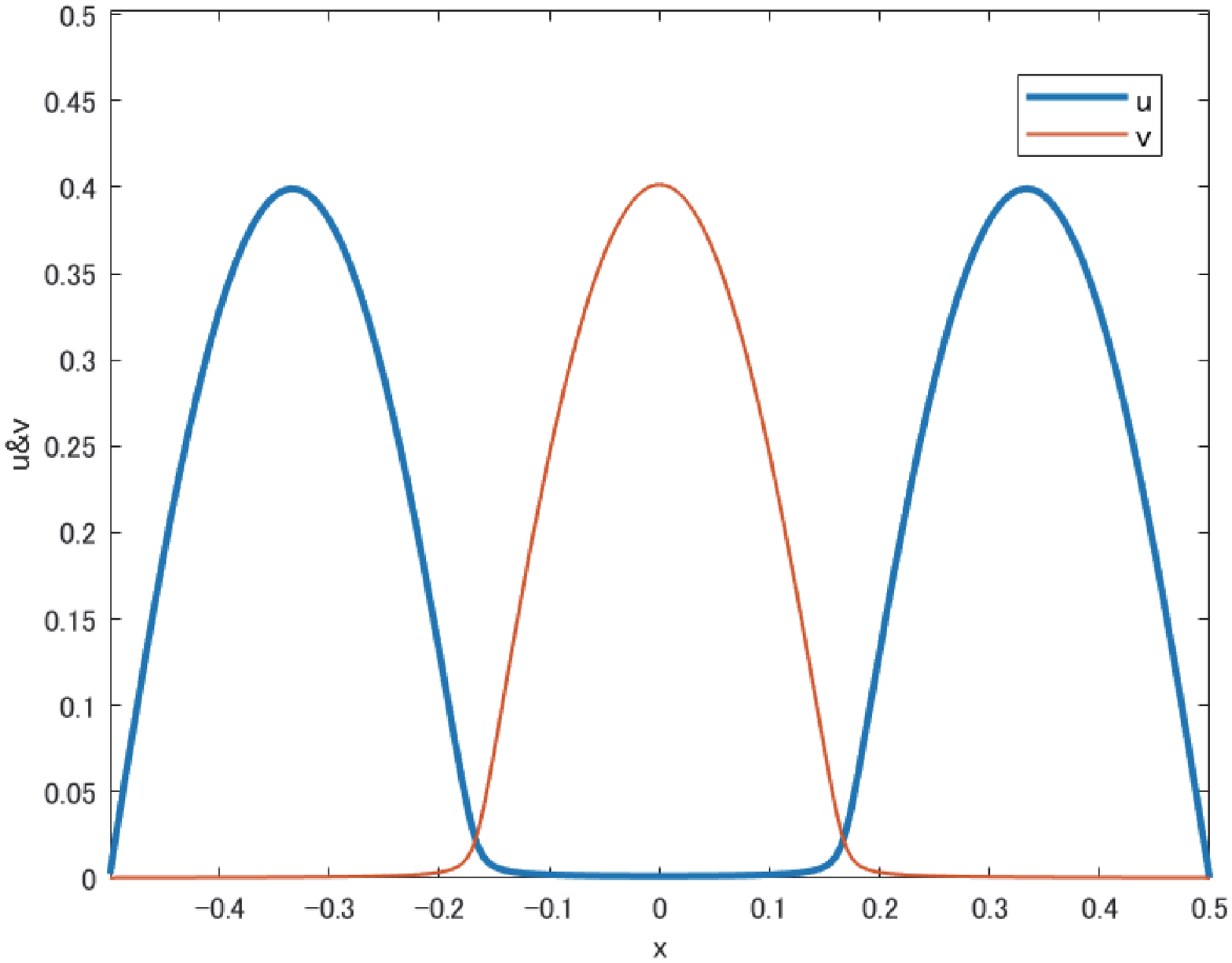}
\label{figd}
}
\caption{
Profiles of solutions on red and purple pitchfork bifurcation branches.}
\label{fig3}
\end{figure}

\subsection{Numerical bifurcation diagram with parameter $d$}
This subsection exhibits
the numerical bifurcation diagram with parameter $d$. 
For \eqref{SKTd}, our setting of parameters for the numerical simulation 
as follows:
$$
\Omega=(-0.5,0.5),
\quad \alpha=20,
\quad m(x)=1, 
\quad b_{1}=1, 
\quad b_{2}=2, 
\quad c_{1}=1, 
\quad c_{2}=1.
$$
Figure 5 shows the numerical bifurcation diagram in which
the horizontal axis represents the bifurcation parameter $d$, 
and the vertical axis represents the $L^{2}$ norm of the 
$u$ component of positive solutions $(u,v)$ of \eqref{SKTd}.
The blue curve is corresponding to $\widehat{\mathcal{C}}_{20}$ 
which bifurcates from the trivial solution at $d=d_{1}$.
The theoretical value of the bifurcation point is calculated as
$d_1=1/\pi^2\,(\,\fallingdotseq 0.10132)$, 
which seems to be very near the numerical bifurcation point 
of the blue curve in Figure 5.
Furthermore, the blue curve remains small and extend to
$d=0$,
and corresponds to the branch of small coexistence.
The profile of a positive 
solution $(u,v)$ on the blue curve at $d=0.05$
is shown in Figure 6,
where
one can find that $u$ and $v$ are small and
very close to each other.

In Figure 5, 
the red curve corresponds to
the upper and lower branches $\widehat{\mathcal{S}}_{2,20}^{\pm}$
of 
the pitchfork bifurcation curve $\widehat{\mathcal{S}}_{2,20}$.
In view of \eqref{evensym2}, 
we are convinced that these upper and lower branches are superimposed.
This red pitchfork bifurcation curve bifurcates from
a solution on the blue curve $\widehat{\mathcal{C}}_{2,20}$ 
at $d=\delta_{2,20}^{*}$.
By \eqref{jbif}, the secondary bifurcation point 
$\delta_{2,\alpha}^{*}$ theoretically converges to 
$d_{2}=1/(2\pi)^2\,(\,\fallingdotseq 0.02533)$
as $\alpha\to\infty$.
It can be observed in Figure 5 that
the numerical secondary bifurcation point 
is near the theoretical limit.
In Figure 7, (a) and (b) show the profiles of
solutions on the upper branch $\widehat{\mathcal{S}}_{2,20}^{+}$
and the lower one $\widehat{\mathcal{S}}_{2,20}^{-}$
on the red pitchfork bifurcation curve $\widehat{\mathcal{S}}_{2,20}$ at
$d=0.025$. It can be observed that $u$ and $v$ are somehow spatially segregate.
In Figure 7, (c) and (d) show the profiles of
solutions on the upper branch and the lower one
on the red pitchfork bifurcation curve at
$d=0.0209$, where $u$ and $v$ considerably segregate each other.

In Figure 5, 
the purple curve corresponds to
the upper branch $\widehat{\mathcal{S}}_{3,20}^{+}$
and lower one $\widehat{\mathcal{S}}_{3,20}^{-}$ of 
the pitchfork bifurcation curve $\widehat{\mathcal{S}}_{3,20}$,
which bifurcates from
a solution on the blue curve at $d=\delta_{3,20}$.
By \eqref{jbif}, this bifurcation point 
$\delta_{3,\alpha}$ theoretically tends to 
$d_{2}=1/(3\pi)^2\,(\,\fallingdotseq 0.01125)$
as $\alpha\to\infty$.
The numerical bifurcation point $\delta_{3,20}$ also seems to
be very close to the limit.
In Figure 7, (e) and (f) shows the profiles of
solutions on the upper branch $\widehat{\mathcal{S}}_{3,20}^{+}$
and the lower one $\widehat{\mathcal{S}}_{3,20}^{-}$ of 
the purple pitchfork bifurcation curve $\widehat{\mathcal{S}}_{3,20}$, 
respectively.
It can be seen that, on the upper (resp.\,lower) branch,
the territory of $u$ (resp.\,$v$) is formed in 
the left and right two-thirds of 
$\Omega$, while the territory $v$ (resp. $u$) 
is formed in the middle third of $\Omega$.


\end{document}